\documentclass[twoside,a4paper,reqno,11pt]{amsart} 
\usepackage[top=28mm,right=28mm,bottom=28mm,left=28mm]{geometry}
\usepackage{microtype}
\usepackage{amsfonts, amsmath, amssymb, amsbsy, mathrsfs, bm, latexsym, stmaryrd, hyperref, array, enumitem, textcomp, url}
\usepackage[table]{xcolor}

\renewcommand{\a}{\alpha}
\renewcommand{\b}{\beta}
\newcommand{\g}{\gamma}
\renewcommand{\d}{\delta}
\newcommand{\e}{\epsilon}

\renewcommand{\r}{\rho} 
\newcommand{\s}{\sigma}
\renewcommand{\O}{\Omega}
\renewcommand{\t}{\tau}
\newcommand{\p}{\varphi}
\newcommand{\<}{\langle}
\renewcommand{\>}{\rangle}
\newcommand{\la}{\langle}
\newcommand{\ra}{\rangle}
\newcommand{\leqs}{\leqslant}
\newcommand{\geqs}{\geqslant}

\newcommand{\Aut}{\operatorname{Aut}}
\newcommand{\Out}{\operatorname{Out}}
\newcommand{\Inndiag}{\operatorname{Inndiag}}

\newcommand{\vs}{\vspace{3mm}}

\makeatletter
\newcommand{\imod}[1]{\allowbreak\mkern4mu({\operator@font mod}\,\,#1)}
\makeatother

\renewcommand{\mod}[1]{\imod{#1}}
\renewcommand{\leq}{\leqs}
\renewcommand{\geq}{\geqs}
\newcommand{\leqn}{\trianglelefteqslant}

\renewcommand{\:}{\colon}

\theoremstyle{plain}
\newtheorem{theorem}{Theorem} 
 
\newtheorem{corol}[theorem]{Corollary}
\newtheorem{thm}{Theorem}[section] 
\newtheorem{lem}[thm]{Lemma}
\newtheorem{prop}[thm]{Proposition} 

\newtheorem{cor}[thm]{Corollary} 
\newtheorem*{theorem*}{Theorem} 
\newtheorem*{conj*}{Conjecture}

\theoremstyle{definition}
\newtheorem{rem}[thm]{Remark}
\newtheorem{defn}[thm]{Definition}
\newtheorem{ex}[thm]{Example}
\newtheorem*{def-non}{Definition}

\begin{document}

\title[The spread of a finite group]{The spread of a finite group}

\author{Timothy C. Burness}
\address{T.C. Burness, School of Mathematics, University of Bristol, Bristol BS8 1UG, UK}
\email{t.burness@bristol.ac.uk}
 
\author{Robert M. Guralnick}
\address{R.M. Guralnick, Department of Mathematics, University of Southern California, Los Angeles, CA 90089-2532, USA}
\email{guralnic@usc.edu}

\author{Scott Harper}
\address{S. Harper, School of Mathematics, University of Bristol, Bristol BS8 1UG, UK, and Heilbronn Institute for Mathematical Research, Bristol, UK}
\email{scott.harper@bristol.ac.uk}

\date{\today} 

\begin{abstract}
A group $G$ is said to be $\frac{3}{2}$-generated if every nontrivial element belongs to a generating pair. It is easy to see that if $G$ has this property then every proper quotient of $G$ is cyclic. In this paper we prove that the converse is true for finite groups, which settles a conjecture of Breuer, Guralnick and Kantor from 2008. In fact, we prove a much stronger result, which solves a problem posed by Brenner and Wiegold in 1975. Namely, if $G$ is a finite group and every proper quotient of $G$ is cyclic, then for any pair of nontrivial elements $x_1,x_2 \in G$, there exists $y \in G$ such that $G = \la x_1, y \ra = \la x_2, y \ra$. In other words, $s(G) \geqs 2$, where $s(G)$ is the spread of $G$. Moreover, if $u(G)$ denotes the more restrictive uniform spread of $G$, then we can completely characterise the finite groups $G$ with $u(G) = 0$ and $u(G)=1$. To prove these results, we first establish a reduction to almost simple groups. For simple groups, the result was proved by Guralnick and Kantor in 2000 using probabilistic methods and since then the almost simple groups have been the subject of several papers. By combining our reduction theorem and this earlier work, it remains to handle the groups with  socle an exceptional group of Lie type and this is the case we treat in this paper.
\end{abstract}

\maketitle

\section{Introduction} \label{s:intro}

In this paper we study the spread and uniform spread of finite groups. These natural invariants encode interesting generation properties and they have been the subject of numerous papers spanning a period of more than 50 years. We begin with their definitions.

\begin{def-non}
Let $G$ be a group. 
\begin{itemize}\addtolength{\itemsep}{0.2\baselineskip}
\item[{\rm (i)}]   The \emph{spread} of $G$, denoted $s(G)$, is the largest integer $k$ such that for any nontrivial elements $x_1, \dots, x_k$ in $G$, there exists $y \in G$ with $G = \< x_i, y \>$ for all $i$.
\item[{\rm (ii)}]  The \emph{uniform spread} of $G$, denoted $u(G)$, is the largest integer $k$ such that there is a conjugacy class $C$ of $G$ with the property that for any nontrivial elements $x_1, \dots, x_k$, there exists $y \in C$ with $G = \< x_i, y \>$ for all $i$. Here we say $C$ \emph{witnesses} $u(G) \geq k$.
\item[{\rm (iii)}] If no such largest integer exists in (i) or (ii), then we write $s(G) = \infty$ or $u(G) = \infty$, respectively.
\end{itemize}
\end{def-non}

Let us observe that for any group $G$ we have $s(G) \geq u(G) \geq 0$, and if $G$ is cyclic, then $s(G) = u(G) = \infty$. A group $G$ is \emph{$\frac{3}{2}$-generated} if every nontrivial element belongs to a generating pair, which is equivalent to the condition $s(G) \geq 1$. Therefore, we can view the concepts of spread and uniform spread as natural extensions of the $\frac{3}{2}$-generation property.

The notion of spread was first introduced in the 1970s by Brenner and Wiegold in \cite{BW}, where numerous results on the spread of  soluble groups and certain families of simple groups (such as alternating groups and linear groups of the form ${\rm L}_{2}(q)$) are established. However, it turns out that the spread of finite groups has been studied since as early as the 1930s. For instance, a 1939 paper of Piccard \cite{Piccard} proves that the symmetric group $G = {\rm Sym}_{n}$ has positive spread for all $n \geqs 5$ and later work of Binder \cite{Binder68,Binder70I} in the 1960s extended this to $s(G) \geqs 2$. The more restrictive definition of uniform spread was formally introduced much more recently by Breuer, Guralnick and Kantor \cite{BGK}, although one finds work of Binder \cite{Binder70II} from 1970 on the uniform spread of symmetric groups.

As a consequence of the Classification of Finite Simple Groups, we know that every nonabelian finite simple group can be generated by two elements. This is a routine exercise for the alternating groups and a theorem of Steinberg \cite{Steinberg} for groups of Lie type. The property was verified for the sporadic groups by Aschbacher and Guralnick in \cite{AG}. In view of this fundamental result, it is natural to study the spread and uniform spread of finite simple groups and there is an extensive literature on this topic.  

The first main result is due to Guralnick and Kantor \cite{GK}, who proved that $u(G) \geqs 1$ for every finite simple group $G$ (also see Stein \cite{Stein}). The proof combines powerful probabilistic methods with a detailed analysis of the conjugacy classes and subgroup structure of simple groups. It follows that  every finite simple group is $\frac{3}{2}$-generated, as predicted by Steinberg in his $2$-generation paper of 1962 (see \cite[Section~1]{Steinberg}). These results for simple groups $G$ were extended in a subsequent paper by Breuer, Guralnick and Kantor \cite{BGK} who showed that $u(G) \geqs 2$, with equality if and only if
\begin{equation}\label{e:u2}
G \in \{ {\rm Alt}_5, \, {\rm Alt}_6, \, \O_{8}^{+}(2), \, {\rm Sp}_{2r}(2) \, (r \geqs 3)\}
\end{equation}
(for each of these groups, it is worth noting that $s(G) = 2$). Asymptotic results on the spread and uniform spread of simple groups are established by Guralnick and Shalev in \cite{GSh}.

It is easy to see that if a group $G$ is $\frac{3}{2}$-generated, then every
proper quotient of $G$ is cyclic (that is, $G/N$ is cyclic for all nontrivial normal subgroups $N$ of $G$). The converse statement is false for infinite groups since there exist infinite simple groups that are not finitely generated, such as the alternating group ${\rm Alt}_\infty$. In fact, there even exist finitely generated simple groups that are not $2$-generated (see \cite{Guba}). However, recent work of Donoven and Harper \cite{DH} shows that Thompson's group $V$, and related infinite families of finitely presented groups, are $\frac{3}{2}$-generated.

It is natural to ask if the cyclic quotient property is equivalent to $\frac{3}{2}$-generation for \emph{finite} groups. This is a conjecture of Breuer, Guralnick and Kantor (see \cite[Conjecture~1.8]{BGK}). 

\begin{conj*}
Let $G$ be a finite group. Then $s(G) \geqs 1$ if and only if every proper quotient of $G$ is cyclic.
\end{conj*}

This conjecture has been established in a handful of special cases. For example, see \cite[Theorem~2.01]{BW} for soluble groups and the main theorem of \cite{GK} for simple groups. In this paper, we prove a much stronger form of the conjecture in full generality.

\begin{theorem}\label{t:spread}
Let $G$ be a finite group. Then $s(G) \geqs 2$ if and only if every proper quotient of $G$ is cyclic.
\end{theorem}

As noted above, there are infinitely many finite simple groups $G$ with $s(G)=2$. Moreover,   Corollary~\ref{c:spread2} shows that if $G$ is one of the simple groups in \eqref{e:u2}, then $s(G \wr C_k) = 2$ for all $k \geq 1$.

In \cite{BW}, Brenner and Wiegold prove that every finite soluble group $G$ with $s(G) \geq 1$ satisfies the stronger bound $s(G) \geq 3$ (see \cite[Corollary~2.02]{BW}). In \cite[Problem~1.04]{BW}, they seek a classification of the finite groups $G$ with $s(G)=1$, and they speculate that there are only finitely many such groups. As an immediate corollary to Theorem~\ref{t:spread}, we can now give the definitive solution to this problem, which has remained open since 1975: there are none.

\begin{corol}\label{c:spread}
There is no finite group $G$ with $s(G) = 1$.
\end{corol}

We will prove Theorem~\ref{t:spread} by studying the uniform spread of finite groups. The following result characterises the finite groups $G$ with $u(G)=0$ and $u(G) = 1$.

\begin{theorem}\label{t:uspread}
Let $G$ be a finite group.
\begin{itemize}\addtolength{\itemsep}{0.2\baselineskip}
\item[{\rm  (i)}] $u(G) = 0$ if and only if $G$ has a noncyclic proper quotient, or $G$ is ${\rm Sym}_6$ or $C_p \times C_p$ for a prime $p$.
\item[{\rm (ii)}] $u(G) = 1$ if and only if $G$ has a unique minimal normal subgroup 
\[
N = T_1 \times \cdots \times T_{k} = ({\rm Alt}_6)^{k},
\]
where $k \geqs 2$, $G/N$ is cyclic and $N_G(T_i)/C_G(T_i) = {\rm Sym}_6$ for all $i$.
\end{itemize}
\end{theorem}

The next result is an immediate corollary (for part (ii), observe that ${\rm Sym}_{6}$ can be generated by two $6$-cycles).

\begin{corol}\label{c:conj}
Let $G$ be a finite group such that every proper quotient of $G$ is cyclic. 
\begin{itemize}\addtolength{\itemsep}{0.2\baselineskip}
\item[{\rm (i)}]  If $G$ has even order, then every involution in $G$ belongs to a generating pair. 
\item[{\rm (ii)}] If $G \ne C_p \times C_p$ for a prime $p$, then $G$ can be generated by two conjugate elements.
\end{itemize}
\end{corol}

Recall that a finite group $G$ is \emph{almost simple} if it has a unique minimal normal subgroup $G_0$ that is nonabelian and simple (in particular, $G_0 \leqs G \leqs {\rm Aut}(G_0)$ and $G_0$ is the socle of $G$). As a special case of Theorem~\ref{t:uspread}, we obtain the following result, which highlights the anomaly of the symmetric group of degree $6$.

\begin{corol}\label{c:uspread}
Let $G = \la G_0, g \ra$ be a finite almost simple group with socle $G_0$. Then $u(G)<2$ if and only if $G = {\rm Sym}_6$, in which case $s(G)=2$ and $u(G) = 0$.
\end{corol}

Let $G$ be a finite group and recall that the \emph{generating graph} of $G$ is an undirected graph $\Gamma(G)$ with vertices the nontrivial elements of $G$ so that $x$ and $y$ are adjacent if and only if $G = \la x,y \ra$. This graph was first introduced by Liebeck and Shalev \cite[Section~7]{LSh96} and it has been widely studied in recent years, especially in the setting where $G$ is a simple group (see \cite{Bur19} and the references therein). The following result, which is an immediate corollary of Theorem~\ref{t:spread}, establishes a remarkable dichotomy for generating graphs of finite groups.

\begin{corol} \label{c:gg}
Let $G$ be a finite group and let $\Gamma(G)$ be the generating graph of $G$. Then either
\begin{itemize}\addtolength{\itemsep}{0.2\baselineskip}
\item[{\rm (i)}] $\Gamma(G)$ has isolated vertices; or
\item[{\rm (ii)}] $\Gamma(G)$ is connected and has diameter at most $2$.
\end{itemize}
\end{corol}

We now turn to a further application of spread. Let $G$ be a finite group and let $k \geq d(G)$ be an integer, where $d(G)$ is the smallest size of a generating set for $G$. The vertices of the \emph{product replacement graph} $\Gamma_k(G)$ are the generating $k$-tuples of $G$ and the neighbours of $(x_1, \dots, x_i, \dots, x_k)$ in this graph are $(x_1, \dots, x_ix_j^\pm, \dots, x_k)$ and $(x_1, \dots, x_j^\pm x_i, \dots, x_k)$, for each $1 \leq i \neq j \leq k$. Two generating tuples in $\Gamma_k(G)$ are \emph{equivalent} if they are connected by a path in $\Gamma_k(G)$. A generating tuple is \emph{redundant} if one of the entries can be removed and the remaining entries still generate $G$.

This graph arises naturally in several different contexts. For example, the well known  product replacement algorithm for computing random elements of $G$ involves a random walk on $\Gamma_k(G)$ (see \cite{PRA}). A straightforward argument shows that if $s(G) \geq 2$, then all redundant generating $k$-tuples of $G$ are equivalent for $k > 2$ (see \cite[Lemma~2.8]{Evans}), so Theorem~\ref{t:spread} yields the following corollary. This is related to a much more general conjecture of Pak \cite[Conjecture~2.5.5]{Pak}, which asserts that $\Gamma_k(G)$ is connected for $k > d(G)$.

\begin{corol} \label{c:prg}
Let $k \geq 3$ and let $G$ be a finite group such that every proper quotient is cyclic. Then all redundant generating $k$-tuples are connected in the product replacement graph $\Gamma_k(G)$.
\end{corol}

Let $G$ be a finite group such that every proper quotient is cyclic. We adopt a two-step strategy for proving Theorems~\ref{t:spread} and~\ref{t:uspread}. The first step involves a reduction to almost simple groups; this is the content of Section~\ref{s:reduction}. It is straightforward to reduce to the case where $G$ has a unique minimal normal subgroup $N = T_1 \times \cdots \times T_k$ with each $T_i$ isomorphic to a nonabelian finite simple group $T$. We then proceed by induction on $k$, applying a slightly stronger form of Corollary~\ref{c:uspread} for almost simple groups (see Theorem~\ref{t:almostsimple1}). 

The case $k=1$ is the base for the induction. Let $G = \la G_0, g \ra$ be an almost simple group with socle $G_0$. By the Classification of Finite Simple Groups we know that $G_0$ is an alternating group, a sporadic group or a group of Lie type (classical or exceptional). As previously mentioned, the result for simple groups (the case $G = G_0$) is due to Breuer, Guralnick and Kantor \cite[Theorem~1.2]{BGK}, so we may assume $G \ne G_0$. This setting has been the focus of several recent papers and the desired result has been proved when $G_0$ is one of the following:

\begin{itemize}\addtolength{\itemsep}{0.3\baselineskip} \setlength{\itemindent}{-0.5mm}
\item[{\rm (a)}] ${\rm Alt}_n$: Breuer, Guralnick \& Kantor \cite[Lemma~6.5]{BGK}, Burness \& Harper \cite[Theorem~4.4]{BH}
\item[{\rm (b)}] Sporadic: Breuer, Guralnick \& Kantor \cite[Table~9]{BGK}
\item[{\rm (c)}] ${\rm L}_{n}(q)$: Burness \& Guest \cite[Theorem~2]{BG}
\item[{\rm (d)}] ${\rm PSp}_{2m}(q)$ or $\Omega_{2m+1}(q)$: Harper \cite[Theorem~1]{Harper17} 
\item[{\rm (e)}] ${\rm U}_n(q)$ or ${\rm P}\Omega^{\pm}_{2m}(q)$: Harper \cite[Theorem~2]{HarperClassical}.
\end{itemize}

In view of this earlier work, and with the reduction theorem in hand, it just remains to consider the case where $G_0$ is an exceptional group of Lie type. To complete the picture, in this paper we handle the final remaining case.
 
\begin{theorem}\label{t:exceptional}
Let $G = \la G_0, g \ra$ be a finite almost simple group whose socle $G_0$ is an exceptional group of Lie type. Then $u(G) \geqs 2$. Moreover, if $(G_n)$ is a sequence of almost simple exceptional groups of this form such that $|G_n| \to \infty$, then $u(G_n) \to \infty$.
\end{theorem}

The proof of Theorem~\ref{t:exceptional} is given in Sections~\ref{s:low}--\ref{s:3d4}, with a number of preliminary results presented in Section~\ref{s:prelims}.

By combining the asymptotic statement in Theorem~\ref{t:exceptional} with similar results in \cite{BGK,BG,BH,GSh,Harper17,HarperClassical} for alternating, symmetric and classical groups, we obtain the following corollary. In the statement, $\mathcal{G}$ denotes the collection of almost simple groups of the form $G = \la G_0,g\ra$, where $G_0$ is the socle of $G$.

\begin{corol}\label{c:asymptotic}
Let $(G_n)$ be a sequence of almost simple groups such that $G_n \in \mathcal{G}$ for all $n$ and $|G_n| \to \infty$. In addition, assume $(G_n)$ has no infinite subsequence of groups of Lie type defined over fields of bounded size. Then either $u(G_n) \to \infty$, or $(G_n)$ has an infinite subsequence of
\begin{itemize}\addtolength{\itemsep}{0.2\baselineskip}
\item[{\rm (i)}]  symmetric groups; or
\item[{\rm (ii)}] alternating groups of degree all divisible by a fixed prime.
\end{itemize}
\end{corol}

The exceptions in Corollary~\ref{c:asymptotic} are genuine. In particular, for $n \geq 5$, \cite[Theorem~2]{BH} gives
\[
u({\rm Sym}_{n}) = \left\{\begin{array}{ll}
0 & \mbox{if $n = 6$} \\
2 & \mbox{otherwise.}
\end{array}\right.
\]

Let $G = \la G_0, g \ra$ be an almost simple group whose socle $G_0$ is an exceptional group of Lie type over $\mathbb{F}_q$. At the heart of our proof of Theorem~\ref{t:exceptional} is the probabilistic method for studying uniform spread, which was introduced by Guralnick and Kantor \cite{GK}. This is encapsulated in Lemma~\ref{l:prob_method}, which states that if there exists an element $x \in G_0g$ such that 
\[
\sum_{H \in \mathcal{M}(x)}^{} {\rm fpr}(z,G/H) < \frac{1}{k}
\]
for all nontrivial $z \in G$, then $u(G) \geqs k$, witnessed by $x^G$. Here $\mathcal{M}(x)$ is the set of maximal overgroups of $x$ in $G$ and ${\rm fpr}(z,G/H)$ is the fixed point ratio of $z$, which is the proportion of cosets in $G/H$ fixed by $z$ with respect to the natural transitive action of $G$ on $G/H$. 
Typically, we will aim to derive an explicit upper bound $f(q)$ on the above summation for a suitable choice of element $x$ (and independent of $z$) with the property that $f(q)<\frac{1}{2}$ and $f(q) \to 0$ as $q$ tends to infinity. In particular, the latter property is needed to prove the asymptotic statement in Theorem~\ref{t:exceptional}.

In order to effectively apply this approach, we need to select an appropriate element $x$ in the coset $G_0g$ in such a way that we can get some control on the subgroups in $\mathcal{M}(x)$. Then for each $H \in \mathcal{M}(x)$,  we need to work with upper bounds on the corresponding fixed point ratios. Bounds on the relevant fixed point ratios for exceptional groups are established in \cite{LLS} and this work plays a key role in our analysis (in a few cases, we need to strengthen their bounds for our application). However, several special difficulties arise in the initial step, where we select $x$ and then determine its maximal overgroups. 

In the special case $G = G_0$, Breuer, Guralnick and Kantor \cite{BGK} appeal to work of Weigel \cite{Weigel}, where a specific semisimple element $x \in G$ is identified that is contained in very few maximal subgroups (typically, $N_G(\la x \ra)$ is the unique maximal overgroup). However, for the almost simple groups we are considering in this paper, we need to select $x$ in the coset  $G_0g$ and a different approach is required, which will depend on the type of automorphism $g$. It is worth emphasising that this constitutes a major difference between the simple groups handled in \cite{BGK} and the almost simple groups we are working with in this paper. In particular, there are some substantial technical difficulties to overcome in the almost simple setting.

To handle these difficulties, we will rely heavily on the theory of Shintani descent, which was exploited in \cite{BG} to study the uniform spread of almost simple groups with socle ${\rm L}_{n}(q)$. These techniques have been subsequently extended and developed by Harper in \cite{Harper17,HarperClassical} and they play a key role in this paper (see Section~\ref{ss:p_shintani} for further details). Needless to say, our approach will also use deep results on the maximal subgroups of exceptional groups, due to Liebeck, Seitz and others (see Theorem~\ref{t:types} for example).

\subsection*{Notation} Let $G$ be a finite group and let $n$ be a positive integer. Our group theoretic notation is fairly standard. In particular, we will write $C_n$, or just $n$, for a cyclic group of order $n$ and $G^n$ will denote the direct product of $n$ copies of $G$. An unspecified extension of $G$ by a group $H$ will be denoted by $G.H$. If $X$ is a subset of $G$, then $i_n(X)$ is the number of elements of order $n$ in $X$ and ${\rm meo}(X)$ is the maximal order of an element in $X$. We will use the notation for simple groups from \cite{KL}, so we write ${\rm L}_{n}(q) = {\rm PSL}_{n}(q)$ and $E_6^{-}(q) = {}^2E_6(q)$, etc. For positive integers $a$ and $b$, $\delta_{a,b}$ is the familiar Kronecker delta and we write $(a,b)$ for the greatest common divisor of $a$ and $b$. In this paper, all logarithms are base two. 

\subsection*{Acknowledgements} Guralnick was partially supported by the NSF grant DMS-1901595 and Simons Foundation Fellowship 609771. Burness and Harper thank the Isaac Newton Institute for Mathematical Sciences for support and hospitality during the programme \emph{Groups, Representations and Applications: New perspectives}, when some of the work on this paper was undertaken. This work was supported by: EPSRC grant number EP/R014604/1.

\section{The reduction} \label{s:reduction}

In this section, we establish reduction theorems which reduce the proofs of Theorems~\ref{t:spread} and~\ref{t:uspread} to almost simple groups. We begin by recording some preliminary results. 

\subsection{Preliminaries}\label{ss:reduction_prelims}

Let $G$ be a finite group and recall the definition of the \emph{spread} and \emph{uniform spread} of $G$, denoted by $s(G)$ and $u(G)$, respectively (see Section~\ref{s:intro}). Let us also recall that $s(G) \geqs 1$ only if every proper quotient of $G$ is cyclic. The following elementary result describes the structure of the groups with this property.

\begin{lem}\label{l:quotient}
Let $G$ be a finite group such that every proper quotient of $G$ is cyclic. 
Then one of the following holds:
\begin{itemize}\addtolength{\itemsep}{0.2\baselineskip}
\item[{\rm   (i)}] $G$ is cyclic and $s(G) = u(G) = \infty$. 
\item[{\rm  (ii)}] $G = C_p \times C_p$ for a prime $p$ and $s(G) = p$ and $u(G) = 0$. 
\item[{\rm (iii)}] $G$ is nonabelian with a unique minimal normal subgroup.
\end{itemize}
\end{lem}

\begin{proof}  
We may assume that $G$ is noncyclic. If $G$ is abelian then it is easy to see that (ii) holds (see \cite[Remark~1(c)]{BH}, for example). Now assume $G$ is nonabelian. If $N_1$ and $N_2$ are distinct minimal normal subgroups, then $G/N_i$ is cyclic and thus $G' \leqs N_1 \cap N_2 = 1$, which is a contradiction. Therefore, (iii) holds.
\end{proof}

For the remainder of Section~\ref{s:reduction}, we may assume that $G$ is a nonabelian group with a unique minimal normal subgroup $N = T_1 \times \cdots \times T_{k}$, where for each $i$ the group $T_i$ is isomorphic to a fixed simple group $T$. In addition, we assume throughout that $G/N$ is cyclic.

The case where $N$ is abelian is easy to deal with. 

\begin{lem}\label{l:abel}  
Let $G$ be a finite nonabelian group with a unique minimal normal subgroup $N$. Assume that $N$ is abelian and $G/N$ is cyclic. Then $s(G) = |N|-\e$ and $u(G) = |N|-1$, where $\e=0$ if $|G/N|$ is a prime, and otherwise $\e=1$. In particular, $u(G) \geqs 2$.
\end{lem}

\begin{proof}
For the spread see \cite[Theorem~2.01]{BW}, while the uniform spread follows from \cite[Theorem~1]{BH}, noting that $|N| \geqs 3$ since $G$ is nonabelian.  
\end{proof}

From now on we can assume that the unique minimal normal subgroup $N$ is nonabelian. Observe that $G$ acts transitively by conjugation on $\{T_1, \ldots, T_k\}$ and for each $i$ we have $N_G(T_i)/C_G(T_i) \cong A$, where $A = \< T, y \>$ is an almost simple group with socle $T$. By conjugating in $\Aut(N)$, we may, and will, assume that $G=G_k$, where
\begin{equation}\label{e:Gk}
\mbox{$G_k = \< N, x \>$, $x = (y,1, \dots, 1)\s \in \Aut(N)$, $\s = (1, 2, \dots, k) \in {\rm Sym}_k$.}
\end{equation}
Note that if $k=1$, then we simply have $G_k = A$ and $x = y$. 

Let us now present some preliminary results that we will use in the proofs of our main reduction theorems. The first two are straightforward computations and we omit their proofs.   

\begin{lem}\label{l:powers}   
Let $d$ be a positive integer. Then $x$ is a $d$-th power in $\Aut(N)$ if and only if $(d,k)=1$ and $y$ is a $d$-th power in $\Aut(T)$.  
\end{lem}

\begin{lem}\label{l:diagonal} 
Suppose $k \geqs 2$ and let $p$ be a prime divisor of $k$. Let
\begin{equation}\label{e:Xi}
X_i = T_i \times T_{i+p} \times \cdots \times T_{i+k-p} \cong T^{k/p}
\end{equation} 
for each $i \in \{ 1, \dots, p\}$.
\begin{itemize}\addtolength{\itemsep}{0.2\baselineskip}
\item[{\rm (i)}] Then $x$ acts transitively on $\{X_1, \ldots, X_p\}$ and $x^p$ normalises each $X_i$, inducing the automorphism $(y,1, \ldots, 1)\mu_i \in \Aut(X_i)$, where $\mu_i = (i, i+p, \ldots, i+k-p)$.
\item[{\rm (ii)}] Suppose $D$ is a diagonal subgroup of $X_1 \times \cdots \times X_p$ of the form 
\begin{equation}\label{e:ddiag}
D = \{ (z, z^{\p_1}, \dots, z^{\p_{p-1}}) \,:\, z \in X_1\} \cong T^{k/p}
\end{equation}
with $\varphi_i \in \Aut(X_1)$. Then $x$ normalises $D$ if and only if $\p_i= \p_1^i$ for each $i$ and $x^p = \p_1^p$ as automorphisms of $X_1$. 
\end{itemize}
\end{lem}

We will also need the next two results on the maximal subgroups of $G$ containing $x$. The first one follows by combining \cite[Theorems 1 and 5]{AS}.
Since the proof is so much simpler in this case, we give details.  

\begin{lem} \label{l:maximal2}  
Let $H$ be a maximal subgroup of $G$ containing $x$. Then either 
\begin{itemize}\addtolength{\itemsep}{0.2\baselineskip}
\item[{\rm(i)}] $H = N_G((M \cap T)^k)$, where $M$ is a maximal subgroup of $A$ containing $y$; or
\item[{\rm(ii)}] $H = N_G(D)$, where $D \cong T^{k/p}$ is a diagonal subgroup of $N$ and $p$ is a prime divisor of $k$. 
\end{itemize}
\end{lem}  

\begin{proof}   
Set $J=H \cap N$ and suppose $J = 1$. Then $H = \la x \ra$ since $G/N$ is cyclic. If $x$ has order coprime to $|N|$, then $x$ normalises a Sylow subgroup of $N$, otherwise $x \in N_G(C_N(z))$ for some element $z \in H$  of prime order. Plainly in both cases we get a contradiction, whence $J$ is nontrivial. 

Suppose the projection of $J$ into $T_{\ell}$ is not surjective for some $\ell$. Then $x$ normalises $J_1 \times \cdots \times J_k$, where $J_i$ is the image of the $i$-th projection, and it follows that $y$ normalises
each $J_i$. Replace $J_1$ by a maximal $y$-invariant subgroup of $T_1$ (so  $M=N_A(J_1)$ is a maximal subgroup of $A$). Since $x$ permutes the components transitively and $y$ normalises each $J_i$, it follows that
$J_i = J_1^{x^{i-1}}$ and so (i) holds.   

For the remainder, we may assume that each projection of $J$ into $T_i$ is surjective. It follows that $J \cong S_1 \times \cdots \times S_m$ for some $m \geqs 1$, where each $S_i$ is isomorphic to $T$. Let $Y_j$ be the direct product of the $T_i$ such that $S_j$ projects onto $T_i$. Then 
$N=Y_1 \times \cdots \times Y_m$ and since $x$ acts transitively on $\{T_1, \ldots, T_k\}$, it follows that $x$ permutes the $Y_i$. Therefore, $k=mr$ for some $r$ and we have $J = D_1 \times \cdots  \times D_m$, where $D_i \cong T$ is a diagonal subgroup of $Y_i$. 

Suppose $r$ is composite. Then the set of components in $Y_1$ is not a minimal block for the permutation action of $x$ on $\{T_1, \ldots, T_k\}$. Therefore, we can write $Y_1 = Z_1 \times \cdots \times Z_s$, where the components of $Z_i$ form a minimal block of prime size $r/s$, and we see that $x$ normalises the product of the $E_{ij}$, where
$E_{ij} \cong T$ is the inverse image of the projection of $D_i$ into $Z_j$.   This contradicts the maximality of $H$ and so $r$ is prime and $J \cong T^{k/r}$. Therefore (ii) holds and $H$ is completely determined by $D_1$ (which corresponds to a minimal block of imprimitivity containing $T_1$). 
\end{proof}  

The next result is essentially a special case of the main results of \cite{AS}.  It also follows from Lemma \ref{l:maximal2}.  

\begin{lem} \label{l:maximal}  
Let $p$ be a prime divisor of $k$ and define $X_i$ as in \eqref{e:Xi}. Let $H$ be a maximal subgroup of $G$ containing $x$ such that the projection of $H \cap N$ onto $X_1$ is surjective. 
Then $H=N_G(D)$ where $D = D_{\varphi}$ is a diagonal subgroup of $N$ of the form 
\begin{equation}\label{e:dphi}
D_{\varphi} = \{ (z, z^{\p}, z^{\p^2}, \dots, z^{\p^{p-1}}) \,:\, z \in X_1\}
\end{equation}
and $\p$ is an automorphism of $X_1$ with $\p^p = x^p$.   Moreover, $p^2$ does not divide $k$.  
\end{lem}

\begin{proof}  
Since $x$ acts transitively on $\{X_1, \dots, X_p\}$, it follows that the projection of $H \cap N$ into each $X_i$ is surjective. Then by applying the main theorem of \cite{AS}, or by inspecting the proof of Lemma \ref{l:maximal2}, we deduce that each maximal subgroup of $G$ containing $H \cap N$ is the normaliser of a diagonal subgroup $D$ of $N$ corresponding to a minimal $x$-invariant partition of $\{T_1, \ldots, T_k\}$ of size $r$ with $r$ prime. In particular, $D \cong T^{k/r}$.   

Let $Y = T_1 \times T_{k/r+1} \times \cdots \times T_{k+1-k/r} \cong T^r$ be the product of the components of $N$ corresponding to the $r$ conjugates of $T_1$ under $\langle x^{k/r} \rangle$.  Suppose that either $r \ne p$ or $p^2$ divides $k$. Then $Y \leqs X_1$ and thus $D$ projects onto $Y$ (since it projects onto $X_1$).  But by the proof of Lemma \ref{l:maximal2} (or the main theorem of \cite{AS}), we see that the image of the projection of $D$ into $Y$ is isomorphic to $T$. This is a contradiction and we conclude that $r = p$ and $p^2$ does not divide $k$.  Finally, since $x$ normalises $D$, we deduce that $D$ has the given form by applying Lemma~\ref{l:diagonal}(ii).
\end{proof}  

Although stronger versions of the following result are available, this will be sufficient for our application. 

\begin{lem}\label{l:cent}
Suppose $G = \< N,z \>$, where $z \in \Aut(N)$ transitively permutes the components of $N$. If $G = \< h,z \>$ for some $h \in N$ with $h^N = h^{\Aut(N)}$, then 
\[
|C_{\Aut(N)}(z)|  \leqs |N:C_N(h)| \leqs \frac{1}{3^k}|N|.
\]
\end{lem}

\begin{proof}  
Set $Z = C_{\Aut(N)}(z)$. Since $G$ contains $N$ we have $C_{\Aut(N)}(G)=1$ and thus $C_Z(h) = Z \cap C_{\Aut(N)}(h)=1$. Therefore $|Z| = |h^Z| \leqs |h^{\Aut(N)}| = |h^N|$. Since $h^N = h^{\Aut(N)}$, each coordinate of $h$ is nontrivial and thus $|C_N(h)| \geqs 3^k$ (there are no self-centralising involutions in $T$). 
\end{proof}

\begin{rem}\label{r:cent}
We will apply Lemma \ref{l:cent} in the proof of Theorem \ref{t:red_odd}. In this setting, we will work with an element $z \in G$ such that for any nontrivial $h \in N$ there exists $g \in G$ such that $G = \la h^g,z \ra$. We can then apply the lemma because there exists an involution $h \in N$ with $h^N = h^{\Aut(N)}$ by \cite[Lemma 12.1]{FGS}.
\end{rem}

The proof of our main reduction theorem (see Theorem~\ref{t:red_odd}) relies on the following deep result for almost simple groups. As explained in the proof below, this follows by combining earlier work in the literature with the proof of Theorem~\ref{t:exceptional} in this paper.

\begin{thm}\label{t:almostsimple1} 
Let $G = \< G_0, g \>$ be an almost simple group with socle $G_0$ and assume that $G \neq {\rm Sym}_6$. Then $u(G) \geq 2$, and this is witnessed by a class $y^G$ such that 
\begin{itemize}\addtolength{\itemsep}{0.2\baselineskip}
\item[{\rm   (i)}] the order of $\<y\> \cap G_0$ does not divide $4$; or
\item[{\rm  (ii)}] $\<y\> \cap G_0$ is nontrivial and $y$ is not a square in $\Aut(G_0)$; or
\item[{\rm (iii)}] $G= \mathrm {Alt}_6$ and $y$ has order $4$. 
\end{itemize}
\end{thm} 

\begin{rem}\label{r:alt6}
As noted in the proof below, if $G = G_0 = {\rm Alt}_6$ then $y^G$ with $y=(1,2,3,4)(5,6)$ is the only class to witness the bound $u(G) \geqs 2$. Here $\< y \> \cap G_0 = \< y \>$ has order $4$ and $y$ is a square in $\Aut(G_0)$, which explains why (iii) is required in the statement of Theorem~\ref{t:almostsimple1}.
\end{rem}

\begin{proof}[Proof of Theorem \ref{t:almostsimple1}]  
First assume that $G_0 \neq {\rm Alt}_6$. As explained in Section \ref{s:intro}, by combining Theorem \ref{t:exceptional} (which is of course independent of the reduction theorems we are considering here) with the main results in \cite{BGK,BG,BH,Harper17,HarperClassical} we see that $u(G) \geq 2$ is witnessed by a class $y^G$, say. Since $y$ is necessarily not contained in any proper normal subgroup of $G$, without loss of generality we may assume that $y^G \subseteq G_0g$. Therefore, it suffices to show that $y$ can always be chosen to satisfy one of the conditions (i) or (ii) in the statement.

First assume that $G_0$ is alternating or sporadic (we continue to assume that $G_0 \neq {\rm Alt}_6$). Suppose $G \ne G_0$, which implies that $G = {\rm Aut}(G_0)$ and $|G:G_0|=2$. Here $G_0g$ is not a square in $\Out(G_0)$ and therefore $y \in G_0g$ is not a square in $\Aut(G_0)$. Moreover, for any involution $x \in G$ there exists $h \in G$ such that $G=\<x,y^h\>$, which implies that $|y| > 2$ and thus $\<y\> \cap G_0 \geq \<y^2\> \neq 1$, so condition (ii) holds. Now assume $G$ is simple. Here we inspect the class $y^G$ identified in \cite{BGK} that witnesses $u(G) \geq 2$. If $G = {\rm Alt}_n$, then $y = (1, \dots, n)$ if $n \geq 5$ is odd (see \cite[Proposition~6.7]{BGK}) and $y = (1, \dots, m-k)(m-k+1, \dots, n)$ if $n=2m \geq 8$ is even, where $k = m-(2,m-1)$ (see \cite[Proposition~6.3]{BGK}). If $G$ is sporadic, then the class $y^G$ is given in \cite[Table~7]{BGK}. In all cases, $|y| \geq 5$, so condition (i) is satisfied.

Next assume $G_0$ is a group of Lie type and let $y^G$ be the class identified in the relevant reference above, which witnesses $u(G) \geqs 2$. If $G = \Aut(G_0)$ and $|G:G_0|$ is even, then condition (ii) is satisfied. Otherwise, by considering each case in turn, we see that $y^{|G:G_0|} \in G_0$ has order at least $5$ and thus condition (i) holds. For instance, if $G_0 = E_8(q)$ and $|G:G_0| = e > 1$, then in the proof of Theorem~\ref{t:e8} we choose $y$ such that $|y^e| = q_0^8+q_0^7-q_0^5-q_0^4-q_0^3+q_0+1 \geq 331$, where $q=q_0^e$.

To complete the proof of the theorem, we may assume that $G = \<G_0,g\>$ with $G_0 = {\rm Alt}_6$, and further that $G \neq {\rm Sym}_6$. If $G = {\rm Alt}_6$, then an easy computation in \textsc{Magma} \cite{magma} demonstrates that $u(G) \geq 2$ and the unique class to witness this is $(1,2,3,4)(5,6)^G$, so (iii) holds. Now assume $G$ is a cyclic extension of $G_0$ isomorphic to either ${\rm PGL}_2(9)$ or ${\rm M}_{10}$. Here a \textsc{Magma} computation shows that $u(G) \geq 2$, witnessed by $y^G$, say. Condition (ii) is satisfied as $|y| > 2$ and $y \in G_0g$ is not square in $\Aut(G_0)$ since $G_0g$ is not square in $\Out(G_0) = C_2 \times C_2$.
\end{proof}

The following result, which also follows from the main lemma of \cite[Section~2]{LM}, is an immediate corollary of Theorem~\ref{t:almostsimple1} (note that the result is trivial for $G = {\rm Sym}_6$).

\begin{cor}\label{c:lm}
Let $G$ be an almost simple group with socle $G_0$. Then for all $g \in G\setminus G_0$, there exists $h \in G_0g$ such that the order of $h$ is greater than the order of $G_0g \in G/G_0$. 
\end{cor}

\subsection{The main reduction theorem} \label{ss:reduction_main}

Let $G$ be a finite group with a unique minimal normal subgroup 
$N = T_1 \times \cdots \times T_k$, where each $T_i$ is isomorphic to a fixed nonabelian simple group $T$ and $N_{G}(T_i)/C_{G}(T_i) \cong A = \la T, y \ra$ for each $i$. Let us assume that $G/N$ is cyclic.

As previously explained, we may assume that $G=G_k = \la N, x \ra$ (see \eqref{e:Gk}), where 
\[
x = (y, 1, \ldots, 1)\s, \;\; \s = (1, \ldots, k) \in {\rm Sym}_k.
\]
Moreover, in this section we will assume that $A \ne {\rm Sym}_6$ (the special case $A = {\rm Sym}_6$ will be addressed in Section \ref{ss:reduction_sym6}). This means that we may, and will, assume that the element $y$ in the definition of $x$ satisfies the conclusions of Theorem~\ref{t:almostsimple1}, namely:
\begin{itemize}\addtolength{\itemsep}{0.2\baselineskip}
\item[{\rm (I)}] For all nontrivial $r,s \in A$, there exists $z \in A$ such that $A = \la r^z,y \ra = \la s^z ,y \ra$;
\item[{\rm (II)}] $\langle y \rangle \cap T \ne 1$; and 
\item[{\rm (III)}] If $y$ is a square in ${\rm Aut}(T)$, then either $|\langle y \rangle \cap T|$ does not divide $4$, or $A = {\rm Alt}_6$ and $|y|=4$.
\end{itemize} 
In particular, for $k=1$ we observe that $x^{G_k}$ witnesses $u(G_k) \geqs 2$.

Our first result handles the special case where $k$ is a power of $2$.

\begin{thm}\label{t:red_even}
If $A \ne {\rm Sym}_6$ and $k = 2^e \geqs 2$, then $x^{G_k}$ witnesses $u(G_k) \geqs 2$.
\end{thm}

\begin{proof} 
We proceed by induction on $e$. Notice that it suffices to show that for any elements $a,b \in G_k$ of prime order, there exists $g \in G_k$ such that $G_k = \la a, x^g \ra = \la b, x^g \ra$. 
   
First assume $e=1$, so $G=G_2 = \la N, x \ra$, where $N= T_1 \times T_2$ and $x^2 = (y,y)$. The special case $A = {\rm Alt}_6$ can be checked by direct computation, so we will assume $A \ne {\rm Alt}_6$ for the remainder of the proof for $k=2$. 

Suppose $a,b \in G$ have prime order. There are two types of prime order elements in $G$, namely:
\begin{itemize}\addtolength{\itemsep}{0.2\baselineskip}
\item[{\rm (i)}] Elements $(a_1,a_2) \in  \Aut(T_1)  \times \Aut(T_2)$ of prime order; and
\item[{\rm (ii)}] Involutions of the form $(a_1,a_1^{-1})\s$ with $a_1 \in \Aut(T_1)$.
\end{itemize}
Note that elements of type (ii) exist if and only if $Ty \in A/T$ has odd order. These two types of prime order elements give us three separate cases to consider.

\vs

\noindent \emph{Case 1. $a = (a_1,a_2)$ and $b=(b_1,b_2)$.}

\vs

Suppose that for each $i$, either $a_i$ or $b_i$ is trivial. In view of (I) above, by conjugating we may assume that $\langle a,x^2 \rangle$ projects onto $T_1$ or $T_2$. By applying Lemma \ref{l:maximal}, it follows that any maximal overgroup of $\la a, x\ra$ in $G$ is of the form $N_G(D_{\varphi})$, where 
\begin{equation}\label{e:dphi2}
D_{\varphi}:=\{(z,z^{\varphi}) \,:\, z \in T_1 \}
\end{equation}
for some $\varphi \in \Aut(T_1)$. But since $a$ has at least one trivial component, it does not normalise such a diagonal subgroup and thus $G = \la a,x \ra$. Similarly, we deduce that $G = \la b, x \ra$. 

We can now assume that $a_i$ and $b_i$ are both nontrivial for some $i$. By conjugating $a$ and $b$ simultaneously, we may assume that $a_1$ and $b_1$ are nontrivial. By a further conjugation, and by appealing to condition (I) above, we may assume that $\langle a_1, y \rangle$ and $\langle b_1,y  \rangle$ project onto $T_1$. Then Lemma \ref{l:maximal} implies that $G \ne \la a,x \ra$ if and only if $\la a,x \ra$ normalises a diagonal subgroup $D_{\varphi}$ of $N$ as in \eqref{e:dphi2}, where $\varphi \in \Aut(T_1)$ and $\varphi^2 = y$ as automorphisms of $T_1$. Note that in this situation we have $a_2 = a_1^{\varphi}$ and 
\[
(y,y) \in \{ (z, z^{\varphi}) \,:\, z \in {\rm Aut}(T_1)\},
\]
so $y^{\varphi} = y$. Moreover, since $\langle a_1, y \rangle$ projects onto $T_1$, it follows that $\varphi$ is uniquely determined by $a_1^{\varphi}$ and we deduce that $a$ is contained in the normaliser of at most one such diagonal subgroup. Similarly, $b$ normalises at most one such subgroup.

Suppose $\la a,x \ra$ normalises $D_{\varphi}$. Let $c = (1,t) \in N$ with $t \in \langle y  \rangle \cap T_2$, 
so $a^c = (a_1,a_2^t)$. If $G \ne \langle a^c, x \rangle$ then by arguing as above we see that 
$\la a^c,x\ra$ normalises $D_{\theta}$ for some $\theta \in \Aut(T_1)$ with $\theta^2=y$ and $a_2^t = a_1^{\theta}$. Then $a_1^{\varphi t} = a_1^{\theta}$ and thus $\varphi t = \theta$ since both automorphisms are uniquely determined by their effect on $a_1$. Since $t$ and $\varphi$ commute, it follows that  
\[
yt^2 = \varphi^2t^2 = (\varphi t)^2 = \theta^2 = y
\]
and thus $t^2=1$, so either $|\langle y \rangle \cap T_2|$ is odd and $t=1$, or there are two possibilities for $t$. In view of (III) above, noting that $y$ is a square in $\Aut(T_1)$, we see that $|\langle y \rangle \cap T_2|$ does not divide $4$. By combining these observations, it follows that if we choose $t \in \langle y \rangle \cap T_2$ at random then $t^2=1$ with probability at most $1/3$, whence $G = \la a^c,x\ra$ with probability at least $2/3$. The same argument applies with $a^c$ replaced by $b^c$ and we conclude that there exists $c \in N$ such that $G = \la a^c,x\ra = \la b^c,x\ra$.

\vs

\noindent \emph{Case 2. $a = (a_1,a_1^{-1}) \s$ and $b=(b_1,b_1^{-1})\s$.}

\vs
 
Suppose $a = (a_1,a_1^{-1}) \s$ and $b=(b_1,b_1^{-1})\s$ are involutions in $G$. By conjugating, we may assume that both $a_1$ and $b_1$ are nontrivial, and then a second conjugation by a diagonal element allows us to assume that $\la a_1, y \ra$ and $\la b_1, y \ra$ both project onto $T_1$.  

Suppose $G \ne \langle a,x \rangle$. Then Lemma \ref{l:maximal} implies that $\langle a,x \rangle$ normalises a diagonal subgroup $D_{\varphi}$ of $N$ as in \eqref{e:dphi2}, where $\varphi \in \Aut(T_1)$ and 
$\varphi^2 = y$. Since $a = \s^{(1, a_1)}$ and the only diagonal subgroups of $N$ normalised by $\s$ are those of the form $D_{\psi}$ with $\psi^2 = 1$, it follows that any diagonal subgroup normalised by $a$ is of the form $D_{\psi a_1}$ with $\psi^2 = 1$, whence $\varphi = \psi a_1$ and $(\psi a_1)^2 = y$. Similarly, if $G \ne \la b,x \ra$ then $(\theta b_1)^2 = y$ for some $\theta \in \Aut(T_1)$ with $\theta^2 = 1$.

If $y$ is not a square in $\Aut(T)$, then $G = \la a,x\ra = \la b,x \ra$ and the result follows. So let us assume $y$ is a square, so (III) implies that $|\la y \ra \cap T_2|$ does not divide $4$. Let $c = (1,t) \in N$ with $t \in \la y \ra \cap T_2$ and note that $a^c = (a_1t,t^{-1}a_1^{-1})\s$ and similarly for $b^c$.
Suppose that $a^c$ normalises $D_{\varphi}$ with $\varphi^2=y$.  As above this implies that $\varphi = \psi a_1 t$ for some $\psi \in \Aut(T_1)$ with 
$\psi^2=1$. Then $t$ and $\psi a_1t$ both centralise $y$, so $\psi a_1$ centralises $y$ and therefore $t$ as well. It follows that $t^2 = y(\psi a_1)^{-2}$. Clearly there are at most two elements in the cyclic group $\la y \ra \cap T_2$ with this property, so the condition in (III) implies that if we choose $t$ at random, then the probability that $G = \la a^c,x\ra$ is at least $2/3$. 
By the same argument, $G = \la b^c,x\ra$ with probability at least $2/3$ and hence there exists $c \in N$ such that $G = \la a^c,x \ra = \la b^c,x \ra$.

 \vs

\noindent \emph{Case 3. $a = (a_1,a_2)$ and $b=(b_1,b_1^{-1})\s$.}

\vs

By conjugating, we may assume that $\langle a_1, y \rangle$ and $\langle b_1, y \rangle$ project onto $T_1$. As before, if neither $a$ nor $b$ normalise a diagonal subgroup, then $G = \la a,x\ra = \la b,x \ra$ and we are done. 

Suppose $\la a, x \ra$ normalises $D_{\varphi}$, so $\varphi^2 = y$ and $a_2 = a_1^{\varphi}$. Consider an element $c=(1,t) \in N$ with $t \in \la y \ra \cap T_2$. By arguing as in Case 1, $G \ne \la a^c, x \ra$ if and only if $t^2=1$. Similarly, by recalling the argument in Case 2 we see that $G \ne \la b^c,x\ra$ if and only if $b^c$ normalises a diagonal subgroup $D_{\theta}$, where $\theta^2 = y$ and $\theta = \psi b_1t$ with $t^2 = y(\psi b_1)^{-2}$. As explained in Cases 1 and 2, if we choose $t \in \la y \ra \cap T_2$ at random then with positive probability we have $t^2 \ne 1$ and $t^2 \ne y(\psi b_1)^{-2}$, so there exists $c \in N$ with $G = \la a^c,x\ra = \la b^c,x\ra$.

To complete the argument we can assume that $G = \la a^c,x\ra$ for all $c = (1,t) \in N$ with $t \in \la y \ra \cap T_2$. Then by arguing as in Case 2, if we choose such an element $c$ at random, then $G = \la b^c,x\ra$ with probability at least $2/3$. The result follows.

\vs

To complete the proof of the theorem, we may assume that $k =2^e \geqs 4$.  Write $G = G_k = \la N, x \ra$, where $N = X_1 \times X_2$ and
\[
X_1 = T_1 \times T_3 \times \cdots \times T_{k-1},\;\; X_2 = T_2 \times T_4 \times \cdots \times T_{k}
\]
as in Lemma \ref{l:diagonal} (with $p=2$). Let us observe that every element in $G$ of prime order normalises $X_1$ and $X_2$. 
 
Let $a = (a_1, a_2)$ and $b = (b_1,b_2)$ be elements in $G$ of prime order, where $a_i,b_i \in {\rm Aut}(X_i)$. By simultaneously conjugating $a$ and $b$ by a suitable element of $G$, and by applying the inductive hypothesis, we may assume that both $\< a,x^2 \>$ and $\< b,x^2 \>$ project onto at least one of $X_1$ and $X_2$. Since $x$ interchanges $X_1$ and $X_2$, Lemma \ref{l:maximal} implies that the only possible maximal overgroups of $\la a,x\ra$ in $G$ are the normalisers of diagonal subgroups $D_{\varphi} \cong T^{k/2}$ of $N$ as in \eqref{e:dphi}, where $\varphi \in {\rm Aut}(X_1)$ and $\varphi^2=x^2$ as automorphisms of $X_1$. However, there is no such automorphism $\varphi$ by Lemma \ref{l:powers} and we conclude that $G = \langle  a,x \rangle$. The same argument shows that $G = \langle  b,x \rangle$ and the result follows.
\end{proof}

We can now establish our main reduction theorem.
 
\begin{thm}\label{t:red_odd}     
If $A \ne {\rm Sym}_6$ and $k \geqs 1$, then $x^{G_k}$ witnesses $u(G_k) \geqs 2$.
\end{thm}

\begin{proof}    
We proceed by induction on $k$. As before, it suffices to show that for any elements $a,b \in G_k$ of prime order, there exists $g \in G_k$ such that $G_k = \la a, x^g \ra = \la b, x^g \ra$. 

The base case $k=1$ is clear since $x=y$ has been chosen via Theorem \ref{t:almostsimple1} so that $x^{G_1}$ witnesses $u(G_1) \geqs 2$. In addition, the result follows from Theorem \ref{t:red_even} if $k=2^e \geqs 2$. Therefore, we may assume that $k$ is divisible by an odd prime $p$. As in Lemma \ref{l:diagonal}, let 
\[
X_i = T_i \times T_{i+p} \times \cdots \times T_{i+k-p} \cong T^{k/p}
\]
for each $i \in \{1, \ldots, p\}$. Let $a$ and $b$ be elements of $G_k$ of prime order. Note that the action of $a$ (and similarly $b$) on $\{X_1, \ldots, X_p\}$ is either trivial or transitive (indeed, if $a$ normalises some $X_i$, then it normalises every $X_i$). It follows that there are three cases to consider, according to the actions of $a$ and $b$ on $\{X_1, \ldots, X_p\}$. For the remainder of the proof, we will write $G = G_k$.

\vs

\noindent \emph{Case 1. Both $a$ and $b$ act trivially on $\{X_1, \ldots, X_p\}$.}

\vs

First we assume $a$ and $b$ both normalise some (and hence all) $X_i$. Write 
\[
a=(a_1, \ldots, a_{p}),\;\; b=(b_1, \ldots, b_{p}),
\]
with $a_i, b_i \in {\rm Aut}(X_i)$. 

Suppose that for each $i$, either $a_i$ or $b_i$ is trivial.  By the inductive hypothesis, we can assume that $\langle a,x^p \rangle$ projects onto $X_i$ for some $i$ and similarly $\langle b,x^p \rangle$ projects onto $X_j$ some $j$. By applying Lemma \ref{l:maximal}, it follows that any maximal overgroup of $\la a, x\ra$ in $G$ is of the form $N_G(D_{\varphi})$, where $D_{\varphi}$ is a 
diagonal subgroup of $N$ as in \eqref{e:dphi}. But we are assuming that $a$ has at least one trivial component, so $G = \la a,x \ra$ since $a$ does not normalise such a diagonal subgroup. Similarly, we deduce that $G = \la b, x \ra$. 

Therefore, we may assume that $a_i$ and $b_i$ are both nontrivial for some $i$. By conjugating $a$ and $b$ simultaneously, we may assume that $a_1$ and $b_1$ are nontrivial. Then by applying the inductive hypothesis, we can conjugate $a$ and $b$ simultaneously so that
$\langle a,x^p \rangle$ and $\langle b,x^p \rangle$ both project onto $X_1$. As above, the only possible maximal overgroups of $\la a, x\ra$ in $G$ are the normalisers of diagonal subgroups $D_{\varphi}$ as in \eqref{e:dphi}, where $\varphi \in \Aut(X_1)$ and $\varphi^p=x^p$ as automorphisms of $X_1$. The latter equality implies that  
\[
(x^p, \ldots, x^p) \in \{ (z, z^{\varphi}, \ldots, z^{\varphi^{p-1}}) \,:\, z \in {\rm Aut}(X_1)\}
\]
and thus $(x^p)^{\varphi} = x^p$. Moreover, since $\langle a_1,x^p \rangle$ projects onto $X_1$,
we see that $\varphi$ is uniquely determined by $a_1^{\varphi}$ and thus $a$ is contained in the normaliser of at most one such diagonal subgroup. Similarly, $b$ normalises at most one such subgroup.

Suppose $G \ne \la a,x\ra$ and let $N_G(D_{\varphi})$ be the unique maximal overgroup of $\la a,x\ra$ in $G$. Set $c=(1,1, c_3, \ldots, c_p) \in N$ with $c_i \in X_i$, so $a^c = (a_1,a_2,a_3^{c_3}, \ldots, a_p^{c_p})$. Since the first component of $a^c$ is $a_1$, the previous argument implies that either $G = \la a^c,x\ra$, or $\la a^c,x \ra$ normalises $D_{\varphi}$ and we have $a_i^{c_i} = a_1^{\varphi^{i-1}}$ for $i = 3, \ldots, p$. Since there are at most $|C_{X_i}(a_i)|$ elements $c_i \in X_i$ with $a_i^{c_i} = a_1^{\varphi^{i-1}}$, if we choose such an element $c$ at random, then 
the probability that $G = \la a^c,x\ra$ is at least
\[
1 - \prod_{i=3}^{p} |X_i: C_{X_i}(a_i)|^{-1} \geqs \frac{4}{5}.
\]
In the same way, the probability that $G = \la b^c,x\ra$ is at least $4/5$. Therefore, 
there exists $c$ as above with $G = \langle a^c,x \rangle = \langle b^c,x \rangle$ and the result follows. 

\vs

\noindent \emph{Case 2. Both $a$ and $b$ act transitively on $\{X_1, \ldots, X_p\}$.}

\vs

Here $a$ and $b$ have order $p$ and we may write 
\[
x = (x^p, 1, \ldots, 1)\gamma \in ({\rm Aut}(X_1) \times \cdots \times {\rm Aut}(X_p)){:}{\rm Sym}_p,
\] 
where $\gamma = (1,2,\ldots, p) \in {\rm Sym}_p$ and we view $x^p$ as an automorphism of $X_1$ (see Lemma \ref{l:diagonal}(i)). Then
\[
a = (a_1, \ldots, a_{p})\gamma, \;\; b=(b_1, \ldots, b_{p})\gamma,
\]
with $a_i,b_i \in {\rm Aut}(X_i)$. Note that $\prod_{i}a_i =  \prod_{i}b_i  =1$ (since $|a|=|b|=p$).

Conjugating $a$ and $b$ simultaneously by an element $(c_1, 1, \ldots, 1) \in N$ with $c_1 \in X_1$, we may assume that both $a_1$ and $b_1$ are nontrivial. Then  conjugating by an element of the form $(c, \ldots, c) \in N$,
we may (by the inductive hypothesis) assume that $\langle a_1,x^p \rangle$ and $\langle b_1,x^p \rangle$ both contain subgroups projecting onto $X_1$.   
By Lemma \ref{l:maximal}, it follows that the only possible maximal subgroups of $G$ containing
either $\la a, x\ra$ or $\la b,x \ra$ are the normalisers of diagonal subgroups $D_{\varphi}$ of $N$ as in \eqref{e:dphi}.

Suppose $G \ne \la a,x \ra$, so $\la a, x\ra$ normalises $D_{\varphi}$. Here $\varphi^p=x^p$ as automorphisms of $X_1$ and as in Case 1 we note that 
$(x^p)^{\varphi} = x^p$ and $\varphi$ is uniquely determined by $a_1^{\varphi}$. Since $ax^{-1} = (a_1x^{-p}, a_2, \ldots, a_p)$ also normalises $D_{\varphi}$, it follows that $a_i = a_1^{\varphi^{i-1}}x^{-p}$ for $i =2, \ldots, p$ and thus $N_G(D_{\varphi})$ is the unique maximal overgroup of $\la a, x \ra$ in $G$. 

Set $c = (1,\ldots, 1, d,1) \in N$ with $d  \in X_{p-1}$, so 
\[
a^c = (a_1, \ldots, a_{p-2}, d^{-1}a_{p-1},a_pd)\gamma.
\]
Notice that the first component of $a^c$ is still $a_1$, so either $G = \la a^c,x\ra$, or $a^c$ normalises $D_{\varphi}$. Let us assume $a^c$ normalises $D_{\varphi}$. Then $a^cx^{-1}$ also normalises $D_{\varphi}$ and this implies that $d^{-1}a_{p-1}$ is $C_{\Aut(X_1)}(x^p)$-conjugate to $a_1$. Let $h \in X_1$ be an involution with $h^{X_1} = h^{\Aut(X_1)}$ (see Remark \ref{r:cent}), so by the inductive hypothesis there exists $g \in  \la X_1, x^p \ra$ such that $\la X_1, x^p \ra = \la h^g, x^p \ra$. Then by applying Lemma \ref{l:cent} we deduce that if we choose $d  \in X_{p-1}$ at random, then the probability that $d^{-1}a_{p-1}$ is $C_{\Aut(X_1)}(x^p)$-conjugate to $a_1$ is at most $1/3$. In particular, the probability that $G = \la a^c,x\ra$ is at least $2/3$ and an entirely similar argument gives the same conclusion with $a^c$ replaced by $b^c$. Therefore, there exists $c \in N$ as above such that $G = \langle a^c, x \rangle = \langle b^c, x \rangle$. 

\vs

\noindent \emph{Case 3. $a$ acts trivially and $b$ act transitively on $\{X_1, \ldots, X_p\}$.}

\vs

As above, we may write
\[
a = (a_1, \ldots, a_p),\;\; b = (b_1, \ldots, b_p)\gamma,
\]
where $a_i,b_i \in \Aut(X_i)$ and $\gamma = (1, \ldots, p) \in {\rm Sym}_p$. By applying the inductive hypothesis, and by replacing $a$ and $b$ by suitable (simultaneous) conjugates, we may assume that $\la a_1, x^p \ra$ and $\la b_1, x^p \ra$ both project onto $X_1$. Then either $G = \la a,x\ra$, or $\la a,x \ra$ normalises a diagonal subgroup $D_{\varphi}$ as in \eqref{e:dphi}, where $\varphi \in C_{\Aut(X_1)}(x^p)$ and $\varphi$ is uniquely determined by $a_1^{\varphi}$. And similarly for $\la b,x\ra$.

Set $c = (1, \ldots, 1, c_{p-1},c_p) \in N$, where $c_{p-1} \in X_{p-1}$ and 
$c_p \in  X_p \cap \langle x^p \rangle$.  Then 
\[
b^cx^{-1} = (b_1c_px^{-p},b_2, \ldots, b_{p-2},c_{p-1}^{-1}b_{p-1},c_p^{-1}b_pc_{p-1})
\]
and we note that $\langle x^p, b_1c_p \rangle$ projects onto $X_1$ (since 
$c_p \in \langle x^p \rangle$).  

If $G \ne \la b^c,x\ra$ then $b^cx^{-1}$ must normalise a diagonal subgroup $D_{\varphi}$ and thus $c_{p-1}^{-1}b_{p-1}$ is $C_{\Aut(X_1)}(x^p)$-conjugate to $b_1c_px^{-p}$. As we argued in Case 2, if we fix $c_p \in X_p \cap \langle x^p \rangle$ and we choose $c_{p-1} \in X_{p-1}$ at random, then the probability that $c_{p-1}^{-1}b_{p-1}$ is $C_{\Aut(X_1)}(x^p)$-conjugate to $b_1c_px^{-p}$ is at most $1/3$. In particular, the probability that $G = \la b^c,x\ra$ is at least $2/3$.

If $G \ne \la a^c,x\ra$ then $a^c$ normalises some $D_{\theta}$, where 
$a_p^{c_p}= a_1^{\theta^{p-1}}$ and $\theta$ is uniquely determined by $a_1^{\theta}$. Since $(x^p)^{\theta} = x^p$ and $\la a_1, x^p \ra$ projects onto $X_1$, it follows that $\la a_p,x^p \ra$ projects onto $X_p$ and thus the conjugates $a_p^{c_p}$ are distinct as $c_p$ runs through $X_p \cap \langle x^p \rangle$. In particular, there is at most one $c_p$ such that $a_p^{c_p}= a_1^{\theta^{p- 1}}$. Therefore, if we fix $c_{p-1} \in X_{p-1}$ and choose $c_p \in X_p \cap \langle x^p \rangle$ at random, then the probability that $G = \la a^c,x\ra$ is at least $1 - |\langle x^p \rangle \cap X_p|^{-1} \geqs 1/2$ (note that $\langle x^p \rangle \cap X_p \ne 1$ by condition (II) above).

Finally, by combining the two previous arguments we conclude that there exists $c \in N$ such that $G = \langle a^c, x \rangle = \langle b^c, x \rangle$.
\end{proof} 

Subject to proving Theorem~\ref{t:exceptional}, by Theorem~\ref{t:red_odd} we conclude that the proofs of Theorems~\ref{t:spread} and~\ref{t:uspread} are complete, unless $A = {\rm Sym}_6$. The groups $G_k$ for which $A = {\rm Sym}_6$ are handled in Theorem~\ref{t:s6} in the following section.

\subsection{The special case \texorpdfstring{$A = {\rm Sym}_6$}{A = Sym(6)}} \label{ss:reduction_sym6}

For the proof of Theorem \ref{t:s6}, it will be useful to introduce some additional notation. Let $G$ be a finite group with a unique minimal normal subgroup $N$. Write $s_0(G)$ for the largest integer $k \geqs 0$ such that for any nontrivial elements $x_1, \ldots, x_k$ of $N$, there exists $y \in G$ with $G = \la x_i,y\ra$ for all $i$. Define $u_0(G)$ in the same way, with the condition $y \in G$ replaced by $y \in C$, where $C$ is a specified conjugacy class of $G$. Clearly, we have $s(G) \leqs s_0(G)$ and $u(G) \leqs u_0(G)$. 

The following observation will be useful. Here $G_k$ is defined as in \eqref{e:Gk}.

\begin{lem}\label{l:upperbound}  
We have $s_0(G_k) \leqs s_0(A)$ and $u_0(G_k) \leqs u_0(A)$.
\end{lem}

\begin{proof}  
We prove the first inequality; the proof of the second is essentially the same. Write $s_0(A) = m-1$ and fix nontrivial elements $y_1, \dots, y_m \in T$ such that no element of $A$ generates with each of the $y_i$. Seeking a contradiction, suppose that $s_0(G_k) \geq m$. 

For each $i$, let $x_i = (y_i, 1, \dots, 1) \in N$. Suppose that $w$ generates with each $x_i$. Since $w$ necessarily permutes the $k$ factors of $N$ transitively, by replacing $w$ with a suitable power, we can assume that $w = (w_1,\dots,w_k)\s$. Set $g = (1,w_2 w_3 \cdots w_k, w_3 \cdots w_k, \dots, w_k) \in A^k$ and $v = w_1 \cdots w_k \in A$. Then $w^g = (v, 1, \dots, 1)\s$ and $x_i^g=x_i$, whence  $G_k^g = \< x_i, w^g \>$ for all $i$. Since 
$\< N, x^k \>  = \< N, x^k \>^g \leq G_k^g$, we deduce that $A =  \< y_i, v \>$ for all $i$ and we have reached a contradiction.
\end{proof}

We now complete our reduction. 

\begin{thm}\label{t:s6}
Suppose $G = G_k$ and $A = {\rm Sym}_6$. Then $s(G) = 2$ and $u(G) = 1-\delta_{1,k}$.
\end{thm} 

\begin{proof}
We begin by establishing upper bounds on $s(G)$ and $u(G)$. By Lemma \ref{l:upperbound} we have $s(G) \leqs s_0(G) \leqs s_0(A)$ and it is easy to check that $s_0(A) \leqs 2$. For example, if we take $x_1 = (1,2)(3,4)$, $x_2 = (1,2)(5,6)$ and $x_3 = (3,4)(5,6)$ then there is no $y \in A$ such that $A = \la x_i, y \ra$ for all $i$. Similarly, it is easy to check that $y^A$ witnesses $u_0(A) \geqs 1$ if and only if $y$ has order $6$. But if we take $y = (1,2,3)(4,5)$, $x_1 = (1,2,3)$ and $x_2 = (4,5,6)$, then there is no $c \in A$ such that $G = \la x_1,y^c\ra = \la x_2,y^c \ra$. By applying an outer automorphism of $A$, we see that the class of $6$-cycles in $A$ also fails to witness $u_0(A) \geqs 2$ and we conclude that $u(G) \leqs u_0(G) \leqs u_0(A) \leqs 1$. We have now shown that $s(G) \leqs 2$ and $u(G) \leqs 1$. 

The case $k=1$ is an easy computation and it is also a special case of \cite[Theorem~2(i)]{BH}, which gives the exact spread and uniform spread of all symmetric groups (see also Remark~\ref{r:s6}). Similarly, if $k=2$ then it is straightforward to verify the bounds $s(G) \geqs 2$ and $u(G) \geqs 1$ by direct computation, which gives $s(G) =2$ and $u(G)=1$. For the remainder of the proof, let us assume $k \geq 3$.

To show that $s(G) \geqs 2$, which gives $s(G) = 2$, the argument is essentially identical to the general case handled above.  As before we choose $y \in A$ such that $A = \la T, y\ra$ and we write $G = \la N,x\ra$ with $x = (y,1,\ldots, 1)\s$ and $\s = (1,\ldots, k) \in {\rm Sym}_{k}$. Given elements $a,b \in G$ of prime order, the goal is to show that there exists $c \in G$ such that $G = \la a^c, x \ra = \la b^c,x \ra$. Since $s(A)=2$ and $u(A)=0$, the difference here is that we choose $y$ (and hence $x$) according to the choice of $a$ and $b$, rather than picking it uniformly 
as we did before.  

To complete the proof, it remains to show that $u(G) \geqs 1$ for $k \geqs 3$. To do this, write $A = \la T,y\ra$ and $G = \la N, x\ra$, where $x = (y,1, \ldots, 1)\s$ and $\s = (1,\ldots, k) \in {\rm Sym}_k$. We will show that if $a \in G$ has prime order, then there exists $c \in G$ such that $G = \la a^c,x\ra$.

First assume $k$ is a prime and set $y = (1,2,3)(4,5) \in A$. One checks that if $z \in A$ is nontrivial and not a transposition, then $A = \la y^c, z \ra$ for some $c \in A$. Therefore, we can proceed as in the proof of Theorem \ref{t:red_odd}, unless $a = (a_1, \ldots, a_k) \in \Aut(T)^k$ and each $a_i$ is a  transposition. By conjugating by an element of $N = T^k$, we can assume that $A= \la a_1,a_2,y \ra$ and $a_3 = y^3 = (4,5)$. In addition, we may assume that the projections of $\la a,x^k\ra$ on to $T_1$ and $T_2$ are $H_1 = {\rm Alt}_5$ (intransitive) and $H_2 = C_3 \times C_3$, respectively. Since ${\rm Alt}_6 = \la H_1, H_2 \ra$ and $x$ acts transitively on $\{T_1, \ldots, T_k\}$, it follows that $\la a,x \ra \cap N$ is a subdirect product of $N$ and so either $G = \la a,x \ra$, or $\la a,x\ra$ normalises a diagonal subgroup $D_{\varphi}$ of $N$.
If $\la a, x \ra$ normalises $D_{\varphi}$, then $y^{\varphi}=y$ and $a_i = a_1^{\varphi^{i-1}}$ for $i = 2, \ldots, k$, so 
\[
A = \la y^{\varphi}, a_1^{\varphi},a_2^{\varphi}\ra = \la y,a_2,y^3\ra = \la y,a_2\ra.
\]
But this is a contradiction since $A \ne \la y, z\ra$ for all transpositions $z \in A$. The result follows.
  
Finally, if $k \geqs 4$ is composite, then a suitably modified version of the induction proof for Theorem \ref{t:red_odd} goes through (but the argument here is easier since we only need to deal with a single element rather than a pair).
\end{proof}

\begin{rem}\label{r:s6}
For completeness, let us present a direct argument to show that $u({\rm Sym}_6) = 0$. Let $G = {\rm Sym}_n$, where $n \geq 6$ is even. Suppose that $u(G) > 0$ is witnessed by the class $x^G$. Since a conjugate of $x$ generates with $(1,2,3)$, $x$ must be odd. Similarly, since a conjugate of $x$ generates with $(1,2)$, we see that $x$ must have at most two cycles. Since $n$ is even, it follows that $x$ is a $n$-cycle. However, if $n=6$ and $\p \in \Aut(G) \setminus G$, then $x^\p \in (1,2,3)(4,5)^G$ also witnesses $u(G) > 0$, which is a contradiction.
\end{rem}

We close this section by establishing, subject to proving Theorem~\ref{t:uspread}, that there are infinitely many groups with spread two that are not almost simple.

\begin{cor} \label{c:spread2} 
Let $G = T \wr C_k$ where $k \geqs 1$ and $T$ is ${\rm Alt}_5$, ${\rm Alt}_6$, $\O^+_8(2)$ or ${\rm Sp}_{2r}(2)$ with $r \geq 3$. Then $s(G) = u(G) = 2$.  
\end{cor} 

\begin{proof}
As noted in \eqref{e:u2}, $s(T) = u(T) = 2$ by \cite{BGK}, so Lemma~\ref{l:upperbound} implies that $u(G) \leq s(G) \leq 2$. Combining this with Theorem~\ref{t:uspread}, we see that $u(G) \geq 2$ and hence $s(G) = u(G) = 2$.
\end{proof}

In view of the main results in this section, we have now reduced the proofs of Theorems~\ref{t:spread} and~\ref{t:uspread} to the proof of Theorem~\ref{t:exceptional}. Strictly speaking, we need the slightly stronger conclusion given in Theorem~\ref{t:almostsimple1}, but this will follow easily from our proof. Therefore, for the remainder of the paper, our goal is to prove Theorem~\ref{t:exceptional}. We begin by recording some preliminary results for exceptional groups of Lie type.

\section{Preliminaries on exceptional groups} \label{s:prelims}

In this section, we collect together some general results on almost simple exceptional groups of Lie type that will be crucial to our proof of Theorem~\ref{t:exceptional}. In addition, we will introduce the probabilistic approach for bounding the uniform spread of a finite group, which is at the heart of our proof, and we will discuss the relevant notation and set up for applying Shintani descent in this context.

For this discussion, it will be convenient to partition the finite simple exceptional groups  over $\mathbb{F}_q$ into two collections:
\begin{align*}
\mathcal{A} &= \{ {}^2B_2(q), \, {}^2G_2(q)', \, {}^2F_4(q)', \, G_2(q)' \} \\
\mathcal{B} &= \{ E_8(q), \, E_7(q), \, E_6^{\e}(q), \, F_4(q), \, {}^3D_4(q) \}.
\end{align*}
The proof of Theorem~\ref{t:exceptional} for the low rank groups with socle in $\mathcal{A}$ will be given in Section~\ref{s:low} and the remaining groups whose socle is in $\mathcal{B}$ will be handled in Sections~\ref{s:e8}--\ref{s:3d4}.

\begin{rem}\label{r:nota}
In this paper, we always use expressions such as $E_7(q)$ and ${}^2E_6(q)$ to denote the corresponding \emph{simple} groups. 
\end{rem}

\subsection{Subgroup structure}\label{ss:p_subgroups}

Let $G$ be a finite almost simple exceptional group of Lie type over $\mathbb{F}_q$ with socle $G_0$. Write $q=p^f$ with $p$ prime. Let $\mathcal{M}$ be the set of maximal subgroups $H$ of $G$ with $G = HG_0$.

First assume $G_0 \in \mathcal{A} \cup \{ {}^3D_4(q) \}$. In each of these cases, the maximal subgroups of $G$ have been determined up to conjugacy. For $G_0 = {}^2F_4(q)'$ this is due to Malle \cite{Malle} and in the other cases we refer the reader to the relevant table in \cite[Chapter 8]{BHR} for a convenient list of the subgroups that arise. These tables reproduce the original results of Suzuki \cite{Suzuki} for ${}^2B_2(q)$, Cooperstein \cite{Coop} for $G_2(q)'$ ($q$ even) and Kleidman \cite{Kleidman3D4,KleidmanG2} for $G_2(q)$ ($q$ odd), ${}^2G_2(q)'$ and ${}^3D_4(q)$. We will make extensive use of this work in the proof of Theorem \ref{t:exceptional}.

For the remainder of Section~\ref{ss:p_subgroups}, we will assume $G_0 \in \mathcal{B}'$, where
\[
\mathcal{B}' = \{ E_8(q), \, E_7(q), \, E_6^{\e}(q), \, F_4(q) \}.
\]
Here we only have a complete description of the maximal subgroups of $G$ up to conjugacy when $G_0$ is one of
\[
E_7(2), \; E_6(2), \; {}^2E_6(2), \; F_4(2)
\]
(see \cite{BBR}, \cite{KW}, \cite{ATLAS, Wilson2E62} and \cite{NW}, respectively). However, as described below, we are able to appeal to some powerful reduction theorems to obtain a very useful description of the maximal subgroups in the general cases.

Write $G_0 = (\bar{G}_{\s})'$, where $\bar{G}$ is a simple algebraic group of adjoint type over the algebraic closure of $\mathbb{F}_p$ and $\s$ is an appropriate Steinberg endomorphism of $\bar{G}$. The subgroups in $\mathcal{M}$ fall into several families according to the following fundamental theorem (see \cite[Theorem~2]{LS90}). 

\begin{thm}\label{t:types}
Let $G$ be an almost simple group with socle $G_0 = (\bar{G}_{\s})' \in \mathcal{B}'$ and let $H \in \mathcal{M}$. Then one of the following holds:
\begin{itemize}\addtolength{\itemsep}{0.2\baselineskip}
\item[{\rm (I)}]   $H = N_G(\bar{H}_{\s})$ for a maximal closed $\sigma$-stable positive dimensional subgroup $\bar{H}$ of $\bar{G}$;
\item[{\rm (II)}]  $H$ is of the same type as $G$ (possibly twisted) over a subfield of $\mathbb{F}_{q}$;
\item[{\rm (III)}] $H$ is an exotic local subgroup (see \cite{CLSS});
\item[{\rm (IV)}]  $G_0=E_8(q)$, $p \geqs 7$ and $H \cap G_0 = ({\rm Alt}_5 \times {\rm Alt}_6).2^2$;
\item[{\rm (V)}]   $H$ is almost simple and not of type (I) or (II).
\end{itemize}
\end{thm}

In view of Theorem~\ref{t:types}, it will be convenient to write 
\begin{equation}\label{e:mpart}
\mathcal{M} = \mathcal{M}_1 \cup \mathcal{M}_2 \cup \mathcal{M}_3
\end{equation}
where $\mathcal{M}_1$ comprises the maximal subgroups of type (I)--(IV) and $\mathcal{M}_2 \cup \mathcal{M}_3$ is the remaining collection of almost simple subgroups of type (V). Specifically, if $H$ is a type (V) subgroup with socle $S$ and ${\rm Lie}(p)$ denotes the set of finite simple groups of Lie type over a field of characteristic $p$, then we write $H \in \mathcal{M}_2$ if $S \in {\rm Lie}(p)$ and $H \in \mathcal{M}_3$ otherwise. 

Through the work of numerous authors, the subgroups comprising $\mathcal{M}_1$ are well understood and they have been determined up to conjugacy. However, there is no equivalent result for the subgroups in $\mathcal{M}_2 \cup \mathcal{M}_3$, although there has been some substantial progress. In particular, there is a short list of possibilities for $S$ up to isomorphism (see Theorem~\ref{t:simples} below for $\mathcal{M}_2$ and \cite{LS99, Litt} for $\mathcal{M}_3$), but the conjugacy problem remains open in general. Extensive ongoing work of Craven \cite{Craven17, CravenPSL, CravenMedium} seeks to significantly shorten the list of candidate subgroups in $\mathcal{M}_2 \cup \mathcal{M}_3$, with the ultimate goal of a complete classification.

The remainder of this section is dedicated to deducing the information we need on the subgroups in $\mathcal{M}$ for the proof of Theorem~\ref{t:exceptional}. We begin by studying the conjugacy classes of subgroups in $\mathcal{M}_1$. Recall our convention that logarithms are base two.

\begin{prop}\label{p:classes}
The number of $\bar{G}_{\s}$-classes of subgroups in $\mathcal{M}_1$ is at most $a(G_0) + \log\log q$, where 
\[
\begin{array}{ccccc} 
\hline
G_0    & F_4(q) & E_6^{\e}(q) & E_7(q) & E_8(q) \\ 
a(G_0) & 25     & 25          & 30     & 49     \\ 
\hline
\end{array}
\]
\end{prop}

\begin{proof}
The argument is similar in each case and we just give details for $G_0 = E_8(q)$. First consider the subgroups of type (I) in Theorem~\ref{t:types}, so $H = N_G(\bar{H}_{\s})$. Clearly, there are $8$ classes of maximal parabolic subgroups (one for each node in the Dynkin diagram) and by inspecting \cite{LSS} we find that there are at most $29$ additional classes of maximal subgroups of type (I) with $\bar{H}$ of maximal rank. The remaining possibilities for $H$ are listed in \cite[Table~3]{LS03}, together with the case recorded in \cite[Theorem~8(I)(d)]{LS03}; this gives at most $9$ further classes. Altogether, this demonstrates that there are at most $46$ classes of maximal subgroups in $\mathcal{M}_1$ of type (I). The subgroups of type (II) are subfield subgroups; there is a unique class for each maximal subfield of $\mathbb{F}_q$ and there are at most $\log\log q$ such subfields (this is an upper bound on the number of prime divisors of $f$, where $q=p^f$). By the main theorem of \cite{CLSS}, there are at most $2$ classes of subgroups of type (III) and there is at most $1$ additional class of type (IV). By bringing the above estimates together, we conclude that there are at most $49 + \log\log q$ distinct $\bar{G}_{\s}$-classes of subgroups in $\mathcal{M}_1$.
\end{proof}

Now assume that $H \in \mathcal{M}_2 \cup \mathcal{M}_3$ and let $S$ be the socle of $H$. Note that $S$ is a subgroup of $G_0$. The following result significantly restricts the subgroups in $\mathcal{M}_2$ (see \cite[Theorem~8]{LS03}, noting that the value of $b(E_8(q))$ in part~(iii) is taken from \cite{Lawther}). In the statement, if $X$ is a simple group of Lie type, then ${\rm rk}(X)$ denotes the untwisted Lie rank of $X$ (that is, ${\rm rk}(X)$ is the rank of the ambient simple algebraic group).

\begin{thm}\label{t:simples}
Suppose that $H \in \mathcal{M}_2$ has socle $S$, a simple group of Lie type over $\mathbb{F}_t$, where $t$ is a power of $p$. Then ${\rm rk}(S) \leqs \frac{1}{2}{\rm rk}(G_0)$ and one of the following holds: 
\begin{itemize}\addtolength{\itemsep}{0.2\baselineskip}
\item[{\rm (i)}]   $t \leqs 9$;
\item[{\rm (ii)}]  $S = {\rm L}_{3}^{\e}(16)$;
\item[{\rm (iii)}] $S \in \{ {\rm L}_{2}(t), {}^2B_2(t), {}^2G_2(t)\}$, where $t \leqs (2,q-1)\,b(G_0)$ and  
\[
\begin{array}{cccccc} 
\hline
G_0    & F_4(q) & E_6^{\e}(q) & E_7(q) & E_8(q) \\
b(G_0) & 68     & 124         & 388    & 1312   \\ 
\hline
\end{array}
\]
\end{itemize}
\end{thm}

It remains to discuss the situation where $H \in \mathcal{M}_3$. In this case, the possibilities for $S$ (up to isomorphism) are described in \cite{LS99} (see \cite[Tables~10.1--10.4]{LS99}) and we note that substantial refinements are established in \cite{Craven17, CravenPSL, CravenMedium, Litt}. For instance, the main theorem of \cite{Craven17} states that if $S = {\rm Alt}_n$ then $n=6$ and $n=7$ are the only options, whereas \cite{LS99} gives $n \leqs 18$.

We conclude this section by studying the maximal order of an element in a subgroup contained in $\mathcal{M}_2 \cup \mathcal{M}_3$. Given a subset $X$ of a finite group, set 
\[
{\rm meo}(X) = \max\{|x| \,:\, x \in X\}.
\]
The following result gives an upper bound on ${\rm meo}(H)$, where $H$ is almost simple and either classical or a low rank exceptional group.

\begin{prop}\label{p:meo}
Let $S$ be a finite simple group of Lie type over $\mathbb{F}_t$ with ${\rm rk}(S)=m$.
\begin{itemize}\addtolength{\itemsep}{0.2\baselineskip}
\item[{\rm (i)}] If $S$ is a classical group, then either 
\[
{\rm meo}({\rm Aut}(S)) \leqs \frac{t^{m+1}}{t-1},
\]
or $S = {\rm PSp}_{4}(2)'$ and ${\rm meo}({\rm Aut}(S)) = 10$.
\item[{\rm (ii)}] If $S$ is an exceptional group with $m \leqs 4$, then ${\rm meo}({\rm Aut}(S)) \leqs c(S)$, where $c(S)$ is given in Table~\ref{tab:meo}.
\end{itemize}
\end{prop}

\begin{proof}
Part~(i) is an immediate corollary of \cite[Theorem~2.16]{GMPS}. For~(ii), we argue as in the proof of \cite[Theorem~1.2]{GMPS} (see \cite[p.7683]{GMPS}). If $t$ is odd, then ${\rm meo}(S)$ is given in \cite[Table~A.7]{KS} and the result follows from the trivial bound
\begin{equation}\label{e:meo}
{\rm meo}({\rm Aut}(S)) \leqs |{\rm Out}(S)|\, {\rm meo}(S).
\end{equation}
Now assume $t$ is even. For $S = {}^2B_2(t)$ with $t = 2^{2k+1}>2$, we have ${\rm meo}(S) = 2^{2k+1}+2^{k+1}+1$ (see \cite[Proposition~16]{Suzuki}) and the bound in Table~\ref{tab:meo} follows via \eqref{e:meo}. In the remaining cases, we use 
\[
{\rm meo}({\rm Aut}(S)) \leqs \a\b|{\rm Out}(S)|,
\]
where $\a$ and $\b$ are upper bounds on the maximal orders of semisimple and unipotent elements in $S$, respectively (see the proof of \cite[Theorem~1.2]{GMPS}). Expressions for $\a$ and $\b$ are given in \cite[Table~5]{GMPS} and the desired result follows.
\end{proof}

\begin{rem}\label{r:meo}
For $S = {\rm L}_{d}^{\e}(t)$, the precise value of ${\rm meo}({\rm Aut}(S))$ is recorded in \cite[Table~3]{GMPS}. In particular, we note that ${\rm meo}({\rm Aut}({\rm L}_{3}(16))) = 273$ and ${\rm meo}({\rm Aut}({\rm U}_{3}(16))) = 255$.
\end{rem}

\begin{table}
\caption{Bounds on ${\rm meo}({\rm Aut}(S))$, $S$ exceptional, ${\rm rk}(S) \leqs 4$} \label{tab:meo}
\vspace{-8mm}
\begin{center}
\[
\begin{array}{ll} 
\hline
S           & c(S)                 \\ 
\hline
F_4(t)      & 32(t+1)(t^3-1)\log t \\
G_2(t)      & 8(t^2+t+1)\log t     \\
{}^2F_4(t), \, t=2^{2k+1}, \, k \geqs 1 & 16(2k+1)(2^{4k+2}+2^{3k+2}+2^{2k+1}+2^{k+1}+1) \\
{}^2F_4(2)' & 20 \\
{}^3D_4(t)  & 24(t^3-1)(t+1)\log t \\
{}^2G_2(t),\, t=3^{2k+1},  \, k \geqs 1 & (2k+1)(3^{2k+1}+3^{k+1}+1) \\
{}^2G_2(3)' &  9 \\
{}^2B_2(t),\, t=2^{2k+1},  \, k \geqs 1 & (2k+1)(2^{2k+1}+2^{k+1}+1) \\ 
\hline
\end{array}
\]
\end{center}
\end{table}

For the subgroups in $\mathcal{M}_3$, we have the following result on element orders.

\begin{prop}\label{p:orders}
Suppose $H \in \mathcal{M}_3$. Then ${\rm meo}(H) \leqs d(G_0)$, where 
\[
\begin{array}{ccccc} 
\hline
G_0    & F_4(q) & E_6^{\e}(q) & E_7(q) & E_8(q) \\
d(G_0) & 40     & 60          & 63     & 210    \\ 
\hline
\end{array}
\]
\end{prop}

\begin{proof}
Let $S$ be the socle of $H$. As previously noted, the possibilities for $S$ (up to isomorphism) are recorded in \cite[Tables 10.1--10.4]{LS99} and it is straightforward to determine ${\rm meo}({\rm Aut}(S))$ in every case, either via {\sc Magma} \cite{magma} or by inspecting the {\sc Atlas} \cite{ATLAS}.
\end{proof}

We will also need the following result to handle some special cases.

\begin{prop}\label{p:sp}
If $S \in \{{\rm L}_{4}(8), \, {\rm U}_5(8),  \, {\rm PSp}_{6}(8), \, G_2(8), \, G_2(9)\}$, then 
\[
{\rm meo}({\rm Aut}(S) \setminus S) \leqs e(S),
\]
where 
\[
\begin{array}{cccccc} 
\hline
S    & {\rm L}_{4}(8) & {\rm U}_5(8) & {\rm PSp}_{6}(8) & G_2(8) & G_2(9) \\ 
e(S) & 130            & 130          & 45               & 36     & 36 \\ 
\hline
\end{array}
\]
\end{prop}

\begin{proof}
This can be verified with {\sc Magma} \cite{magma}, using \texttt{AutomorphismGroupSimpleGroup} to construct suitable permutation representations of the relevant automorphism groups.
\end{proof}

\subsection{Automorphisms}\label{ss:p_aut} 

Continue to assume that $G_0$ is a finite simple exceptional group of Lie type over $\mathbb{F}_q$ and write $q=p^f$ where $p$ is prime. In this section we determine the precise list of almost simple groups with socle $G_0$ that we need to consider in order to prove Theorem~\ref{t:exceptional}. Naturally, this will involve a careful study of the automorphisms of $G_0$ and the structure of the outer automorphism group ${\rm Out}(G_0) = \Aut(G_0)/G_0$. Our main result to this end is Proposition~\ref{p:cases}.

In this discussion, for clarity of exposition, we will assume that $G_0$ is not one of
\begin{equation}
G_2(2)' \cong {\rm U}_3(3), \; {}^2F_4(2)', \; {}^2G_2(3)' \cong {\rm L}_2(8). \label{e:excluded}
\end{equation}
(In the first two cases, ${\rm Aut}(G_0) = G_0.2$ and in the latter we have ${\rm Aut}(G_0) = G_0.3$.) Let us partition the remaining possibilities for $G_0$ into three classes:
\begin{gather}
E_8(q), \; E_7(q), \; E_6^{\e}(q), \; F_4(q) \, (p \neq 2), \; G_2(q) \, (p \neq 3), \; {}^3D_4(q) \label{e:good} \\
F_4(2^f), \; G_2(3^f)                                                                              \label{e:bad}  \\
{}^2F_4(2^{2k+1}), \; {}^2G_2(3^{2k+1}), \; {}^2B_2(2^{2k+1}).                                     \label{e:ugly} 
\end{gather}

We begin by describing $\Aut(G_0)$, where we follow \cite[Chapter~2.5]{GLS} (see \cite[Theorem~2.5.12]{GLS} in particular). Write $G_0 = (\bar{G}_\s)'$, where $\bar{G}$ is a simple algebraic group over $k=\bar{\mathbb{F}}_p$ of adjoint type and $\s$ is a Steinberg endomorphism. We refer to $\bar{G}_\s$ as the \emph{innerdiagonal group} of automorphisms of $G_0$ and we write $\bar{G}_\s = \Inndiag(G_0)$. We refer to the elements in $\Inndiag(G_0) \setminus G_0$ as \emph{diagonal automorphisms}. Then $\Aut(G_0)$ is a split extension of $\Inndiag(G_0)$ by a soluble group generated by \emph{field}, \emph{graph} and \emph{graph-field automorphisms} that are defined naturally from automorphisms of the underlying field $\mathbb{F}_q$ and symmetries of the Dynkin diagram of $\bar{G}$. 

Let us fix our notation for automorphisms of $G_0$. In part (iii) of the following definition, we write $D_4$ for the adjoint group ${\rm PSO}_8(k)$.

\begin{defn} \label{d:aut}
Let $G_0 = (\bar{G}_\s)' = {}^dX(q)$ be a finite simple exceptional group as above and let
$\p$ be a standard Frobenius endomorphism of $\bar{G}$.
\begin{itemize}\addtolength{\itemsep}{0.2\baselineskip}
\item[{\rm (i)}]  If $G_0$ is not in \eqref{e:ugly}, then we identify $\p$ with the restriction $\p|_{G_0}$. Then $\p \in \Aut(G_0)$ is a field or graph automorphism such that $|\p|=df$.
\item[{\rm (ii)}]  If $G_0$ is in \eqref{e:bad} or \eqref{e:ugly}, then let $\rho$ be the Steinberg endomorphism of $\bar{G}$ such that $\rho^2=\p$ and identify $\rho$ with the restriction $\rho|_{G_0}$. Then $\rho \in \Aut(G_0)$ is a graph-field automorphism with $|\rho|=2f/d$.
\item[{\rm (iii)}] Let $\g$ be an involutory graph automorphism of $\bar{G}=E_6$ such that $[\p,\g] = 1$ and $C_{E_6}(\g) = F_4$, and identify $\g$ with the restriction $\g|_{E_6^\e(q)}$. Similarly, let $\t$ be an order $3$ triality graph automorphism of $\bar{G}=D_4$ with $[\p,\t]=1$ and $C_{D_4}(\t) = G_2$, and identify $\t$ with the restriction $\t|_{{}^3D_4(q)}$. 
\item[{\rm (iv)}]  If $G_0=E_7(q)$ and $q$ is odd, then fix a diagonal automorphism $\d \in \Inndiag(G_0)$ of order $2$. Similarly, if $G_0=E_6^\e(q)$ and $q \equiv \e \mod{3}$, then let $\d \in \Inndiag(G_0)$ be a diagonal automorphism of order $3$. 
\end{itemize}
\end{defn}

For $g \in \Aut(G_0)$, we write $\ddot{g}$ for the coset $G_0g$, so 
\[
\Out(G_0) = \{ \ddot{g} \,:\, g \in \Aut(G_0) \}.
\]
If $G_0$ is not $E^\e_6(q)$, then the structure of $\Out(G_0)$ can be immediately deduced from \cite[Theorem~2.5.12]{GLS} and we present the details in Table~\ref{tab:out}. The structure of $\Out(E^\e_6(q))$ is given in Lemmas \ref{l:out_e6} and \ref{l:out_2e6} in the untwisted and twisted cases, respectively.  

\begin{table}
\caption{$\Out(G_0)$ for a finite simple exceptional group $G_0$} \label{tab:out}
\vspace{-8mm}
\begin{center}
\[
\begin{array}{lllll} 
\hline
G_0        &                         & \Out(G_0)                          &                & \text{Comments}                 \\
\hline
E_8(q)     &                         & \<\ddot{\p}\>                      & C_f            &                              \\[5pt]              
E_7(q)     & p \neq 2                & \<\ddot{\d}\> \times \<\ddot{\p}\> & C_2 \times C_f &                              \\
& p = 2                   & \<\ddot{\p}\>                      & C_f            &                              \\
           [5pt]
E_6(q)     & q \not\equiv 1 \mod{3}  & \<\ddot{\g}\> \times \<\ddot{\p}\> & C_2 \times C_f &                              \\
           & q     \equiv 1 \mod{3}  & \<\ddot{\d},\ddot{\g},\ddot{\p}\>  & {\rm Sym}_3 \times C_f & \text{See Lemma~\ref{l:out_e6}}  \\[5pt]
{}^2E_6(q) & q \not\equiv 2 \mod{3}  & \<\ddot{\p}\>                      & C_{2f}         & \ddot{\p}^f = \ddot{\g}      \\
           & q     \equiv 2 \mod{3}  & \<\ddot{\d}, \ddot{\p}\>           & {\rm Sym}_3 \times C_f & \text{See Lemma~\ref{l:out_2e6}} \\[5pt]     
F_4(q)     & p \neq 2                & \<\ddot{\p}\>                      & C_f            &                              \\
           & p = 2                   & \<\ddot{\r}\>                      & C_{2f}         & \ddot{\r}^2 = \ddot{\p}      \\[5pt]
G_2(q)     & p \neq 3, q>2           & \<\ddot{\p}\>                      & C_f            &                              \\
           & p = 3                   & \<\ddot{\r}\>                      & C_{2f}         & \ddot{\r}^2 = \ddot{\p}      \\
           [5pt]
{}^3D_4(q) &                         & \<\ddot{\p}\>                      & C_{3f}         & \ddot{\p}^f = \ddot{\t}      \\
{}^2F_4(q)  &  q>2                   & \<\ddot{\r}\>                      & C_f            &                              \\
{}^2G_2(q) & q>3                     & \<\ddot{\r}\>                      & C_f            &                              \\
{}^2B_2(q) &                         & \<\ddot{\r}\>                      & C_f            &                              \\ \hline
\end{array}
\]
\end{center}
\end{table}

\begin{lem} \label{l:out_e6}
Let $G_0 = E_6(q)$. Then
\[
\Out(G_0) = \left\{ 
\begin{array}{ll}
\< \ddot{\g} \> \times \< \ddot{\p} \> \cong C_2         \times C_f & \text{if $q \not\equiv 1 \mod{3}$} \\
\< \ddot{\d}, \ddot{\g}, \ddot{\p} \>  \cong {\rm Sym}_3 \times C_f & \text{if $q     \equiv 1 \mod{3}$}.
\end{array}
\right.
\]
\end{lem}

\begin{proof}
According to \cite[Theorem~2.5.12(a)]{GLS}, we have $\Aut(G_0) = \Inndiag(G_0){:}\<\gamma,\p\>$. In particular, if $q \not\equiv 1 \mod{3}$ then 
\[
\Out(G_0) = \< \ddot{\g}, \ddot{\p} \> = \<\ddot{\g} \> \times \< \ddot{\p} \> \cong C_2 \times C_f
\] 
as claimed. 

For the remainder, we may assume $q \equiv 1 \mod{3}$. Here
\begin{equation} \label{e:out_e6_facts}
\mbox{$\Out(G_0) = \< \ddot{\d}, \ddot{\g}, \ddot{\p} \>$ and $|\ddot{\d}|=3$, $|\ddot{\g}| = 2$, $|\ddot{\p}| = f$, $[\ddot{\g},\ddot{\p}] = 1$, $\ddot{\d}^{\ddot{\g}} = \ddot{\d}^{-1}$,  $\ddot{\d}^{\ddot{\p}} = \ddot{\d}^p$}
\end{equation}
(for the final two claims, see \cite[Theorem~2.5.12(i)]{GLS} and \cite[Theorem~2.5.12(g)]{GLS}, respectively). If $p \equiv 1 \mod{3}$, then $[\ddot{\d},\ddot{\p}]=1$ and thus
\[
\Out(G_0) = \< \ddot{\d}, \ddot{\g} \> \times \< \ddot{\p} \> \cong {\rm Sym}_3 \times C_f.
\]
Now assume that $p \equiv 2 \mod{3}$. Here the condition $q \equiv 1 \mod{3}$ implies that $f$ is even, so $|\ddot{\g}\ddot{\p}|=f$. In addition, $[\ddot{\g},\ddot{\g}\ddot{\p}] = 1$ and $[\ddot{\d},\ddot{\g}\ddot{\p}]=1$, where the latter claim holds since $\ddot{\d}^{\ddot{\g}\ddot{\p}} = (\ddot{\d}^{-1})^{\ddot{\p}} = \ddot{\d}$. Therefore,
\[
\Out(G_0) = \< \ddot{\d}, \ddot{\g} \> \times \< \ddot{\g}\ddot{\p} \> \cong {\rm Sym}_3 \times C_f. \qedhere
\]
\end{proof}

For future reference, it will be convenient to record the following set of conditions:
\begin{equation}\label{e:cases_e6_condition}
\mbox{$p \equiv 2 \mod{3}$, $f$ is even and $i$ is odd.}
\end{equation}

\begin{lem} \label{l:out_e6_facts}
Let $G_0 = E_6(q)$ with $q \equiv 1 \mod{3}$ and fix an integer $0 \leqs i < f$. Then the following hold:
\begin{itemize}\addtolength{\itemsep}{0.2\baselineskip}
\item[{\rm (i)}] $\ddot{\d}\ddot{\p}^i$ and $\ddot{\d}^2\ddot{\p}^i$ are $\Out(G_0)$-conjugate.
\item[{\rm (ii)}] $\ddot{\d}\ddot{\g}\ddot{\p}^i$ and $\ddot{\d}^2\ddot{\g}\ddot{\p}^i$ are $\Out(G_0)$-conjugate.
\item[{\rm (iii)}] $\ddot{\p}^i$ and $\ddot{\d}\ddot{\p}^i$ are $\Out(G_0)$-conjugate if each condition in \eqref{e:cases_e6_condition} holds.
\item[{\rm (iv)}] $\ddot{\g}\ddot{\p}^i$ and $\ddot{\d}\ddot{\g}\ddot{\p}^i$ are $\Out(G_0)$-conjugate if any of the conditions in \eqref{e:cases_e6_condition} fail to hold.
\end{itemize}
\end{lem}

\begin{proof}
Let $A = \< \ddot{\d}, \ddot{\g} \> \cong {\rm Sym}_3$ and note that the conjugacy classes of $A$ are as follows:
\[
\{ \ddot{1} \}, \; \{ \ddot{\d},  \ddot{\d}^2 \}, \; \{ \ddot{\g},  \ddot{\d}\ddot{\g},  \ddot{\d}^2\ddot{\g} \}.
\]
If any one of the conditions in \eqref{e:cases_e6_condition} is not satisfied, then $\ddot{\p}^i \in Z(\Out(G_0))$ and (i), (ii) and (iv) follow. On the other hand, if all the conditions in \eqref{e:cases_e6_condition} are satisfied, then $\ddot{\g}\ddot{\p}^i \in Z(\Out(G_0))$ and by writing 
\begin{gather*}
\ddot{\d}\ddot{\p}^i = \ddot{\d}\ddot{\g}(\ddot{\g}\ddot{\p}^i) \mbox{ and } \ddot{\d}^2\ddot{\p}^i = \ddot{\d}^2\ddot{\g}(\ddot{\g}\ddot{\p}^i) \\
\ddot{\d}\ddot{\g}\ddot{\p}^i = \ddot{\d}(\ddot{\g}\ddot{\p}^i) \mbox{ and } \ddot{\d}^2\ddot{\g}\ddot{\p}^i = \ddot{\d}^2(\ddot{\g}\ddot{\p}^i) \\
\ddot{\p}^i = \ddot{\g}(\ddot{\g}\ddot{\p}^i) \mbox{ and } \ddot{\d}\ddot{\p}^i = \ddot{\d}\ddot{\g}(\ddot{\g}\ddot{\p}^i)
\end{gather*}
we deduce that (i), (ii) and (iii) hold. 
\end{proof}

We now turn to the twisted version of $E_6$. 

\begin{lem} \label{l:out_2e6}
Let $G_0 = {}^2E_6(q)$. Then
\[
\Out(G_0) = \left\{ 
\begin{array}{ll}
\< \ddot{\p} \>           \cong C_{2f}                 & \text{if $q \not\equiv 2 \mod{3}$} \\
\< \ddot{\d},\ddot{\p} \> \cong {\rm Sym}_3 \times C_f & \text{if $q     \equiv 2 \mod{3}$.}
\end{array}
\right.
\]
\end{lem}

\begin{proof}
By \cite[Theorem~2.5.12(a)]{GLS}, we have $\Aut(G_0) = \Inndiag(G_0){:}\<\p\>$. Therefore, if $q \not\equiv 2 \mod{3}$, then $\Out(G_0) = \<\ddot{\p}\> \cong C_{2f}$. For the remainder, let us assume $q \equiv 2 \mod{3}$. Here $p \equiv 2 \mod{3}$, $f$ is odd and 
\begin{equation} \label{e:out_2e6_facts}
\mbox{$\Out(G_0) = \< \ddot{\d}, \ddot{\p} \>$ and $|\ddot{\d}|=3$, $|\ddot{\p}| = 2f$, $\ddot{\d}^{\ddot{\p}} = \ddot{\d}^{-1}$}
\end{equation}
(see \cite[Theorem~2.5.12(g)]{GLS} for the final claim). Since $\<\p\> = \< \p^f \> \times \< \p^2\>$, we obtain
\[
\Out(G_0) = \< \ddot{\d}, \ddot{\p}^f \> \times \< \ddot{\p}^2\> \cong {\rm Sym}_3 \times C_f. \qedhere
\]
\end{proof}

\begin{lem} \label{l:out_2e6_facts}
Let $G_0 = {}^2E_6(q)$ with $q \equiv 2 \mod{3}$ and fix an integer $0 \leqs i < 2f$. Then the following hold:
\begin{itemize}\addtolength{\itemsep}{0.2\baselineskip}
\item[{\rm (i)}] $\ddot{\d}\ddot{\p}^i$ and $\ddot{\d}^2\ddot{\p}^i$ are $\Out(G_0)$-conjugate.
\item[{\rm (ii)}] If $i$ is odd, then $\ddot{\p}^i$ and $\ddot{\d}\ddot{\p}^i$ are $\Out(G_0)$-conjugate.
\end{itemize}
\end{lem}

\begin{proof}
By \eqref{e:out_2e6_facts}, we have $(\ddot{\d}\ddot{\p}^i)^{\ddot{\p}} = \ddot{\d}^2\ddot{\p}^i$. Moreover, if $i$ is odd then 
\[
(\ddot{\p}^i)^{\ddot{\d}} = \ddot{\d}^{-1}\ddot{\p}^i\ddot{\d} = \ddot{\d}^{-1}\ddot{\d}^{\ddot{\p}^{-i}}\ddot{\p}^i = \ddot{\d}^{-1}\ddot{\d}^{-1}\ddot{\p}^i = \ddot{\d}\ddot{\p}^i
\]
and the result follows.
\end{proof}

The following elementary lemma will be useful in the proof of Proposition \ref{p:cases} (for a proof, see \cite[Lemma 5.2.1]{HarperClassical}).

\begin{lem} \label{l:divides}
Let $\<a\>{:}\<b\>$ be a semidirect product of finite cyclic groups. For all $i > 0$, there exist nonnegative integers $j$ and $k$ such that $\<ab^i\> = \<a^jb^k\>$ and $k$ divides $|b|$.
\end{lem}

We now use the above information on $\Out(G_0)$ to determine the specific groups we need to consider in order to prove Theorem~\ref{t:exceptional}. Note that in  Table~\ref{tab:cases_e6}, $i$ is a proper divisor of $f$ and the symbols $\star$ and $\dagger$ refer to notes presented in Remark~\ref{r:cases_e6}.

\begin{prop} \label{p:cases}
Let $G_0$ be a finite simple exceptional group over $\mathbb{F}_q$, where $q=p^f$ with $p$ prime. Assume $G_0$ is not one of the groups in \eqref{e:excluded} and let $h$ be a non-inner automorphism of $G_0$. Then $\<G_0,h\>$ is $\Aut(G_0)$-conjugate to $\<G_0,g\>$, where $g \in \Aut(G_0)$ is one of the following: 
\begin{itemize}\addtolength{\itemsep}{0.2\baselineskip}
\item[{\rm (i)}]   $G_0$ is in \eqref{e:ugly} and $g=\r^i$ for a proper divisor $i$ of $f$.
\item[{\rm (ii)}]  $G_0$ is in \eqref{e:bad} and either \vspace{2pt}
\begin{itemize}\addtolength{\itemsep}{0.2\baselineskip}
\item[{\rm (a)}] $g = \p^i$ for a proper divisor $i$ of $f$; or
\item[{\rm (b)}] $g = \r^i$ for an odd divisor $i$ of $f$.
\end{itemize}
\item[{\rm (iii)}] $G_0$ is in \eqref{e:good}, $G_0 \neq E^\e_6(q)$, and either \vspace{2pt}
\begin{itemize}\addtolength{\itemsep}{0.2\baselineskip}
\item[{\rm (a)}] $g = \p^i$ for a proper divisor $i$ of $f$; 
\item[{\rm (b)}] $G_0 = {}^3D_4(q)$ and $g = \t\p^i$ for a divisor $i$ of $f$; or
\item[{\rm (c)}] $G_0 = E_7(q)$ with $q$ odd and $g$ is $\d$ or $\d\p^i$ for a proper divisor $i$ of $f$.
\end{itemize}
\item[{\rm (iv)}] $G_0 = E_6^\e(q)$ and either \vspace{2pt}
\begin{itemize}\addtolength{\itemsep}{0.2\baselineskip}
\item[{\rm (a)}] $g$ is in Row~{\rm (R1)} of Table~\ref{tab:cases_e6}; or
\item[{\rm (b)}] $q \equiv \e \mod{3}$ and $g$ is in Row~{\rm (R2)} of Table~\ref{tab:cases_e6}.
\end{itemize}
\end{itemize}
\end{prop}

\begin{table}
\caption{The automorphisms of $G_0=E_6^\e(q)$ in Proposition \ref{p:cases}(iv)}\label{tab:cases_e6}
\begin{center}
\vspace{-8mm}
\[
\begin{array}{cccccccc}
\hline 
\e           & \pm        & \pm        & +           & +           & +           & -           & -           \\
\hline
g &
\begin{array}{c} \      \\ \d           \end{array} & 
\begin{array}{c} \g     \\ \            \end{array} &
\begin{array}{c} \p^i   \\ \d^\pm\p^i   \end{array} &
\begin{array}{c} \g\p^i \\ \d^\pm\g\p^i \end{array} &
\begin{array}{c} \g\p^i \\ \            \end{array} &
\begin{array}{c} \g\p^i \\ \d^\pm\g\p^i \end{array} &
\begin{array}{c} \p^i   \\ \            \end{array} \\
\hline 
f/i          &            &            & \text{any}  & \text{even} & \text{odd}  & \text{odd}  & \text{any}  \\
\text{notes} &            &            & \star       & \dagger     &             &             &             \\
\hline
\end{array}
{\arrayrulecolor{white}
\begin{array}{c}
\hline
    \\
\hline
\text{(R1)} \\
\text{(R2)} \\
\hline
    \\
    \\
\hline
\end{array}}
\]
\end{center}
\end{table}
\arrayrulecolor{black}

\begin{rem}\label{r:cases_e6}
In Table~\ref{tab:cases_e6}, the symbol $\d^\pm$ denotes that we may consider either $\d$ or $\d^{-1}$ (but there is no need to consider both). The notes labelled $\star$ and $\dagger$  impose further restrictions on the automorphisms we need to consider:
\begin{itemize}\addtolength{\itemsep}{0.2\baselineskip}
\item[$\star$] We need only consider one of the automorphisms in $\{ \p^i, \d\p^i, \d^2\p^i\}$ in the very special case when all the conditions in \eqref{e:cases_e6_condition} are satisfied.
\item[$\dagger$] We need only consider one automorphism in $\{ \g\p^i, \d\g\p^i, \d^2\g\p^i \}$ \emph{unless} all the conditions in \eqref{e:cases_e6_condition} hold.
\end{itemize}
\end{rem}

\begin{proof}[Proof of Proposition~\ref{p:cases}]
Since $\<G_0,g\>$ and $\<G_0,h\>$ are $\Aut(G_0)$-conjugate if and only if $\<\ddot{g}\>$ and $\<\ddot{h}\>$ are $\Out(G_0)$-conjugate, we must determine the conjugacy classes of cyclic subgroups of $\Out(G_0)$. Fix an automorphism $h \in \Aut(G_0) \setminus G_0$. 

If $G_0$ is in \eqref{e:ugly} or \eqref{e:bad}, then Table~\ref{tab:out} indicates that $G_0$ has a graph-field automorphism $\r$ such that $\Out(G_0) = \< \ddot{\r} \>$. Moreover, if $G_0$ is in \eqref{e:ugly}, then $|\ddot{\r}|=f$, so $\< \ddot{h} \> = \<\ddot{\r}^i\>$ for some proper divisor $i$ of $f$, as we claim. Similarly, if $G_0$ is in \eqref{e:bad}, then $|\ddot{\r}|=2f$, so $\< \ddot{h} \> = \<\ddot{\r}^i\>$ for some proper divisor $i$ of $2f$. In particular, $\<\ddot{h}\>$ is either equal to $\<\ddot{\r}^i\>$ for some odd divisor $i$ of $f$ (as in (ii)(b)), or $\<\ddot{\r}^{2i}\> = \<\ddot{\p}^i\>$ for some proper divisor $i$ of $f$ (as in (ii)(a)).

Next assume $G_0$ is in \eqref{e:good} with $G_0 \neq E_6^\e(q)$. First assume that $\Out(G_0) = \< \ddot{\p} \>$, so $\<\ddot{h}\> = \<\ddot{\p}^i\>$ for some proper divisor $i$ of $|\p|$. If $G_0 \neq {}^3D_4(q)$, then $|\p|=f$ and we are in case (iii)(a). Now suppose  $G_0 = {}^3D_4(q)$, so $\<\ddot{h}\> = \<\ddot{\p}^i\>$ for some divisor of $i$ of $|\p|=3f$. If $3$ divides $3f/i$, then $i$ divides $f$ and we are in (iii)(a) once again. Otherwise, $3$ divides $i$ and $f/j$ is not divisible by $3$, where $j=i/3$. Here $3f/(3f,f+j) = 3f/(3f,j)$ and 
\[
\<\ddot{h}\> = \<\ddot{\p}^i\> = \<\ddot{\p}^{f+j}\> = \<\ddot{\t}\ddot{\p}^j\>,
\]
which puts us in case (iii)(b). Finally, if $\Out(G_0) \neq \<\ddot{\p}\>$ then $G_0 = E_7(q)$ is the only option (see Table~\ref{tab:out}), where $q$ is odd and $\Out(G_0) = \<\ddot{\d}\> \times \<\ddot{\p}\>$. Here Lemma~\ref{l:divides} implies that $\<\ddot{h}\> = \<\ddot{\p}^i\>$ or $\<\ddot{\d}\ddot{\p}^i\>$ for some divisor $i$ of $f$, and these possibilities are covered by cases (iii)(a) and (iii)(c), respectively.

To complete the proof, we may assume that $G_0 = E_6^\e(q)$. First we handle the case  $\e=+$. Here $\<\ddot{h}\> = \<\ddot{h}_0\p^i\>$, where $h_0$ is a product of diagonal and graph automorphisms, and by Lemma~\ref{l:divides} we may assume that $i=0$ or $i$ divides $f$. If $q \not\equiv 1 \mod{3}$, then $h_0 \in \{1,\g\}$, so $\<\ddot{h}\> =  \<\ddot{g}\>$ for an automorphism $g$ in Row (R1) of Table~\ref{tab:cases_e6}. Now assume $q \equiv 1 \mod{3}$. Here $h_0 = \d^j\g^k$ with  $j \in \{0,1,2\}$ and $k \in \{0,1\}$; we claim that $\<\ddot{h}\>$ is $\Out(G_0)$-conjugate to $\<\ddot{g}\>$ for an automorphism $g$ in Table~\ref{tab:cases_e6}. To see this, first observe that $\ddot{\d}\ddot{\p}^i$ and $\ddot{\d}^2\ddot{\p}^i$ are $\Out(G_0)$-conjugate and so are $\ddot{\d}\ddot{\g}\ddot{\p}^i$ and $\ddot{\d}^2\ddot{\g}\ddot{\p}^i$ (see parts (i) and (ii) in Lemma~\ref{l:out_e6_facts}). Therefore, it remains to prove the claim when $h \in \{ \d\g\p^i, \d^2\g\p^i \}$ and $i=0$ or $f/i$ is odd, together with the additional claims in $\star$ and $\dagger$ (see Remark \ref{r:cases_e6}). If $i=0$ or $f/i$ is odd, then \eqref{e:cases_e6_condition} does not hold, so Lemma~\ref{l:out_e6_facts}(iv) implies that $\ddot{h}$ is $\Out(G_0)$-conjugate to $\ddot{\g}$. In addition, the claims in $\star$ and $\dagger$ follow immediately from parts (iv) and (iii) in Lemma~\ref{l:out_e6_facts}, respectively.

Finally, let us assume $G_0 = {}^2E_6(q)$. Here $\<\ddot{g}\>$ is $\Out(G_0)$-conjugate to $\< \ddot{h}\ddot{\p}^i\>$ where $h$ is trivial or diagonal, and $i$ is either $0$ or a divisor of $2f$. If $i > 0$ and $2f/i$ is even, then $i$ divides $f$. On the other hand, if $i > 0$ and $2f/i$ is odd, then $f/j$ is odd for $j=i/2$ and we note that $2f/(2f,i) = 2f/(2f,f+j)$. Therefore, $\<\ddot{\g}\>$ is $\Out(G_0)$-conjugate to one of $\<\ddot{h}\>$, $\<\ddot{h}\ddot{\p}^f\> = \<\ddot{h}\ddot{\g}\>$ or $\<\ddot{h}\ddot{\p}^i\>$, where $i$ is a proper divisor of $f$, or $\<\ddot{h}\ddot{\p}^{f+j}\> = \<\ddot{h}\ddot{\g}\ddot{\p}^j\>$ and $j$ is a proper divisor of $f$ such that $f/j$ is odd. Therefore, $\<\ddot{h}\>$ is $\Out(G_0)$-conjugate to $\<\ddot{g}\>$ for an automorphism $g$ in Table~\ref{tab:cases_e6} and for the case $q \equiv 2 \mod{3}$ we conclude by appealing to Lemma \ref{l:out_2e6_facts}.
\end{proof}

\subsection{Probabilistic method}\label{ss:p_prob}

In this section, we discuss a probabilistic approach for bounding the uniform spread of a finite group, which was introduced by Guralnick and Kantor \cite{GK}. This approach plays a central role in the sequence of papers \cite{BGK,BG, GK, Harper17, HarperClassical}, and it is also a core technique in our proof of Theorem \ref{t:exceptional} in this paper. Here we recall the general set up and we introduce the relevant notation. 

Let $G$ be a finite group, let $H$ be a subgroup of $G$ and consider the natural transitive action of $G$ on the set of cosets $G/H$. In terms of this action, the \emph{fixed point ratio} of $z \in G$ is 
\[
{\rm fpr}(z,G/H) = \frac{|\{ \omega \in G/H \,:\, \omega z = \omega \}|}{|G/H|} = \frac{|z^G \cap H|}{|z^G|}.
\]
For $z,x \in G$, let $P(z,x)$ be the probability that $z$ and a uniformly randomly chosen conjugate of $x$ do \emph{not} generate $G$, that is, 
\[
P(z,x) = \frac{|\{y \in x^G \,:\, \< z, y \> \neq G \}|}{|x^G|}.
\]

Now let us specialise to the case where $G$ is an almost simple group with socle $G_0$. Recall that $\mathcal{M}$ is the set of maximal subgroups $H$ of $G$ such that $G = HG_0$. For an element $x \in G$, write $\mathcal{M}(x)$ for the set of subgroups $H \in \mathcal{M}$ that contain $x$. Notice that if the conjugacy class $x^G$ witnesses $u(G) \geqs 1$, then we must have $G/G_0 = \<G_0x\>$ and thus $\mathcal{M}(x)$ is simply the set of all maximal subgroups of $G$ that contain $x$. Given this observation, the following result is a combination of \cite[Lemmas~2.1 and~2.2]{BG}.

\begin{lem} \label{l:prob_method}
Let $G$ be an almost simple group with socle $G_0$. Let $x \in G$ with $G/G_0 = \<G_0x\>$. 
\begin{itemize}\addtolength{\itemsep}{0.2\baselineskip}
\item[{\rm (i)}]  For $z \in G$, we have
\[ 
P(z,x) \leq \sum_{H \in \mathcal{M}(x)}^{} {\rm fpr}(z,G/H).
\]
\item[{\rm (ii)}] If $P(z,x) < 1/k$ for all nontrivial $z \in G$, then $u(G) \geq k$, witnessed by $x^G$.
\end{itemize}
\end{lem}

Roughly speaking, in order to effectively apply Lemma~\ref{l:prob_method} we need to do two things:
\begin{itemize}\addtolength{\itemsep}{0.2\baselineskip}
\item[(a)] First we must identify an appropriate element $x \in G$ such that $G/G_0 = \la G_0x \ra$ and we have some control on the set of maximal overgroups $\mathcal{M}(x)$;
\item[(b)] Then we need to compute upper bounds on the fixed point ratios ${\rm fpr}(z,G/H)$ for all $H \in \mathcal{M}(x)$ and all nontrivial $z \in G$.
\end{itemize}

In the case where $G_0$ is a simple exceptional group of Lie type, upper bounds on ${\rm fpr}(z,G/H)$ for all maximal subgroups $H$ of $G$ are determined by Lawther, Liebeck and Seitz in \cite{LLS} and we will make extensive use of their work (and in a few cases, we will need to strengthen the bounds in \cite{LLS}). 

To handle the problem identified in (a), we will often appeal to the theory of \emph{Shintani descent}, both to find an element $x$ and to control the maximal subgroups containing $x$. We discuss this approach in the next section.

\subsection{Shintani descent}\label{ss:p_shintani}

To close this preliminary section, we briefly recall the general theory of Shintani descent, which is our principal method for identifying and studying elements in the nontrivial cosets of the socle of an almost simple group of Lie type. The general method was introduced by Shintani \cite{Shintani} and Kawanaka \cite{Kawanaka} in the 1970s and it has found important applications in character theory. It was first adapted for studying the uniform spread of almost simple groups in \cite{BG} and we refer the reader to \cite[Chapter~3]{HarperClassical} for a convenient overview of the relevant techniques. 

To describe the general set up, let $\bar{G}$ be a connected algebraic group over an algebraically closed field and let $\s$ be a Steinberg endomorphism of $\bar{G}$. Fix an integer $e > 1$. By identifying $\s$ with its restriction to $\bar{G}_{\s^e}$, we can consider the finite semidirect product $\bar{G}_{\s^e}{:}\<\s\> = \bar{G}_{\s^e}.e$. 

\begin{defn}\label{d:shin}
A \emph{Shintani map} of $(\bar{G},\s,e)$ is a map of conjugacy classes of the form
\[
F\: \{(g\s)^{\bar{G}_{\s^e}} \,:\, g \in \bar{G}_{\s^e} \} \to \{y^{\bar{G}_{\s}} \,:\, y \in \bar{G}_{\s} \}, \;\; (g\s)^{\bar{G}_{\s^e}} \mapsto (a^{-1}(g\s)^ea)^{\bar{G}_{\s}}
\] 
where $a \in \bar{G}$ satisfies $g=aa^{-\s^{-1}}$ (such an element $a$ exists by the Lang-Steinberg theorem, see \cite[Theorem 2.1.1]{GLS}).
\end{defn}

We now present the main theorem of Shintani descent (see \cite[Lemma~2.2]{Kawanaka}).

\begin{thm}\label{t:shintani}
Let $F$ be a Shintani map of $(\bar{G},\s,e)$. Then $F$ is a well-defined bijection from the set of $\bar{G}_{\s^e}$-conjugacy classes in the coset $\bar{G}_{\s^e}\s$ to the set of conjugacy classes in $\bar{G}_{\s}$. Moreover, $F$ does not depend on the choice of element $a \in \bar{G}$.
\end{thm}

In light of Theorem~\ref{t:shintani}, we refer to $F$ as \emph{the} Shintani map of $(\bar{G},\s,e)$. To simplify the notation, if the setting is understood, we will write $F\: \bar{G}_{\s^e}\s \to \bar{G}_\s$ for the Shintani map and $F(g\s)$ for a representative of the $\bar{G}_\s$-class $F((g\s)^{\bar{G}_{\s^e}})$. We refer to $g\s$ as a \emph{Shintani correspondent} of $F(g\s)$.

The following elementary observation highlights the relationship between the order of an element in $\bar{G}_{\s}$ and the order of a Shintani correspondent in the coset $\bar{G}_{\s^e}\s$. 

\begin{lem}\label{l:shintani_order}
Let $y \in \bar{G}_\s$ and let $g \in \bar{G}_{\s^e}$ such that $F(g\s) = y$. Then $|g\s| = e|y|$.
\end{lem}

\begin{proof}
Since $g\s \in \bar{G}_{\s^e}{:}\<\s\>$, it follows that $e$ divides the order of $g\s$. Therefore, $|g\s|=e|(g\s)^e|$ and we conclude that $|g\s|=e|y|$ since $(g\s)^e$ is $\bar{G}$-conjugate to $y$.
\end{proof}

We will need the following technical result \cite[Corollary~3.2.3]{HarperClassical} (in the statement, for a finite group $X$ we write $O^{p'}(X)$ for the normal subgroup generated by the $p$-elements of $X$). 

\begin{lem}\label{l:shintani_split}
Let $\bar{G}$ be a simple algebraic group over $\bar{\mathbb{F}}_p$ of adjoint type and set $G_0 = (\bar{G}_{\s^e})'$. If $\< G_0, \s \> \leqn \< \bar{G}_{\s^e}, \s \>$, then the Shintani map $F$ of $(\bar{G},\s,e)$ restricts to a bijection
\[
\{ (g\s)^{\bar{G}_{\s^e}} \,:\, g \in G_0 \} \to \{ y^{\bar{G}_\s} \,:\, y \in O^{p'}(\bar{G}_\s) \}.
\]  
\end{lem}

Let us provide an example to demonstrate how we will use Lemma~\ref{l:shintani_split}.

\begin{ex}\label{ex:shintani_split}
Here we explain how we use Shintani descent to identify a conjugacy class in the coset $E_7(q)h$, where $q=p^f$ and $h$ is a field automorphism. 

Let $\bar{G}$ be the adjoint algebraic group of type $E_7$ over $\bar{\mathbb{F}}_p$. Let $\p$ be a standard Frobenius endomorphism of $\bar{G}$, let $\s = \p^i$ for a proper divisor $i$ of $f$ and set $e=f/i > 1$. Write $q = q_0^e$ and let $F$ be the Shintani map of $(\bar{G},\s,e)$. 

If $q$ is even, then $\bar{G}_{\s^e}$ and $\bar{G}_\s$ are the simple groups $E_7(q)$ and $E_7(q_0)$, respectively, so
\[
F\: \{ (g\p^i)^{E_7(q)} \,:\, g \in E_7(q) \} \to \{ y^{E_7(q_0)} \,:\, y \in E_7(q_0) \}.
\]
Therefore, we may select an element in the coset $E_7(q)\p^i$ by identifying an element in the subgroup $E_7(q_0)$ and taking its Shintani correspondent. However, if $q$ is odd, then $|\bar{G}_{\s^e}:E_7(q)| = |\bar{G}_\s:E_7(q_0)| = 2$ and the Shintani map
\[
F\: \{ (g\p^i)^{\bar{G}_{\s^e}} \,:\, g \in \bar{G}_{\s^e} \} \to \{ y^{\bar{G}_\s} \,:\, y \in \bar{G}_\s \}
\]
allows us to identify an element in $\bar{G}_{\s^e}\p^i$ but it does not tell us which coset of $E_7(q)$ this element is contained in. This is where Lemma~\ref{l:shintani_split} comes into play. 

Observe that $E_7(q) = (\bar{G}_{\s^e})'$ and $E_7(q_0) = O^{p'}(\bar{G}_\s)$. Moreover, $\<\ddot{\s}\>$ is an index two subgroup of $\<\ddot{\d},\ddot{\s}\> = \< \bar{G}_{\s^e}, \s\>/G_0$ (see Table~\ref{tab:out}), so $\<G_0,\s\> \leqn \<\bar{G}_{\s^e}, \s\>$. Therefore, Lemma~\ref{l:shintani_split} implies that $F$ restricts to a bijection
\[
\{ (g\p^i)^{\bar{G}_{\s^e}} \,:\, g \in E_7(q) \} \to \{ y^{\bar{G}_\s} \,:\, y \in E_7(q_0) \}.
\]
This means that the coset of $E_7(q_0)$ in $\bar{G}_{\s}$ containing a given element $y \in \bar{G}_{\s}$ controls the coset of $G_0=E_7(q)$ in ${\rm Aut}(G_0)$ that contains the Shintani correspondent of $y$.
\end{ex}

It is important to observe that the Shintani map gives more than just the bijection between conjugacy classes stated in Theorem \ref{t:shintani}. Indeed, we can use it to shed light on the overgroups in $\< \bar{G}_{\s^e}, \s \>$ of an element in the coset $\bar{G}_{\s^e}\s$. This is encapsulated in Lemmas~\ref{l:shintani_subgroups} and~\ref{l:centraliser_bound} below, which coincide with Lemmas~3.3.2 and~3.3.4 in \cite{HarperClassical} (in turn these results are closely related to Corollary~2.15 and Proposition~2.16(i) in \cite{BG}).

\begin{lem}\label{l:shintani_subgroups}
Let $\bar{H}$ be a closed connected $\s$-stable subgroup of $\bar{G}$ such that $N_{\bar{G}_{\s}}(\bar{H}_{\s}) = \bar{H}_{\s}$ and $N_{\bar{G}_{\s^e}}(\bar{H}_{\s^e}) = \bar{H}_{\s^e}$. Then for all $g \in \bar{G}_{\s^e}$, the number of $\bar{G}_{\s^e}$-conjugates of $\bar{H}_{\s^e}$ normalised by $g\s$ equals the number of $\bar{G}_{\s}$-conjugates of $\bar{H}_{\s}$ containing $F(g\s)$.
\end{lem}

\begin{cor}\label{c:shintani_parabolic}
Let $\bar{G}$ be a simple algebraic group and let $g \in \bar{G}_{\s^e}$. Then the number of maximal parabolic subgroups of $G = \< \bar{G}_{\s^e}, \s \>$ that contain $g\s$ is equal to the number of maximal parabolic subgroups of $\bar{G}_\s$ that contain $F(g\s)$.
\end{cor}

\begin{proof}
Let $\bar{H}$ be a maximal $\s$-stable parabolic subgroup of $\bar{G}$, so $\bar{H}$ is connected and self-normalising. Then $\bar{H}_\s$ is a maximal parabolic subgroup of $\bar{G}_\s$ and we have $N_{\bar{G}_\s}(\bar{H}_\s) = \bar{H}_\s$. Similarly, $H = N_G(\bar{H}_{\s^e}) = \< \bar{H}_{\s^e}, \s \>$ is a maximal parabolic subgroup of $G$ and $N_G(H) = H$. Therefore, Lemma~\ref{l:shintani_subgroups} implies that the number of $G$-conjugates of $H$ that contain $g\s$ equals the number of $\bar{G}_\s$-conjugates of $\bar{H}_\s$ that contain $F(g\s)$.

Let us now explain why this gives the desired result. First observe that every maximal parabolic subgroup of $\bar{G}_\s$ is $\bar{G}_\s$-conjugate to $\bar{H}_\s$ for a maximal $\s$-stable parabolic subgroup $\bar{H}$ of $\bar{G}$, and similarly, every maximal parabolic subgroup of $G$ is $G$-conjugate to $N_G(\bar{H}_{\s^e})$ for a maximal $\s$-stable parabolic subgroup $\bar{H}$ of $\bar{G}$ (in the latter case,  $\bar{H}$ is $\s$-stable, not just $\s^e$-stable, because otherwise $N_G(\bar{H}_{\s^e})$ would not be maximal in $G = \<\bar{G}_{\s^e},\s\>$). Moreover, if $\bar{H}$ and $\bar{K}$ are two different maximal $\s$-stable parabolic subgroups of $\bar{G}$, then $\bar{H}_\s$ and $\bar{K}_\s$ are $\bar{G}_\s$-conjugate if and only if $N_{G}(\bar{H}_{\s^e})$ and $N_{G}(\bar{K}_{\s^e})$ are $G$-conjugate, since both of these conditions are equivalent to $\bar{H}$ and $\bar{K}$ being $\bar{G}$-conjugate. The result follows.
\end{proof}

\begin{lem}\label{l:centraliser_bound}
If $g \in \bar{G}_{\s^e}$ and $H \leq \<\bar{G}_{\s^e}, \s\>$, then $g\s$ is contained in at most $|C_{\bar{G}_{\s}}(F(g\s))|$ distinct $\<\bar{G}_{\s^e}, \s\>$-conjugates of $H$. 
\end{lem}

In the proof of Theorem~\ref{t:exceptional}, there will be some cases where we will be unable to apply Shintani descent directly (for instance, see Example~\ref{ex:shintani_substitute}). In such a situation, we will often appeal to the following result (see \cite[Lemma~3.4.1]{HarperClassical}). In the statement of the lemma, by an automorphism $\rho$ of $\bar{G}$ we mean an \emph{algebraic} automorphism, in the sense that both $\rho$ and $\rho^{-1}$ are morphisms of varieties. 

\begin{lem}\label{l:shintani_substitute}
Let $\rho$ be an automorphism of $\bar{G}$ and let $\bar{K}$ be a closed connected $\s$-stable subgroup of $C_{\bar{G}}(\rho)$. Set $G = \bar{G}_{\rho\s^e}{:}\<\rho,\s\>$ and let $y \in \bar{K}_{\s} \leq \bar{G}_{\rho\s^e}$. 
\begin{itemize}\addtolength{\itemsep}{0.2\baselineskip}
\item[{\rm (i)}] There exists $g \in \bar{K}_{\s^e} \leq \bar{G}_{\rho\s^e}$ such that $(g\s)^e$ and $y\rho^{-1}$ are $\bar{G}$-conjugate elements of $G$.
\item[{\rm (ii)}] Suppose there is a positive integer $d$ such that $(\rho\s^e)^d = \s^{ed}$ as endomorphisms of $\bar{G}$.

\vspace{1mm}

\begin{itemize}\addtolength{\itemsep}{0.2\baselineskip}
\item[{\rm (a)}]  For each subgroup $H$ of $\< \bar{G}_{\rho\s^e}, \s\>$,  $g\s$ is contained in at most $|C_{\bar{G}_{\s}}(y^d)|$ distinct $\bar{G}_{\rho\s^e}$-conjugates of $H$.
\item[{\rm (b)}] For all closed connected $\s$-stable subgroups $\bar{H}$ of $\bar{G}$ such that $N_{\bar{G}_{\s}}(\bar{H}_{\s}) = \bar{H}_{\s}$ and $N_{\bar{G}_{\s^{de}}}(\bar{H}_{\s^{de}}) = \bar{H}_{\s^{de}}$, the number of $\bar{G}_{\s^{de}}$-conjugates of $\bar{H}_{\s^{de}}$ normalised by $g\s$ is equal to the number of $\bar{G}_{\s}$-conjugates of $\bar{H}_{\s}$ containing $y^d$.
\end{itemize}
\end{itemize}
\end{lem}

\begin{rem}\label{r:shintani_order}
Adopt the notation in Lemma~\ref{l:shintani_substitute} and fix an appropriate element $g \in \bar{G}_{\rho\s^e}$ as in part (i). Now $e$ divides $|g\s|$ and $(g\s)^e$ is $\bar{G}$-conjugate to $y\rho^{-1}$, so $|g\s| = e|y\rho^{-1}|$. Since $y \in C_{\bar{G}}(\rho)$ we have $|y\rho^{-1}| = |y||\rho|/(|y|,|\rho|)$ and thus $|g\s| = e|y||\rho|/(|y|,|\rho|)$.
\end{rem}

The following example explains why Lemma~\ref{l:shintani_substitute} will be useful in the proof of Theorem~\ref{t:exceptional}.

\begin{ex}\label{ex:shintani_substitute}
Here we explain how we can use Shintani descent to identify a conjugacy class in the coset $E_6(q)h$, where $q=3^f$ and $h$ is a graph-field automorphism. 

Let $\bar{G}$ be the adjoint simple algebraic group $E_6$ over $\bar{\mathbb{F}}_3$. Let $\p$ and $\g$ be the standard Frobenius endomorphism and graph automorphism of $\bar{G}$, respectively, so $[\g,\p]=1$ (see Definition~\ref{d:aut}). Write $\s = \g\p^i$,  where $i$ divides $f$, and set $e=f/i > 1$. Write $q=q_0^e$ and let $F$ be the Shintani map of $(\bar{G},\s,e)$. 

If $e$ is even, then $\bar{G}_{\s^e} = \bar{G}_{\p^f} = E_6(q)$ and $\bar{G}_\s = \bar{G}_{\g\p} = {}^2E_6(q_0)$. Therefore,
\[
F\: \{ (g\g\p^i)^{E_6(q)} \,\:\, g \in E_6(q) \} \to \{ y^{{}^2E_6(q_0)} \,\:\, y \in {}^2E_6(q_0) \}
\]
and we can use $F$ to choose an element in the coset $E_6(q)\g\p^i$ as desired.

However, if $e$ is odd, then $\bar{G}_{\s^e} = \bar{G}_{\g\p^f} = {}^2E_6(q)$ and the Shintani map 
\[
F\: \{ (g\g\p^i)^{{}^2E_6(q)} \,\:\, g \in {}^2E_6(q) \} \to \{ y^{{}^2E_6(q_0)} \,\:\, y \in {}^2E_6(q_0) \}
\]
provides no information about the coset $E_6(q)\g\p^i$. In this case we apply Lemma~\ref{l:shintani_substitute}, with $\r = \g$. To this end, let $\bar{K} = C_{\bar{G}}(\g) = F_4$, which is connected. Then Lemma~\ref{l:shintani_substitute} allows us to choose an element in the coset $E_6(q)\gamma\p^i$. More precisely, part~(i) of the lemma implies that for all $y \in F_4(q_0) \leq {}^2E_6(q_0)$, there exists $g \in E_6(q)$ such that $(g\g\p^i)^e$ is $\bar{G}$-conjugate to $y\g$. In addition, part~(ii) provides  information on the maximal overgroups of $g\g\p^i$.
\end{ex}

\section{Proof of Theorem \ref{t:exceptional}: low rank groups} \label{s:low}

We now turn to the proof of Theorem~\ref{t:exceptional}, which will be spread across Sections~\ref{s:low}--\ref{s:3d4}. Since the theorem for simple exceptional groups is proved in \cite{BGK}, we will always assume that $G$ is almost simple, but not simple.

We begin in this section by handling the low rank almost simple groups $G$ with socle 
\begin{equation} \label{e:a}
G_0 \in \{ {}^2B_2(q), \, {}^2G_2(q)', \, {}^2F_4(q)', \, G_2(q)' \}.
\end{equation}

First we establish Theorem~\ref{t:exceptional} in some special cases. 

\begin{prop} \label{p:small}
The conclusion to Theorem \ref{t:exceptional} holds when 
\begin{equation} \label{e:small}
G_0 \in \{ {}^2B_2(8),  {}^2G_2(3)',  {}^2F_4(2)',  G_2(2)',  G_2(3),  G_2(4) \}.
\end{equation}
\end{prop}

\begin{proof}
In each of these cases, we may assume that $G = {\rm Aut}(G_0)$ since this is the only almost simple group $G$ with ${\rm soc}(G) = G_0$ and $G \ne G_0$. We prove the result by way of computation in {\sc Magma} \cite{magma}. 

To do this, we first construct $G$ using the command {\tt AutomorphismGroupSimpleGroup} and we note that $|G:G_0|$ is prime. Our method for studying $u(G)$ computationally is described in \cite[Section~2.3]{Harper17} and the relevant code is given in \cite[Appendix~A]{HarperClassical}. In this way, we can verify that the bound $u(G) \geq k$ is witnessed by the conjugacy class $x^G$, where $k$ and $x^G$ are are as follows (in terms of the {\sc Atlas} \cite{ATLAS} notation):
\[
\begin{array}{ccccccc}
\hline
G_0     & {}^2B_2(8)   & {}^2G_2(3)'  & {}^2F_4(2)'  & G_2(2)'      & G_2(3)       & G_2(4)       \\
|G:G_0| & 3            & 3            & 2            & 2            & 2            & 2            \\
x^G       & \texttt{15A} & \texttt{9D}  & \texttt{12C} & \texttt{12C} & \texttt{18A} & \texttt{24B} \\
k       & 90           & 6            & 18           & 3            & 23           & 10           \\
\hline
\end{array}
\]
(The computations were carried out using {\sc Magma} 2.24-4 on a 2.7~GHz machine with 128~GB RAM. The largest computation took 2~seconds and 32~MB of memory.)
\end{proof}

Suppose $G = \< G_0, g \>$ with $G_0$ as in \eqref{e:a} and write $q=p^f$ where $p$ is prime. In view of Proposition~\ref{p:small}, we may (and will) assume for the remainder of this section that $G_0$ is not one of the groups in \eqref{e:small}. Then by Proposition~\ref{p:cases}, it suffices to consider the groups recorded in Table \ref{tab:small_cases}. In the table (and the proofs below), we refer freely to the notation for automorphisms in Definition~\ref{d:aut}.

\begin{table}
\caption{The relevant groups $G = \la G_0,g \ra$ for $G_0$ in \eqref{e:a}} \label{tab:small_cases}
\begin{center}
\vspace{-8mm}
\[
\begin{array}{clcl} 
\hline
\text{Case} & G_0                                & g      & \text{Conditions}                             \\ 
\hline    
\text{(a)}  & G_2(q)                             & \p^i   & \text{$i$ is a proper divisor of $f$}         \\
\text{(b)}  & G_2(q)                             & \rho^i & \text{$i$ is an odd divisor of $f$ \&\ $p=3$} \\
\text{(c)}  & {}^2B_2(q), {}^2G_2(q), {}^2F_4(q) & \rho^i & \text{$i$ is a proper divisor of $f$}         \\ 
\hline 
\end{array}
\]
\end{center}
\end{table}

\begin{prop} \label{p:low_a}
The conclusion to Theorem~\ref{t:exceptional} holds in case~(a) of Table \ref{tab:small_cases}.
\end{prop}

\begin{proof}
Let $G_0 = G_2(q)$ where $q=p^f$ with $f > 1$ and $q \geq 8$. Let $\bar{G}$ be the simple algebraic group $G_2$ over the algebraic closure of $\mathbb{F}_p$ and let $\p$ be a standard Frobenius endomorphism of $\bar{G}$. Let $\s = \p^i$ and write $e=f/i$ and $q_0=p^i$, so $q=q_0^e$ and $e > 1$. Then $\bar{G}_\s = G_2(q_0)$ and $\bar{G}_{\s^e} = G_2(q)$, and by identifying $\s$ with its restriction to $\bar{G}_{\s^e}$ we see that $\s=g$. Let $F\:G_2(q)g \to G_2(q_0)$ be the Shintani map of $(\bar{G},\s,e)$ (see Definition \ref{d:shin}) and choose $y \in G_2(q_0)$ such that 
\[
|y| = \left\{
\begin{array}{ll}
q_0^2-q_0+1 & \text{if $q_0 > 2$} \\
7           & \text{if $q_0=2$}.
\end{array}
\right.
\]
Note that $C_{G_2(q_0)}(y)=\<y\>$ (see \cite{Chang, Enomoto}). By Theorem~\ref{t:shintani}, fix $x \in G_0g$ such that $F(x) = y$.

Recall that $\mathcal{M}$ is the set of maximal subgroups $H$ of $G$ with $G = HG_0$ and $\mathcal{M}(x)$ is the collection of subgroups in $\mathcal{M}$ containing $x$. The maximal subgroups of $G$ are recorded in \cite[Tables~8.30, 8.41 and~8.42]{BHR}. The element $y$ is not contained in any maximal parabolic subgroup of $G_2(q_0)$ since $|y|$ does not divide the order of any such subgroup. Therefore, Corollary~\ref{c:shintani_parabolic} informs us that there are no maximal parabolic subgroups in $\mathcal{M}(x)$. Consequently, by inspecting the relevant tables in \cite[Chapter 8]{BHR}, we see that there are at most $6 - 3\delta_{2,p} + \log\log q$ conjugacy classes of subgroups in $\mathcal{M}(x)$. Moreover, if $H$ is any subgroup of $G$, then Lemma~\ref{l:centraliser_bound} implies that $x$ is contained in at most $|C_{G_2(q_0)}(y)|=|y|$ distinct $G$-conjugates of $H$. Therefore,
\[
|\mathcal{M}(x)| \leq (6 - 3\delta_{2,p} + \log\log q) \cdot |y|.
\]

Let $z \in G$ be nontrivial. Then \cite[Theorem~1]{LLS} gives ${\rm fpr}(z,G/H) \leqs (q^2-q+1)^{-1}$ for all $H \in \mathcal{M}$ and thus Lemma~\ref{l:prob_method}(i) yields
\[
P(z,x) \leq \sum_{H \in \mathcal{M}(x)} {\rm fpr}(z,G/H) \leq  (6-3\delta_{2,p}+\log\log{q}) \cdot |y| \cdot (q^2-q+1)^{-1}. 
\] 
For $q > 49$, this upper bound proves that $P(z,x) < q^{-1/2}$, so $P(z,x) \to 0$ as $q \to \infty$. In view of Lemma~\ref{l:prob_method}(ii), we conclude that $u(G) \to \infty$ as $q \to \infty$.

Moreover, since $q \geqs 8$, one checks that this upper bound is less than $\frac{1}{2}$ unless $q \in \{8,9\}$. If $q=9$, then $|y|=7$ and we check that there are only $3$ conjugacy classes of subgroups in $\mathcal{M}$ with order divisible by $7$ (here we are using the fact that $G$ does not contain any graph-field automorphisms). This allows us to replace the leading factor $6+\log\log{q}$ in the above bound by $3$ and this is sufficient to see that $P(z,x) < \frac{1}{2}$. Similarly, if $q=8$ then we can replace $3+\log\log q$ by $4$, which yields
\[
P(z,x) \leq \sum_{H \in \mathcal{M}(x)}{\rm fpr}(z,G/H) \leqs 4 \cdot 7 \cdot \frac{1}{57}  = \frac{28}{57}< \frac{1}{2}.
\]
Therefore, $P(z,x) < \frac{1}{2}$ in all cases and thus Lemma~\ref{l:prob_method}(ii) implies that $u(G) \geq 2$.
\end{proof}

\begin{prop} \label{p:low_b}
The conclusion to Theorem~\ref{t:exceptional} holds in case~(b) of Table \ref{tab:small_cases}.
\end{prop}

\begin{proof}
Let $G_0=G_2(q)$ where $q=3^f$ and $f > 1$. Let $\bar{G} = G_2$ and let $\r$ be the Steinberg endomorphism of $\bar{G}$ from Definition~\ref{d:aut}(ii). Let $\s = \r^i$ and write $e=f/i$ and $q_0=3^i$, so $q=q_0^e$ and $e \geq 1$. Let $F\:G_2(q)g \to {}^2G_2(q_0)$ be the Shintani map of $(\bar{G},\s,2e)$, and fix $y \in {}^2G_2(q_0)$ with
\[
|y| = q_0+\sqrt{3q_0}+1.
\] 
Note that $C_{{}^2G_2(q_0)}(y)=\<y\>$ (see (3) in the main theorem of \cite{Ward}). Let $x \in G$ satisfy $F(x) = y$. 

By \cite{KleidmanG2}, there are at most $7 + \log\log q$ classes of subgroups in $\mathcal{M}$ and by Lemma~\ref{l:centraliser_bound}, $\mathcal{M}(x)$ contains at most $|C_{{}^2G_2(q_0)}(y)|=|y|$ conjugates of any given subgroup $H$ of $G$. Let $z \in G$ be nontrivial. Then \cite[Theorem~1]{LLS} gives ${\rm fpr}(z,G/H) \leqs (q^2-q+1)^{-1}$ for all $H \in \mathcal{M}$, so
\[
P(z,x) \leq \sum_{H \in \mathcal{M}(x)}{\rm fpr}(z,G/H) \leq (7+\log\log{q}) \cdot |y| \cdot (q^2-q+1)^{-1}.
\] 
This upper bound is less than $\frac{1}{2}$ for $q > 9$ and less than $q^{-1/2}$ for $q > 27$. Finally, if $q=9$ then there are only $2$ classes of subgroups in $\mathcal{M}$ with order divisible by $|y|=7$ and we obtain $P(z,x) < \frac{1}{2}$ by replacing the $7+\log\log{q}$ factor in the above bound by $2$. The result now follows by Lemma~\ref{l:prob_method}.
\end{proof}

\begin{prop} \label{p:low_c}
The conclusion to Theorem~\ref{t:exceptional} holds in case~(c) of Table \ref{tab:small_cases}.
\end{prop}

\begin{proof}
Let $G_0 \in \{{}^2B_2(q), \, {}^2G_2(q), \, {}^2F_4(q)\}$. As usual, let $q=p^f$ where $p$ is prime, and note that $f \geq 3$ is odd. In each case, let $\bar{G}$ be the ambient simple algebraic group and let $\rho$ be the Steinberg endomorphism of $\bar{G}$ from Definition~\ref{d:aut}(ii). Let $\s = \rho^i$ and write $e=f/i$ and $q_0=p^i$, so $q=q_0^e$ and $e \geqs 3$ is odd. Let $F:G_0g \to \bar{G}_{\s}$ be the Shintani map of $(\bar{G},\s,e)$. 

Choose $y \in \bar{G}_\s$ as in Table~\ref{tab:small} and let $x \in G$ be a Shintani correspondent of $y$. By inspecting \cite{KleidmanG2, Malle, Suzuki}, we see that there are at most $m+\log \log q$ classes of subgroups in $\mathcal{M}$, where $m$ is given in Table~\ref{tab:small}. Moreover, $C_{\bar{G}_\s}(y) = \<y\>$ (see \cite{Shinoda75, Suzuki, Ward}), so $|\mathcal{M}(x)| \leq (m+\log\log{q})\cdot|y|$ by Lemma~\ref{l:centraliser_bound}. In addition, \cite[Theorem~1]{LLS} gives ${\rm fpr}(z,G/H) \leq a(q)$ for all $H \in \mathcal{M}$ and all nontrivial $z \in G$, where $a(q)$ is presented in Table~\ref{tab:small} (note that in the first row of Table~\ref{tab:small}, $\ell$ is the least prime divisor of $f$). Therefore, 
\[
P(z,x) \leq \sum_{H \in \mathcal{M}(x)}{\rm fpr}(z,G/H) \leq (m+\log\log{q}) \cdot |y| \cdot a(q).
\]
One can check that this bound gives $P(z,x) < \frac{1}{2}$ and $P(z,x) < q^{-1/6}$, whence $u(G) \geq 2$ and $u(G) \to \infty$ as $q \to \infty$. 
\end{proof}

\begin{table}
\caption{Data for the groups in case (c) of Table \ref{tab:small_cases}} \label{tab:small}
\begin{center}
\vspace{-8mm}
\[
\begin{array}{cccc} 
\hline
G_0        & |y|                                   & m  & a(q)                   \\ 
\hline    
{}^2B_2(q) & q_0+\sqrt{2q_0}+1                     & 4  & (q^{2/\ell}+1)/(q^2+1) \\
{}^2G_2(q) & q_0+\sqrt{3q_0}+1                     & 5  & (q^2-q+1)^{-1}         \\   
{}^2F_4(q) & q_0^2+\sqrt{2q_0^3}+q_0+\sqrt{2q_0}+1 & 11 & q^{-4}                 \\ 
\hline 
\end{array}
\]
\end{center}
\end{table}

By combining Propositions~\ref{p:small}--\ref{p:low_c}, we have now established the following theorem.

\begin{thm} \label{t:low}
The conclusion to Theorem~\ref{t:exceptional} holds when $G_0$ is one of the groups in \eqref{e:a}.
\end{thm}

In the next five sections, we will complete the proof of Theorem~\ref{t:exceptional} by handling the remaining groups with $G_0 \in \{ E_8(q), E_7(q),  E_6^{\e}(q),  F_4(q),  {}^3D_4(q) \}$.

\section{Proof of Theorem \ref{t:exceptional}: \texorpdfstring{$G_0 = E_8(q)$}{G0 = E8(q)}} \label{s:e8}

In this section, we prove Theorem~\ref{t:exceptional} for almost simple groups $G$ with socle $G_0 = E_8(q)$, where $q=p^f$. By Proposition~\ref{p:cases}, we may assume that $G = \< G_0, g \>$, where $g=\p^i$ for the field automorphism $\p$ in Definition \ref{d:aut}(i) and a proper divisor $i$ of $f$.  

Let $\bar{G}$ be the algebraic group $E_8$ over $\bar{\mathbb{F}}_p$, let $\s$ be the Frobenius endomorphism $\p^i$ of $\bar{G}$ and let $e = f/i$, so $G_0 = \bar{G}_{\s^e}$. Set $q=q_0^e$ and let $F\: E_8(q)g \to E_8(q_0)$ be the Shintani map of $(\bar{G},\s,e)$. 

Fix an element $y \in E_8(q_0)$ such that  
\[
|y|= q_0^8+q_0^7-q_0^5-q_0^4-q_0^3+q_0+1
\]
and $C_{E_8(q_0)}(y) = \la y \ra$ (see \cite{Lubeck} or \cite[Section 3]{FJ}). Let $x \in G$ be a Shintani correspondent of $y$ (that is, choose $x \in G$ such that $F(x) = y$). Then by Lemma~\ref{l:shintani_order}, we have $|x|=e|y|$ and we note that $|y| = 331$ if $q_0=2$ and $|y| \geq 8401$ if $q_0 \geq 3$.

Recall that for integers $a, b \geq 2$, a prime $r$ is said to be a \emph{primitive prime divisor} of $a^b-1$ if $r$ divides $a^b-1$ but $r$ does not divide $a^i-1$ for all $1 \leq i < b$. A theorem of Zsigmondy \cite{Zsigmondy} asserts that $a^b-1$ has at least one primitive prime divisor for all integers $a,b \geq 2$ unless $(a,b) = (2,6)$, or $a$ is a Mersenne prime and $b=2$. In particular, $q_0^{30}-1$ has a primitive prime divisor, and by considering the factorisation of $q_0^{30}-1$ as a product of cyclotomic polynomials, we see that such a primitive prime divisor necessarily divides $|y|$. 

As usual, we write $\mathcal{M}$ for the set of maximal subgroups $H$ of $G$ with $G = HG_0$ and $\mathcal{M}(x)$ for the collection of subgroups in $\mathcal{M}$ containing $x$. In the analysis below, we will refer repeatedly to the partition $\mathcal{M} = \mathcal{M}_1 \cup \mathcal{M}_2 \cup \mathcal{M}_3$ in \eqref{e:mpart}.

\begin{prop} \label{p:e8_max}
We have $\mathcal{M}(x) \subseteq \mathcal{M}_1$.
\end{prop}

\begin{proof}
Let $H \in \mathcal{M}(x)$. If $H \in \mathcal{M}_3$, then Proposition~\ref{p:orders} gives ${\rm meo}(H) \leqs 210$, which is incompatible with the bound $|x| \geqs 331e$. Therefore, we may assume $H \in \mathcal{M}_2$. Let $S$ be the socle of $H$, which is a simple group of Lie type over a field $\mathbb{F}_t$ of characteristic $p$. We proceed by considering the possibilities for $S$ given in Theorem~\ref{t:simples}.

If $S = {\rm L}_{3}^{\e}(16)$, then Proposition~\ref{p:meo} gives ${\rm meo}(H) \leqs 273 < |x|$, so this case does not arise. Next assume that $S = {\rm L}_{2}(t)$ and $t \leqs 1312(2,t-1)$. By applying Proposition~\ref{p:meo}, we reduce to the case $q_0=2$, so $|x|=331e$ and $t = 2^k$ with $k \leqs 10$. However, for each $k$, it is easy to check that $|S|$ is indivisible by $331$, so this case does not arise. Next assume that $S = {}^2B_2(t)$, so $p=2$ and $t = 2^{2k+1}$ with $k \leqs 4$. Here Proposition~\ref{p:meo} gives ${\rm meo}(H) \leqs 4905$, so we immediately reduce to the case $q_0=2$ and one checks that $|S|$ is indivisible by $331$. Similarly, if $S = {}^2G_2(t)'$, then $t = 3^{2k+1}$ with $k \leqs 3$ and ${\rm meo}(H) \leqs 15883 < 8401e$, so this case is also ruled out. 

To complete the proof of the proposition, we may assume that ${\rm rk}(S) \in \{2,3,4\}$ and $t \leqs 9$. We consider each possibility for $S$ in turn, excluding ${}^2B_2(t)$ and ${}^2G_2(t)'$ since these groups were handled above.

To get started, let us assume ${\rm rk}(S)=4$, so 
\[
S \in \{{\rm L}_{5}^{\e}(t), \, {\rm PSp}_{8}(t), \, {\rm P}\Omega_{8}^{\e}(t), \, \Omega_9(t), \, F_4(t), \, {}^2F_4(t), \, {}^3D_4(t)\}.
\]
If $S$ is a classical group, Proposition~\ref{p:meo} gives ${\rm meo}(H) \leqs t^5/(t-1)$ and we immediately reduce to the case $(t,q_0)=(8,2)$. Here one checks that $|S|$ is divisible by $331$ if and only if $S = {\rm U}_{5}(8)$, but this case is ruled out by Proposition~\ref{p:sp}. Now assume that $S = F_4(t)$. By applying the bound on ${\rm meo}(H)$ from Proposition~\ref{p:meo}, we may assume that either $q_0=2$, or $q_0=3$ and $t=9$. For $q_0=2$ we have $t \in \{2,4,8\}$ and $|S|$ is indivisible by $331$. Similarly, if $q_0 = 3$, then $|y| = 8401 = 31\cdot271$, but $|S|$ is indivisible by $31$. The cases where $S$ is ${}^2F_4(t)'$ and ${}^3D_4(t)$ are very similar. For example, if $S = {}^2F_4(t)'$, then we reduce to the case $t=8$ with $q_0=2$ and one checks that $|{}^2F_4(8)|$ is indivisible by $331$.

Now assume ${\rm rk}(S) \in \{2,3\}$. If $S$ is classical, then the bound in Proposition~\ref{p:meo} implies that ${\rm meo}(H) < |x|$. Finally, if $S = G_2(t)'$, then Proposition~\ref{p:meo} gives ${\rm meo}(H) \leqs 8(t^2+t+1)\log t$, which is less than $|x|$ unless $(t,q_0) = (8,2)$. But $|G_2(8)|$ is indivisible by $331$, so this case does not arise and the proof is complete.
\end{proof}

\begin{thm} \label{t:e8}
The conclusion to Theorem~\ref{t:exceptional} holds when $G_0 = E_8(q)$.
\end{thm}

\begin{proof}
We will apply Lemma~\ref{l:prob_method}. Recall that $y \in E_8(q_0)$ and $x$ is a Shintani correspondent of $y$. Let $H \in \mathcal{M}(x)$, so Proposition~\ref{p:e8_max} gives $H \in \mathcal{M}_1$ and Lemma~\ref{l:centraliser_bound} implies that at most $|C_{E_8(q_0)}(y)| = |y|$ distinct $G_0$-conjugates of $H$ are contained in $\mathcal{M}(x)$. Finally, if $z \in G$ is nontrivial then \cite[Theorem~1]{LLS} gives ${\rm fpr}(z,G/H) \leqs q^{-8}(q^4-1)^{-1}$ and therefore Proposition~\ref{p:classes} implies that
\[
P(z,x) \leq \sum_{H \in \mathcal{M}(x)}{\rm fpr}(z,G/H) < (49+\log\log q)\cdot |y| \cdot \frac{1}{q^8(q^4-1)} < \frac{1}{q}, 
\]
noting that $q_0 \leq q^{1/2}$. The result follows.
\end{proof}

\section{Proof of Theorem \ref{t:exceptional}: \texorpdfstring{$G_0 = E_7(q)$}{G0 = E7(q)}} \label{s:e7}

Let $G = \< G_0, g\>$, where $G_0 = E_7(q)$ and $g \in G \setminus G_0$. As usual, write $q=p^f$ with $p$ prime. According to Proposition~\ref{p:cases}, it is enough to prove Theorem~\ref{t:exceptional} for the cases recorded in Table \ref{tab:e7_cases}. In the table, we write $\Delta(f)$ for the set of proper positive divisors of $f$.

\begin{table}
\caption{The relevant groups $G = \la G_0,g \ra$ for $G_0 = E_7(q)$} \label{tab:e7_cases}
\begin{center}
\vspace{-8mm}
\[
\begin{array}{ccll} 
\hline
\text{Case} & g      & \multicolumn{2}{l}{\text{Conditions}} \\ 
\hline    
\text{(a)}  & \d     & \text{$q$ odd} &                      \\
\text{(b)}  & \p^i   &                & i \in \Delta(f)      \\
\text{(c)}  & \d\p^i & \text{$q$ odd} & i \in \Delta(f)      \\
\hline 
\end{array}
\]
\end{center}
\end{table}

\begin{prop} \label{p:e7_a}
The conclusion to Theorem~\ref{t:exceptional} holds in case~(a) of Table~\ref{tab:e7_cases}.
\end{prop}

\begin{proof}
Here $q$ is odd and $G = \< G_0, \d \> =  \Inndiag(G_0)$. Fix an element $x \in G \setminus G_0$ of order $(q+1)(q^6-q^3+1)$. As explained in \cite[Section~4(i)]{Weigel}, $x$ is contained in a unique maximal subgroup of $G$ (namely, a maximal rank subgroup of type ${}^2E_6(q) \times (q+1)$). Therefore \cite[Theorem~1]{LLS} implies that
\[
P(z,x) \leq \sum_{H \in \mathcal{M}(x)} {\rm fpr}(z,G/H) \leq (q^6-q^3+1)^{-1}
\]
for all nontrivial $z \in G$ and the result follows.
\end{proof}

For the remainder of this section, we will assume that we are in cases~(b) and~(c) of Table~\ref{tab:e7_cases}. Therefore, fix a proper divisor $i$ of $f$ and write $e=f/i$ and $q=q_0^e$. Let $\bar{G}$ be the adjoint algebraic group of type $E_7$ over $\bar{\mathbb{F}}_{p}$, let $\s$ be the Steinberg endomorphism $\p^i$ and let 
\[
F\:\Inndiag(E_7(q))\p^i \to \Inndiag(E_7(q_0))
\]
be the Shintani map of $(\bar{G},\s,e)$. 

If $q$ is even, then we are necessarily in case~(b) and we have $F\:E_7(q)g \to E_7(q_0)$, which means that we can proceed as in Section~\ref{s:e8}. The following lemma will allow us to handle cases~(b) and~(c) simultaneously when $q$ is odd.

\begin{lem} \label{l:e7_split}
If $q$ is odd, then the Shintani map $F$ restricts to bijections
\begin{align*}
&\{ (t  \p^i)^{\Inndiag(E_7(q))} \, : \, t \in E_7(q) \} \to \{ y^{\Inndiag(E_7(q_0))} \, : \, y \in E_7(q_0) \} \\
&\{ (t\d\p^i)^{\Inndiag(E_7(q))} \, : \, t \in E_7(q) \} \to \{ y^{\Inndiag(E_7(q_0))} \, : \, y \in \Inndiag(E_7(q_0)) \setminus E_7(q_0) \}.
\end{align*}
\end{lem}

\begin{proof}
This was essentially proved in Example~\ref{ex:shintani_split}. First observe that $E_7(q) = (\bar{G}_{\s^e})'$ and $E_7(q_0) = O^{p'}(\bar{G}_\s) = (\bar{G}_{\s})'$. Let us also note that $\<G_0, \s\> = \< E_7(q), \p^i\>$ is an index two (and hence normal) subgroup of $\<\bar{G}_{\s^e}, \s\> = \<\Inndiag(E_7(q)), \p^i\>$. Therefore, Lemma~\ref{l:shintani_split} implies that the Shintani map $F$ restricts to the bijection
\[
F_1\: \{ (t  \p^i)^{\Inndiag(E_7(q))} \, : \, t \in E_7(q) \} \to \{ y^{\Inndiag(E_7(q_0))} \, : \, y \in E_7(q_0) \},
\]
while the restriction of $F$ to the complement of the domain of $F_1$ is the bijection
\[
F_2\: \{ (t\d\p^i)^{\Inndiag(E_7(q))} \, : \, t \in E_7(q) \} \to \{ y^{\Inndiag(E_7(q_0))} \, : \, y \in \Inndiag(E_7(q_0)) \setminus E_7(q_0) \}.
\]
The result follows.
\end{proof}

Fix an element $y \in \Inndiag(E_7(q_0))$ such that 
\[
|y| = 
\left\{
\begin{array}{ll}
(q_0+1)(q_0^6-q_0^3+1) & \text{if $q_0 > 2$} \\
129                    & \text{if $q_0 = 2$}
\end{array}
\right.
\]
and $C_{{\rm Inndiag}(E_7(q_0))}(y^2) = \la y \ra$ (see \cite{Lubeck}). Let $x \in \<\Inndiag(E_7(q)),\p^i\>$ such that $F(x)$ is $y^2$ in case~(b) and $y$ in case~(c).

If $q$ is even, then we are in case~(b) and we have $y \in E_7(q_0)$ and $x \in G = \< G_0, \p^i\>$. If $q$ is odd, then $y \in \Inndiag(E_7(q_0)) \setminus E_7(q_0)$ and $y^2 \in E_7(q_0)$, so Lemma~\ref{l:e7_split} implies that $x \in G = \< G_0, g\>$ in both cases~(b) and~(c). By Lemma~\ref{l:shintani_order}, if we are in case~(b) with $q$ odd, then $|x| = e|y^2| = \frac{1}{2}e|y|$, whereas $|x| = e|y|$ in every other case (note that $|y|=|y^2|$ if $q$ is even). Let us also note that $|y^2| = 1406$ if $q_0=3$ and $|y^2| \geqs 20165$ if $q_0 \geq 4$.

\begin{prop} \label{p:e7_b_max}
Let $H \in \mathcal{M}(x)$. Then $H \in \mathcal{M}_1$ and $H$ is non-parabolic.
\end{prop}

\begin{proof}
We proceed as in the proof of Proposition~\ref{p:e8_max}. By applying Proposition~\ref{p:orders}, we see that $H \not\in \mathcal{M}_3$. Now assume $H \in \mathcal{M}_2$ and let $S$ be the socle of $H$. We need to consider the possibilities for $S$ described in Theorem~\ref{t:simples}.

First assume that $S = {\rm L}_{3}^{\e}(16)$. Here $|S|$ is indivisible by $43$, so $q_0 \geqs 3$ and consequently ${\rm meo}(H) \leqs 273 < |x|$. Next assume $S = {\rm L}_{2}(t)$ with $t \leqs 388(2,t-1)$, so $t \leqs 2^8$ if $t$ is even. This case is ruled out since Proposition~\ref{p:meo} gives ${\rm meo}(H) \leqs t^2/(t-1)<|x|$. Now assume $S = {}^2B_2(t)$ with $t = 2^{2k+1}$ and $k \leqs 3$. Here ${\rm meo}(H) \leqs 1035$ and we may assume $q_0=2$ and $t=2^7$, but one checks that $|S|$ is indivisible by $43$, so this case does not arise. Similarly, if $S = {}^2G_2(t)'$ with $t = 3^{2k+1}$ and $k \leqs 2$, then ${\rm meo}(H) \leqs 1355<|x|$. 

Now assume that ${\rm rk}(S) \in \{2,3\}$ and $t \leqs 9$. If ${\rm rk}(S)=3$, then ${\rm meo}(H) \leqs t^4/(t-1)$ and we reduce to the case $t=8$ with $q_0=2$, but in every case, one checks that $|S|$ is indivisible by $43$. Finally, let us assume ${\rm rk}(S)=2$. If $S$ is classical, then Proposition~\ref{p:meo} implies that ${\rm meo}(H)<|x|$. If $S = G_2(t)'$, then ${\rm meo}(H) \leqs 8(t^2+t+1)\log t$ and this upper bound is less than $|x|$ unless $q_0=2$ and $t \in \{4,8\}$, but in both cases, $|S|$ is indivisible by $43$.

To complete the proof, let us observe that $y^2 \in E_7(q_0)$ is contained in a unique maximal subgroup of $E_7(q_0)$ (see \cite[Tables~III and~IV]{GK}). In particular, $y$ is not contained in a maximal parabolic subgroup of $E_7(q_0)$, so by applying Corollary~\ref{c:shintani_parabolic}, we deduce that $x$ is not contained in a maximal parabolic subgroup of $G$. 
\end{proof}

\begin{prop} \label{p:e7_b}
The conclusion to Theorem~\ref{t:exceptional} holds in cases~(b) and (c) of Table \ref{tab:e7_cases}.
\end{prop}

\begin{proof}
We proceed as usual, via Lemma~\ref{l:prob_method}. Let $H \in \mathcal{M}(x)$ and let $z \in G$ be nontrivial. Then Proposition~\ref{p:e7_b_max} implies that $H \in \mathcal{M}_1$ and $H$ is not a parabolic subgroup, so \cite[Theorem~2]{LLS} gives ${\rm fpr}(z,G/H) \leqs 2q^{-12}$. Now 
\[
|C_{{\rm Inndiag}(E_7(q_0))}(y^2)| = |y| \leq (q_0+1)(q_0^6-q_0^3+1)
\]
and by applying Proposition~\ref{p:classes} and Lemma~\ref{l:centraliser_bound}, we deduce that
\[
P(z,x) \leq \sum_{H \in \mathcal{M}(x)}{\rm fpr}(z,G/H) < (30+\log\log q)\cdot (q_0+1)(q_0^6-q_0^3+1) \cdot 2q^{-12} < q^{-1}.
\]
The result follows.
\end{proof}

By combining Propositions~\ref{p:e7_a} and~\ref{p:e7_b}, we get the following.

\begin{thm} \label{t:e7}
The conclusion to Theorem~\ref{t:exceptional} holds when $G_0=E_7(q)$.
\end{thm}

\section{Proof of Theorem \ref{t:exceptional}: \texorpdfstring{$G_0 = E_6^{\e}(q)$}{G0 = E6(+/-)(q)}} \label{s:e6}

In this section we study the almost simple groups $G = \< G_0, g \>$, where $G_0 = E_6^\e(q)$ for some sign $\e \in \{+,-\}$. The description of $\Out(G_0)$ in Section~\ref{ss:p_aut} shows that there are several different types of automorphism $g$ that we must consider in order to prove Theorem~\ref{t:exceptional} in this setting. More precisely, in light of Proposition~\ref{p:cases}, it suffices to consider the groups recorded in Table \ref{tab:e6_cases} (as before, we write $\Delta(f)$ for the set of proper positive divisors of $f$). 

\begin{table}
\caption{The relevant groups $G = \la G_0,g \ra$ for $G_0 = E_6^{\e}(q)$} \label{tab:e6_cases}
\begin{center}
\vspace{-8mm}
\[
\begin{array}{ccllll} 
\hline
\text{Case}        & g            & \multicolumn{3}{l}{\text{Conditions}} \\ 
\hline    
\text{(a)}         & \d           & q \equiv \e \mod{3} &                 &                 \\
\text{(b)(i)}      &         \p^i &                     & \e = +          & i \in \Delta(f) \\
\text{(b)(ii)}     &       \g\p^i &                     & \e = (-)^{f/i}  & i \in \Delta(f) \\
\text{(b)(iii)}    & \d^\pm  \p^i & q \equiv 1  \mod{3} & \e = +          & i \in \Delta(f) \\
\text{(b)(iv)}     & \d^\pm\g\p^i & q \equiv \e \mod{3} & \e = (-)^{f/i}  & i \in \Delta(f) \\
\text{(c)(i)}      &         \p^i &                     & \e = -          & i \in \Delta(f) \\
\text{(c)(ii)}     &       \g\p^i &                     & \e = +          & i \in \Delta(f) \ \&\ \text{$f/i$ odd} \\
\text{(d)}         &       \g     &                     &                 &                 \\ 
\hline 
\end{array}
\]
\end{center}
\end{table}

Let us briefly comment on the distinction between cases~(b) and~(c) in Table \ref{tab:e6_cases}. The elements $g$ that arise in these two cases are precisely the automorphisms of $G_0$ that are not contained in $\<{\rm Inndiag}(G_0),\g\>$. One can check that such an automorphism features in case~(b) if and only if 
\[
\<G_0,g\> \cap \<\Inndiag(G_0),\g\> \leq \Inndiag(G_0).
\] 
We will see that Shintani descent applies in the usual way in case~(b), but in case~(c) we need to apply Lemma~\ref{l:shintani_substitute} (see Example~\ref{ex:shintani_substitute}, which contrasts cases~(b)(ii) and~(c)(ii) when $\e=+$ and $p=3$). 

Recall that Remark~\ref{r:cases_e6} (in particular, the notes labelled $\star$ and~$\dagger$) permits us to omit some of the cases in Table \ref{tab:e6_cases} if certain conditions on $p$, $f$ and $i$ are satisfied. We will consider cases~(a)--(d) in Sections~\ref{ss:e6_a}--\ref{ss:e6_d}, respectively.

It will be useful to note that if $G$ is any almost simple group with socle $E_6^{\e}(q)$, then 
\cite[Theorem~1]{LLS} gives
\begin{equation} \label{e:e6_fpr}
{\rm fpr}(z,G/H) \leq 
\left\{
\begin{array}{ll}
(q^4-q^2+1)^{-1} & \text{if $\e=+$ } \\
(q^6-q^3+1)^{-1} & \text{if $\e=-$}
\end{array}
\right.
\end{equation}
for all $H \in \mathcal{M}$ and all nontrivial $z \in G$.

\subsection{Case (a): diagonal automorphisms} \label{ss:e6_a}

We begin by handling the case where $g$ is a diagonal automorphism.

\begin{prop} \label{p:e6_a}
The conclusion to Theorem~\ref{t:exceptional} holds in case~(a) of Table \ref{tab:e6_cases}.
\end{prop}

\begin{proof}
Here $q \equiv \e \mod{3}$ and $G = \< G_0, \d \> =  \Inndiag(G_0)$. Fix an element $x \in G \setminus G_0$ of order $q^6 + \e q^3+1$. By \cite[Sections~4(g) and~(h)]{Weigel}, $x$ is contained in a unique maximal subgroup of $G$ (namely, a subgroup of type ${\rm SL}^\e_3(q^3).3$). Therefore, with the bound in \eqref{e:e6_fpr}, we get
\[
P(z,x) \leq \sum_{H \in \mathcal{M}(x)} {\rm fpr}(z,G/H) \leq (q^4-q^2+1)^{-1}
\]
for all nontrivial $z \in G$ and the result follows.
\end{proof}

\subsection{Case (b): Shintani descent} \label{ss:e6_b}

Here we consider cases (b)(i)--(b)(iv) in Table \ref{tab:e6_cases}. Fix a proper divisor $i$ of $f$ and write $e=f/i$ and $q=q_0^e$. Recall that in cases~(iii) and~(iv), $e$ is even if $\e=+$ and $e$ is odd if $\e=-$. Let $\bar{G}$ be the adjoint algebraic group of type $E_6$ over $\bar{\mathbb{F}}_{p}$ and define
\[
(\s,\eta) =
\left\{
\begin{array}{ll}
(\p^i,  \, +) & \text{in cases~(i)  and~(iii)} \\
(\g\p^i,\, -) & \text{in cases~(ii) and~(iv).}
\end{array}
\right.
\] 
Notice that $\e = \eta^e$ and $\bar{G}_{\s} = {\rm Inndiag}(E_6^{\eta}(q_0))$. Let 
\[
F\:\Inndiag(E^\e_6(q))\s \to \Inndiag(E^\eta_6(q_0))
\]
be the Shintani map of $(\bar{G},\s,e)$. The following result is the analogue of Lemma~\ref{l:e7_split}.

\begin{lem} \label{l:e6_b_split}
If $q_0 \equiv \eta \mod{3}$, then the Shintani map $F$ restricts to bijections
\begin{align*}
&\{ (t  \s)^{\Inndiag(E^\e_6(q))} \, : \, t \in E^\e_6(q) \} \to \{ y^{\Inndiag(E^\eta_6(q_0))} \, : \, y \in E^\eta_6(q_0) \} \\
&\{ (t\d\s)^{\Inndiag(E^\e_6(q))} \, : \, t \in E^\e_6(q) \} \to \{ y^{\Inndiag(E^\eta_6(q_0))} \, : \, y \in \Inndiag(E^\eta_6(q_0)) \setminus E^\eta_6(q_0) \}.
\end{align*}
\end{lem}

\begin{proof}
By hypothesis $q_0 \equiv \eta \mod{3}$, which implies that $q = q_0^e \equiv \eta^e = \e \mod{3}$ and thus
\[
|\Inndiag(E_6^\e(q)):E_6^\e(q)|=|\Inndiag(E_6^\eta(q_0)):E_6^\eta(q_0)|=3.
\] 
We have already noted that $\bar{G}_{\s^e} = \Inndiag(E^\e_6(q))$ and $\bar{G}_\s = \Inndiag(E^\eta_6(q_0))$. Also observe that $E^\e_6(q) = (\bar{G}_{\s^e})'$ and $E^\eta_6(q_0) = O^{p'}(\bar{G}_\s) = (\bar{G}_{\s})'$. Therefore, in order to apply Lemma~\ref{l:shintani_split} it remains to check that $\<G_0,\s\> \leqn \<\bar{G}_{\s^e},\s\>$. In cases~(i) and (iii) we have 
$\s=\p^i$ and $\e=\eta=+$, so $p^i = q_0 \equiv 1 \mod{3}$. From \eqref{e:out_e6_facts} we see that $[\ddot{\p}^i,\ddot{\d}] = \ddot{\d}^{p^i-1}  = 1$, whence $\<\ddot{\s}\> \leqn \<\ddot{\s},\ddot{\d}\> = \<\bar{G}_{\s^e},\s\>/G_0$ and consequently $\<G_0,\s\> \leqn \<\bar{G}_{\s^e},\s\>$. Similarly, $\s=\g\p^i$ and $p^i = q_0 \equiv 2 \mod{3}$ in cases~(ii) and (iv). Here \eqref{e:out_e6_facts} and \eqref{e:out_2e6_facts} give $[\ddot{\g}\ddot{\p}^i,\ddot{\d}] = \ddot{\d}^{-p^i-1} = 1$, so we again obtain $\<\ddot{\s}\> \leqn \<\ddot{\s},\ddot{\d}\>$ and $\<G_0,\s\> \leqn \<\bar{G}_{\s^e},\s\>$. By  applying Lemma~\ref{l:shintani_split}, we see that the Shintani map $F$ restricts to the bijections in the statement.
\end{proof}

Fix a regular semisimple element $y \in \Inndiag(E^\eta_6(q_0))$ such that
\[
|y| = q_0^6 + \eta q_0^3 + 1
\]
and 
\[
C_{{\rm Inndiag}(E^\eta_6(q_0))}(y^3) = \la y \ra
\]
(see \cite{Lubeck}). Choose $x \in \<\Inndiag(E^\e_6(q)),\p^i\>$ such that 
\[
F(x)=
\left\{
\begin{array}{ll}
y^3 & \text{in cases~(i)   and (ii) } \\
y   & \text{in cases~(iii) and (iv).} \\
\end{array}
\right.
\]
If $q \not\equiv \e \mod{3}$, then we are in case~(i) or~(iii) and we have $y \in E^\eta_6(q_0)$ and $x \in G = \< G_0, g\>$. If $q_0 \equiv \eta \mod{3}$ (so $q \equiv \e \mod{3}$), then $y \in \Inndiag(E_6^{\eta}(q_0)) \setminus E_6^{\eta}(q_0)$ and $y^3 \in E_6^{\eta}(q_0)$, so Lemma~\ref{l:e6_b_split} implies that $x \in G = \< G_0, g\>$ once again. Finally, suppose that $q \equiv \e \mod{3}$ and $q \not\equiv \eta \mod{3}$. Here the notes~$\star$ and $\dagger$ in Remark~\ref{r:cases_e6} imply that we only need to consider one automorphism from $\{ \p^i, \d\p^i, \d^2\p^i \}$, so the fact that $x \in \< \Inndiag(E_6^\e(q)), \p^i\>$ is enough to ensure that $x \in \< G_0, g \>$ for a suitable choice of $g$.

By Lemma~\ref{l:shintani_order}, we note that
\[
|x| = 
\left\{
\begin{array}{ll}
e|y^3| & \text{in cases~(i)   and (ii) } \\
e|y|   & \text{in cases~(iii) and (iv).} \\
\end{array}
\right.
\]
Let us also note that if $\eta=+$, then $|y^3|=73$ for $q_0=2$, $|y^3| = 757$ for $q_0=3$, $|y^3|= 1387$ for $q_0=4$ and $|y^3| \geq 15751$ for $q_0 \geq 5$, so $|y^3|$ is prime in the first two cases. Similarly, if $\eta=-$ then $|y^3|=19$ for $q_0=2$, $|y^3| = 703 = 19 \cdot 37$ for $q_0=3$ and $|y^3| \geq 4033$ for $q_0 \geq 4$.

\begin{prop} \label{p:e6_b_max}
Let $H \in \mathcal{M}(x)$. Then $H \in \mathcal{M}_1$ and $H$ is non-parabolic.
\end{prop}

\begin{proof}
Let us begin by noting that $y^3$ is not contained in a maximal parabolic subgroup of $E_6^\eta(q_0)$ (in fact, $y^3$ is contained in a unique maximal subgroup of $E_6^\eta(q_0)$; see  \cite[Sections~(g) and~(h)]{Weigel}), so Corollary~\ref{c:shintani_parabolic} implies that $H$ is non-parabolic.

Seeking a contradiction suppose that $H \in \mathcal{M}_2 \cup \mathcal{M}_3$ with socle $S$. First assume that $H \in \mathcal{M}_3$. Then the bound on ${\rm meo}(H)$ in Proposition~\ref{p:orders} implies that $\eta=-$ and $q_0=2$. Here $|y|=19$ and by inspecting each candidate for $S$ in \cite{LS99}, it is easy to check that ${\rm Aut}(S) \setminus S$ does not contain an element of order divisible by $19$. 

Now suppose $H \in \mathcal{M}_2$. First assume that $\eta=+$. If $S = {\rm L}_{3}^{\e'}(16)$, then $|S|$ is indivisible by $73$, so $q_0 \geqs 3$ and Proposition~\ref{p:meo} implies that ${\rm meo}(H) < |x|$. If $S={\rm L}_{2}(t)$ with $t \leqs 124\,(2,t-1)$, then ${\rm meo}(H) \leqs t^2/(t-1)< |x|$.  If $S = {}^2B_2(t)$ with $t=2^{2k+1}$ then $k \leqs 2$ and ${\rm meo}(H) \leqs 205$, so $q_0=2$ is the only option, but in both cases we find that $|S|$ is indivisible by $73$. If $S = {}^2G_2(t)'$, then Proposition~\ref{p:meo} implies that ${\rm meo}(H)<|x|$.

Next assume that ${\rm rk}(S) = 3$ and $t \leqs 9$ (we continue to assume that $\eta = +$). Here the bound on ${\rm meo}(H)$ from Proposition~\ref{p:meo} gives an immediate reduction to the case $t=8$ with $q_0=2$. The cases $S = {\rm L}_{4}(8)$ and ${\rm PSp}_{6}(8)$ are ruled out by Proposition~\ref{p:sp} and one checks that $|{\rm U}_{4}(8)|$ is indivisible by $73$. Finally, suppose ${\rm rk}(S)=2$. For classical $S$, the bound in Proposition~\ref{p:meo} is sufficient. Similarly, if $S = G_2(t)'$ then by applying Proposition~\ref{p:meo} we reduce to the cases where $(t,q_0)$ is one of $(4,2)$, $(8,2)$ or $(9,3)$. Here $|G_2(4)|$ and $|G_2(9)|$ are indivisible by $73$ and $757$ respectively, and the case $S = G_2(8)$ is ruled out by Proposition~\ref{p:sp}.

Finally, let us assume $\eta=-$. We proceed as before, first noting that the cases where $S$ is one of ${\rm L}_{3}^{\e'}(16)$, ${\rm L}_{2}(t)$, ${}^2B_2(t)$ and ${}^2G_2(t)'$ present no difficulties. Suppose ${\rm rk}(S)=3$ and $t \leqs 9$. Here $S$ is classical and by applying the bound on ${\rm meo}(H)$ in Proposition~\ref{p:meo}, we reduce to the case $q_0=2$ with $t \in \{4,8\}$. One checks that $|{\rm L}_{4}^{\e'}(4)|$ and $|{\rm L}_{4}(8)|$ are indivisible by $19$, so these options are ruled out. For $S = {\rm U}_{4}(8)$, with the aid of {\sc Magma} \cite{magma}, we find that there are no elements in ${\rm Aut}(S) \setminus S$ of order divisible by $19$ and so this possibility is also eliminated.

Now assume that ${\rm rk}(S)=2$. If $S = {\rm L}_{3}(t)$, then Proposition~\ref{p:meo} implies that ${\rm meo}(H) < |x|$ unless $t=8$ and $q_0=2$, but this case does not arise since $|{\rm L}_3(8)|$ is indivisible by $19$. Similar reasoning handles the case $S = {\rm PSp}_{4}(t)$. For $S = G_2(t)'$, the bound coming from Proposition~\ref{p:meo} is effective unless $(t,q_0)$ is one of $(4,2)$, $(8,2)$ or $(9,3)$. We can rule out the latter case since $|G_2(9)|$ is indivisible by $37$. Similarly, $|G_2(4)|$ is indivisible by $19$ and the case $S = G_2(8)$ is eliminated by applying Proposition~\ref{p:sp}. 

Finally, suppose that $S = {\rm U}_{3}(t)$. As before, by applying Proposition~\ref{p:meo} we reduce to the case $t=8$ with $q_0=2$. Using {\sc Magma}, we find that there are elements in ${\rm Aut}(S) \setminus S$ of order $19m$ for some positive integer $m$ if and only if $m=3$, so we must have $G_0 = {}^2E_6(8)$ and $g = \p$ (so $e=3$ and $|x| = e|y^3| = 3(2^6-2^3+1)/(2+1,3) = 57$). To resolve this case, we appeal to \cite[Theorem~1.2]{CravenMedium}, which states that $S$ is \emph{strongly imprimitive} in the ambient simple algebraic group $\bar{G} = E_6$ (also see \cite[Proposition~10.2]{CravenMedium}). In the context of Theorem~\ref{t:types}, this means that every almost simple subgroup of $G$ with socle ${\rm U}_3(8)$ is contained in a type (I) maximal subgroup of $G$, so this case does not arise.
\end{proof}

\begin{prop} \label{p:e6_b}
The conclusion to Theorem~\ref{t:exceptional} holds in case~(b) of Table \ref{tab:e6_cases}.
\end{prop}

\begin{proof}
Let $H \in \mathcal{M}(x)$ and let $z \in G$ be nontrivial. By Proposition~\ref{p:e6_b_max}, $H \in \mathcal{M}_1$ and $H$ is non-parabolic, so \cite[Theorem~2]{LLS} gives ${\rm fpr}(z,G/H) \leqs 2q^{-6}$. Since 
\[
|C_{{\rm Inndiag}(E^\eta_6(q_0))}(y^3)| = |y| = q_0^6 + \eta q_0^3 + 1,
\]
we see that Proposition~\ref{p:classes} and Lemma~\ref{l:centraliser_bound} imply that
\[
P(z,x) \leq \sum_{H \in \mathcal{M}(x)}{\rm fpr}(z,G/H) < (19+\log\log q)\cdot (q_0^6 + \eta q_0^3 + 1) \cdot 2q^{-6}.
\]
If $q>7$, then the upper bound is at most $q^{-1}$, and for $q > 4$ it is less than $\frac{1}{2}$. Finally, if $q=4$ then we find that there are at most $14$ conjugacy classes of subgroups in $\mathcal{M}(x)$ (for example, there are no exotic locals). Replacing the $19+\log\log q$ factor by $14$ in the above bound shows that $P(z,x) < \frac{1}{2}$, which completes the proof.
\end{proof}

\subsection{Case (c): Shintani descent over \texorpdfstring{$F_4$}{F4}} \label{ss:e6_c}

We now turn to case~(c) in Table \ref{tab:e6_cases}, which involves two subcases labelled (i) and (ii). Here we can apply Shintani descent in the indirect manner encapsulated in Lemma~\ref{l:shintani_substitute} (compare with the method adopted in \cite[Sections~5.4.2 and~6.4.2]{HarperClassical}).
 
Fix a proper divisor $i$ of $f$ and write $e=f/i$ and $q=q_0^e$. Recall that in case~(ii) we have $\e=+$ and $e$ is odd, whereas $\e=-$ in case~(i) (and there is no parity condition on $e$). Let $\bar{G}$ be the adjoint algebraic group $E_6$ over $\bar{\mathbb{F}}_p$ and let $\s$ be the Steinberg endomorphism $\p^i$ in case~(i) and $\g\p^i$ in case~(ii). Observe that $\bar{G}_{\g\s^e} = \Inndiag(G_0)$ in both cases. Let $\bar{K} = C_{\bar{G}}(\g) = F_4$ and note that $\bar{K}$ is $\s$-stable. Choose $y \in \bar{K}_\s = F_4(q_0) \leq  E^{-\e}_6(q_0)$ such that
\[
|y| = q_0^4-q_0^2+1
\] 
and $C_{\bar{K}_{\s}}(y) = \la y \ra$. Here $\gamma$ is an algebraic automorphism of $\bar{G}$ of order $2$, so by Lemma~\ref{l:shintani_substitute}(i), there exists $x \in \bar{K}_{\s^e}\s = F_4(q)g \subseteq E^\e_6(q)g$ such $x^e$ is $\bar{G}$-conjugate to $y\g$. In addition, since $|y^2|$ is odd, we note that $|x| = 2e|y|$ (see Remark~\ref{r:shintani_order}). 

\begin{prop} \label{p:e6_c_max}
Let $H \in \mathcal{M}(x)$. Then $H \in \mathcal{M}_1$ and $H$ is non-parabolic.
\end{prop}

\begin{proof}
By considering the order of $y^2$, we observe that $y^2$ is not contained in a maximal parabolic subgroup of $E_6^{-\e}(q_0)$. Therefore, Lemma~\ref{l:shintani_substitute}(ii)(b) implies that there are no parabolic subgroups in $\mathcal{M}(x)$.

For the remainder, let us assume $H \in \mathcal{M}_2 \cup \mathcal{M}_3$ has socle $S$. First assume $H \in \mathcal{M}_2$, noting that the possibilities for $S$ are described in Theorem~\ref{t:simples}. Suppose $S = {\rm L}_{3}(16)$. Here $p=2$ and ${\rm meo}(H) \leqs 273$, so $q_0=2$ is the only possibility and thus $|x| = 26e$, but one checks that there are no elements of order $26e$ (with $e \geqs 2$) in ${\rm Aut}(S) \setminus S$, so this case does not arise. A very similar argument rules out $S = {\rm U}_{3}(16)$. Next assume $S = {\rm L}_{2}(t)$ with $t \leqs 124\,(2,t-1)$. By applying Proposition~\ref{p:meo} we reduce to $q_0 \in \{2,3\}$, but we find that there are no elements of order $2e|y|$ in ${\rm Aut}(S) \setminus S$. Very similar reasoning eliminates the cases where $S$ is either ${}^2B_2(t)$ or ${}^2G_2(t)'$. 

To complete the analysis of the candidates in $\mathcal{M}_2$, we may assume $S$ is a simple group of Lie type over $\mathbb{F}_t$ with ${\rm rk}(S) \in \{2,3\}$ and $t \leqs 9$. Suppose ${\rm rk}(S)=3$, so Proposition~\ref{p:meo} gives ${\rm meo}(H) \leqs t^4/(t-1)$. This bound reduces the problem to a handful of possibilities with $q_0 \in \{2,3\}$, and apart from the cases where $S$ is one of ${\rm L}_{4}(8)$, ${\rm U}_4(8)$ and ${\rm U}_{4}(9)$, one checks that there are no elements in ${\rm Aut}(S) \setminus S$ with order divisible by $|y|$. 

To handle the three special cases, we proceed as follows. First assume $S = {\rm L}_{4}^{\e'}(8)$, so $q_0=2$, $|y|=13$ and one checks that there are elements in ${\rm Aut}(S) \setminus S$ of order $26e$ if and only if $e=5$, so $q=2^5$ and $\e=+$. Now $13$ is a primitive prime divisor of $q^{12}-1$ and by inspecting \cite{Lubeck} we see that $C_{G_0}(y)$ is either $C_{q^4-q^2+1} \times C_{q^2+q+1}$ or ${}^3D_4(q) \times C_{q^2+q+1}$. In particular, $|C_{G_0}(y)|$ is not divisible by $5$. However, the centraliser of any element in $S$ of order $13$ has order $65m$, where $m=9$ if $\e'=+$ and $m=7$ if $\e'=-$. This is clearly a  contradiction since $|C_S(y)|$ must divide $|C_{G_0}(y)|$. A very similar argument rules out the case $S = {\rm U}_4(9)$; here $q_0=3$, $e=5$ and $\e=+$. Moreover, $|y|=73$ is a primitive prime divisor of $q^{12}-1$. Therefore, the possibilities for $C_{G_0}(y)$ are as described above and in both cases we see that $|C_{G_0}(y)|$ is indivisible by $5$. However, this is incompatible with the fact that every element in $S$ of order $73$ is centralised by an element of order $5$.

Next assume that ${\rm rk}(S)=2$. If $S$ is classical, then Proposition~\ref{p:meo} gives ${\rm meo}(H) \leqs t^3/(t-1)$ and this reduces the problem to $t=8$ with $q_0=2$, but in each case, one checks that there are no elements in ${\rm Aut}(S) \setminus S$ of order $26e$. Similarly, if $S = G_2(t)'$, then the bound in Proposition~\ref{p:meo} is sufficient unless $(t,q_0)$ is $(4,2)$ or $(9,3)$, or if $t=8$ and $q_0 \in \{2,4\}$. For the cases with $q_0=2$, it is easy to see that there are no elements in ${\rm Aut}(S) \setminus S$ of order $26e$. We can rule out $S = G_2(9)$ since Proposition~\ref{p:sp} gives ${\rm meo}({\rm Aut}(S) \setminus S) = 36$. Similarly, if $S = G_2(8)$ and $q_0=4$, then $|y|=241$ and we note that $|{\rm Aut}(S)|$ is indivisible by $241$.

Finally, let us assume $H \in \mathcal{M}_3$. By Proposition~\ref{p:orders} we have ${\rm meo}(H) \leqs 60$, so we immediately reduce to the case $q_0=e=2$. Here $|x|=52$ and by inspecting the possibilities for $S$ recorded in \cite{LS99}, it is straightforward to check that there are no elements in ${\rm Aut}(S) \setminus S$ of order $52$.
\end{proof}

\begin{lem} \label{l:e6_c_centraliser}
We have $C_{\Inndiag(E_6^{-\e}(q_0))}(y) = C_{q_0^4-q_0^2+1} \times C_{q_0^2-\e q_0+1}$.
\end{lem}

\begin{proof}
The centralisers of semisimple elements in $\Inndiag(E_6^{-\e}(q_0))$ are listed in \cite{Lubeck}. For a contradiction, suppose that the centraliser is not the one given in the statement. In terms of divisibility, we see that the only other possibility is ${}^3D_4(q_0) \times C_{q_0^2-\e q_0+1}$. In this case, if $\bar{G}=E_6$ and $\bar{H}=F_4$ are the corresponding algebraic groups, then $C_{\bar{G}}(y) = D_4T_2$ and, as explained in the proof of \cite[Lemma~5.4]{LLS}, this implies that $C_{\bar{H}}(y) = B_3T_1$. But $y$ is a regular semisimple element of $\bar{H}$, so this is a contradiction. 
\end{proof}

\begin{prop} \label{p:e6_c}
The conclusion to Theorem~\ref{t:exceptional} holds in case~(c) of Table \ref{tab:e6_cases}.
\end{prop}

\begin{proof}
Let us first observe that since $\<y^2\> = \<y\>$, Lemma~\ref{l:e6_c_centraliser} implies that
\[
C_{\Inndiag(E_6^{-\e}(q_0))}(y^2) = C_{q_0^4-q_0^2+1} \times C_{q_0^2-\e q_0+1}.
\]
Now let $H \in\mathcal{M}(x)$ and let $z \in G$ be nontrivial. Let $a(q)$ be the upper bound on ${\rm fpr}(z,G/H)$ given in \eqref{e:e6_fpr}. Then by combining Propositions \ref{p:classes} and \ref{p:e6_c_max}, noting that there are no parabolic subgroups in $\mathcal{M}(x)$, and using Lemma~\ref{l:shintani_substitute}(b)(i), we deduce that
\[
P(z,x) \leq \sum_{H \in \mathcal{M}(x)}{\rm fpr}(z,G/H) \leqs (19+\log\log q) \cdot (q_0^4-q_0^2+1)(q_0^2+q_0+1) \cdot a(q).
\]
Recalling that $e \geq 3$ if $\e=+$, one checks that this bound is always less than $\frac{1}{2}$ and also less than $q^{-1}$ for $q >27$.
\end{proof}

\subsection{Case (d): involutory graph automorphisms} \label{ss:e6_d}

To complete the proof of Theorem \ref{t:exceptional} for $G_0 = E_6^{\e}(q)$, we may assume that $G = \< G_0, g \>$, where $g$ is the graph automorphism $\g$ in Definition \ref{d:aut}(iii). In particular, since $g$ does not arise from a Steinberg endomorphism of the ambient algebraic group $\bar{G}$, we cannot use Shintani descent in this case and a different approach is required.

Choose $y \in C_{G_0}(g)= F_4(q)$ such that  
\[
|y| = q^4-q^2+1
\] 
and $C_{F_4(q)}(y) = \la y \ra$. Set $x = yg \in G$ and note that $x^2 = y^2 \in G_0$ and $|x|=2|y|$ since $|y|=q^4-q^2+1$ is odd. In addition, $|y| \equiv 1 \mod{3}$ and we note that $|y|$ is divisible by a primitive prime divisor of $q^{12}-1$, so $|y|$ divides $q^i-1$ if and only if $i$ is divisible by $12$. It will also be useful to observe that $|y|$ is $13$, $73$, $241$, $601$ (all of which are prime) when $q$  is $2$, $3$, $4$, $5$, respectively, and $|y| \geqs 2353$ when $q \geqs 7$. 

Since $\mathcal{M}(x) \subseteq \mathcal{M}(y)$, we will focus on determining the subgroups in $\mathcal{M}(y)$ and we proceed by considering the cases arising in Theorem~\ref{t:types}. It will be convenient to handle the cases $q=2$ and $q>2$ separately.

\begin{prop} \label{p:e6_d_2_max}
Assume that $q=2$ and let $H \in \mathcal{M}(x)$.
\begin{itemize}\addtolength{\itemsep}{0.2\baselineskip}
\item[{\rm (i)}]  If $\e=+$, then $H$ has type $F_4(2)$ or ${}^3D_4(2) \times 7$.
\item[{\rm (ii)}] If $\e=-$, then $H$ has type $F_4(2)$ or ${\rm SO}_{7}(3)$. 
\end{itemize}
\end{prop}

\begin{proof}
First note that $|y|=13$ and $|x|=26$. Suppose $\e=+$, so $G = {\rm Aut}(E_6(2)) = E_6(2).2$. In \cite{KW}, the maximal subgroups of $G$ are determined up to conjugacy and it is easy to read off the subgroups with order divisible by $13$, giving the two cases recorded in part (i). Similar reasoning applies when $\e=-$ and $G = {}^2E_6(2).2$, using the list of maximal subgroups presented in the {\sc Atlas} \cite{ATLAS} (also see \cite{Wilson2E62}). In the latter case, note that $|{\rm Fi}_{22}|$ is divisible by $13$, but ${\rm Fi}_{22}{:}2$ does not contain any elements of order $26$.
\end{proof}

\begin{prop} \label{p:e6_d_max}
If $q > 2$ and $H \in \mathcal{M}(x)$, then $H$ has type $F_4(q)$ or ${}^3D_4(q) \times (q^2+\e q +1)$.
\end{prop}

\begin{proof}
First assume $H \in \mathcal{M}_1$. Suppose $H$ is of type (I) in Theorem~\ref{t:types}, so $H = N_G(\bar{H}_{\s})$ for some $\s$-stable closed subgroup $\bar{H}$ of $\bar{G}$. If $\bar{H}$ is parabolic, then $H$ is of type $P_{1,6}$, $P_2$, $P_{3,5}$ or $P_4$ (since $H$ is normalised by a  graph automorphism) and in each case it is straightforward to check that $|H|$ is indivisible by $|y|=q^4-q^2+1$. In the same way, by carefully inspecting \cite{LSS}, we deduce that the only candidate maximal rank subgroups in $\mathcal{M}(x)$ are of type ${}^3D_4(q) \times (q^2+\e q +1)$. In addition, if the rank of $\bar{H}$ is less than $6$, then $H \cap G_0 = F_4(q)$ is the only option and it is easy to see that there are no subgroups in $\mathcal{M}(x)$ of type (II) or (III). 

To complete the proof, let us suppose that $H \in \mathcal{M}_2 \cup \mathcal{M}_3$ with socle $S$. If $H \in \mathcal{M}_3$, then Proposition~\ref{p:orders} implies that ${\rm meo}(H) \leqs 60$, which is a contradiction since $|x| \geqs 146$. Now assume $H \in \mathcal{M}_2$, so the possibilities for $S$ are described in Theorem~\ref{t:simples}. Here we use the bound on ${\rm meo}(H)$ from Proposition~\ref{p:meo} to reduce the problem to a handful of cases with $q \in \{3,4\}$. In each of these cases, one checks that ${\rm Aut}(S)$ contains an element of order $|x|$ if and only if $q=3$ and $S$ is either ${\rm U}_4(9)$ or $G_2(9)$. The latter case is ruled out by Proposition~\ref{p:sp} since ${\rm meo}({\rm Aut}(S) \setminus S) = 36 < |x|$, so let us assume $S = {\rm U}_{4}(9)$. Here $|y|=73$ and ${\rm Aut}(S)\setminus S$ does contain elements of order $|x|=146$, but we can eliminate this case by arguing as follows. By Lemma~\ref{l:e6_c_centraliser}, we have $C_{G_0}(y) = C_{73m}$, where $m=13$ if $\e=+$ and $m=7$ if $\e=-$. However, every element in $S$ of order $73$ commutes with an element of order $5$, which is a contradiction since $|C_S(y)|$ must divide $|C_{G_0}(y)|$.
\end{proof}

We will need the next two lemmas, which, in some special cases of interest, give slightly stronger fixed point ratio estimates than the relevant bounds presented in \cite{LLS}.

\begin{lem} \label{l:e6_d_fpr1}
Let $H$ be a maximal subgroup of $G$ of type ${}^3D_4(q) \times (q^2+\e q+1)$ and let $z \in G$ be nontrivial. Then
\[
{\rm fpr}(z,G/H) \leqs 2q^{-6}.
\]
\end{lem}

\begin{proof}
In view of \cite[Theorem~2]{LLS}, we may assume that $\e=+$ and $z \in G$ is a graph automorphism (here \cite{LLS} only gives ${\rm fpr}(z,G/H) \leqs (q^4-q^2+1)^{-1}$). By inspecting the proof of \cite[Lemma~3.10]{Bur18}, we see that ${\rm fpr}(z,G/H) < q^{-6}$ if $q \geqs 3$. Finally, if $q=2$, then $G = G_0.2 = {\rm Aut}(G_0)$,
\[
H = ({}^3D_4(2) \times D_{14}){:}3, \quad H \cap G_0 = ({}^3D_4(2) \times 7){:}3
\]
(see \cite[Table~1]{KW}, for example) and 
\[
|z^G \cap H| \leqs i_2(H \setminus H \cap G_0) = 487312, \quad |z^G| \geqs |E_6(2):F_4(2)| = 64884736.
\]
The result follows.
\end{proof}

\begin{lem} \label{l:e6_d_fpr2}
Let $K$ be a maximal subgroup of $G$ of type $F_4(q)$ and let $z \in G$ be nontrivial. Then
\[
{\rm fpr}(z,G/K) \leqs (q^6-q^3+1)^{-1}.
\]
\end{lem}

\begin{proof}
Here $K = C_G(g) = F_4(q) \times \la g \ra$ and as in the previous lemma, we may assume $\e=+$ and $z$ is an involutory graph automorphism. Note that each element in $z^G \cap K$ is of the form $sg$, where $s \in F_4(q)$ satisfies $s^2=1$.

First assume that $C_{G_0}(z) = F_4(q)$. As explained in the proof of \cite[Lemma~5.4]{LLS}, if $p \neq 2$, then
\[
|z^G \cap K| = 1+ |x_1^{F_4(q)}| = 1+q^8(q^8+q^4+1)
\]
and $x_1 \in F_4(q)$ is an involution with $C_{F_4}(x_1) = B_4$. Similarly, if $p = 2$, then  
\[
|z^G \cap K| = 1+ |x_2^{F_4(q)}| = 1+ (q^4+1)(q^{12}-1),
\]
where $x_2 \in F_4(q)$ is a short root element. 

Now assume that $C_{G_0}(z) \ne F_4(q)$. Suppose $p \ne 2$, so $|C_{G_0}(z)| = |{\rm Sp}_{8}(q)|$. The group $F_4(q)$ has two classes of involutions and from the previous paragraph we deduce that $z^G \cap K$ is the set of involutions in $K$ of the form $sg$ with $C_{F_4}(s) = A_1C_3$. Therefore,  
\[
|z^G \cap K| = \frac{|F_4(q)|}{|{\rm SL}_{2}(q)||{\rm Sp}_{6}(q)|} = q^{14}(q^6+1)(q^4+q^2+1)(q^4+1).
\]
For $p=2$, we have 
\[
|C_{G_0}(z)| = |C_{F_4(q)}(t)| = q^{24}(q^2-1)(q^4-1)(q^6-1),
\]
where $t \in F_4(q)$ is a long root element, and the proof of \cite[Lemma~5.4]{LLS} gives
\[
|z^G \cap K| = (q^4+1)(q^{12}-1)+(q^4+1)(q^6-1)(q^{12}-1)+q^4(q^4+q^2+1)(q^8-1)(q^{12}-1).
\]
In every case, the desired bound holds.
\end{proof}

We can now handle the case $q>2$.

\begin{prop} \label{p:e6_d}
The conclusion to Theorem~\ref{t:exceptional} holds in case~(d) of Table \ref{tab:e6_cases} if $q>2$.
\end{prop}

\begin{proof}
Let $H$ and $K$ be maximal subgroups of $G$ of type ${}^3D_4(q) \times (q^2+\e q+1)$ and $F_4(q)$, respectively, and note that there is a unique $\Inndiag(G_0)$-class of each type of subgroup. By Proposition~\ref{p:e6_d_max}, each subgroup in $\mathcal{M}(x)$ is conjugate to either $H$ or $K$. 

For any maximal subgroup $M \leq G$, let $n(M)$ be the number of conjugates of $M$ that contain $y$, and note that
\[
n(M) = \frac{|y^G \cap M|}{|y^G|}\cdot \frac{|G|}{|M|}.
\]

First consider $n(H)$. Given the structure of $H$, we see that every element in $H$ of order $q^4-q^2+1$ is contained in the subgroup $L = {}^3D_4(q)$ (indeed, note that $q^4-q^2+1$ and $q^2+\e q+1$ are coprime for all $q$). Each $z \in L$ of order $q^4-q^2+1$ is self-centralising and by inspecting \cite[Table~4.4]{DM} we deduce that $L$ has $\frac{1}{4}q^2(q^2-1)$ distinct classes of semisimple elements with centraliser a cyclic maximal torus of order $q^4-q^2+1$. Therefore, 
\[
|y^G \cap H| \leqs \frac{1}{4}q^2(q^2-1) \cdot \frac{|{}^3D_4(q)|}{q^4-q^2+1}
\]
and this yields $n(H) \leqs \frac{1}{12}q^2(q^2-1)$.

Now let us turn to $n(K)$. By inspecting \cite{Shinoda74, Shoji}, we see that there are precisely $\frac{1}{12}q^2(q^2-1)$ regular semisimple classes in $F_4(q)$ with centraliser a torus of order $q^4-q^2+1$. Therefore,
\[
|y^G \cap K| \leqs \frac{1}{12}q^2(q^2-1) \cdot \frac{|F_4(q)|}{q^4-q^2+1}
\]
and we deduce that  
\[
n(K) \leqs \frac{1}{12}q^2(q^2-1)(q^2+\e q+1).
\]
Alternatively, we can bound $|y^G \cap K|$ by arguing as in the proof of \cite[Lemma~4.5]{LLS}, noting that $|W(E_6):W(F_4)| = 45$ (where $W(X)$ denotes the Weyl group of $X$). Indeed, it follows that $y^G \cap K$ is a union of at most $45$ distinct $K$-classes, so
\[
|y^G \cap K| \leqs 45\cdot \frac{|F_4(q)|}{q^4-q^2+1}
\]
and subsequently $n(K) \leqs 45(q^2+\e q+1)$, which is a better bound for $q \geqs 5$.

Finally, using the bounds in Lemmas \ref{l:e6_d_fpr1} and \ref{l:e6_d_fpr2}, we deduce that
\[
\sum_{H \in \mathcal{M}(x)}{\rm fpr}(z,G/H) \leqs \frac{1}{12}q^2(q^2-1)\cdot 2q^{-6} + a(q) \cdot(q^2+\e q+1)\cdot (q^6-q^3+1)^{-1}
\]
for all nontrivial $z \in G$, where $a(q) = \frac{1}{12}q^2(q^2-1)$ if $q \leqs 4$ and $a(q) = 45$ for $q \geqs 5$. One checks that this upper bound is less than $q^{-1}$ for all $q$. 
\end{proof}

Finally, we deal with the special case $q=2$.

\begin{prop} \label{p:e6_d_2}
The conclusion to Theorem~\ref{t:exceptional} holds in case~(d) of Table \ref{tab:e6_cases}.
\end{prop}

\begin{proof}
In view of Proposition \ref{p:e6_d}, we may assume $G_0 = E_6^\e(2)$, so $|y| = 13$ and we note that $G$ has a unique conjugacy class of elements of order $13$ (see \cite[Table 9]{BLS}, for example). 

First assume $\e=+$, so each subgroup in $\mathcal{M}(x)$ has type ${}^3D_4(2) \times 7$ or $F_4(2)$ (see Proposition~\ref{p:e6_d_2_max}(i)). Note that $|C_{G_0}(y)| = 91$ by Lemma~\ref{l:e6_c_centraliser}. By repeating the argument in the proof of Proposition~\ref{p:e6_d}, we see that $y$ is contained in a unique subgroup of type ${}^3D_4(2) \times 7$ and $7$ subgroups of type $F_4(2)$. Therefore, by applying the fixed point ratio bounds in Lemmas~\ref{l:e6_d_fpr1} and~\ref{l:e6_d_fpr2}, we deduce that
\[
P(z,x) \leq \sum_{H \in \mathcal{M}(x)}{\rm fpr}(z,G/H) \leqs 1\cdot \frac{1}{2^5} + 7 \cdot \frac{1}{2^6-2^3+1} < \frac{1}{2}
\]
for all nontrivial $z \in G$ and the result follows.

Now assume $\e=-$, so Proposition~\ref{p:e6_d_2_max}(ii) informs us that the subgroups in $\mathcal{M}(x)$ are of type $F_4(2)$ or ${\rm SO}_7(3)$. Note that $C_{G_0}(y) = \la y \ra$. If $H \in \mathcal{M}(x)$ has type $F_4(2)$, then $|y^G \cap H| = i_{13}(F_4(2)) = |F_4(2)|/13$ and we deduce that $y$ is contained in a unique conjugate of $H$. Similarly, if $H$ has type ${\rm SO}_7(3)$, then $|y^G \cap H| = i_{13}(H) = |{\rm SO}_{7}(3)|/13$ and thus $y$ is contained in $2$ conjugates of $H$. Therefore, $x$ is contained in at most $3$ maximal subgroups of $G$ and the bound in \eqref{e:e6_fpr} implies that 
\[
P(z,x) \leq \sum_{H \in \mathcal{M}(x)}{\rm fpr}(z,G/H) \leqs 3\cdot \frac{1}{2^6-2^3+1} < \frac{1}{2}
\]
for all nontrivial $z \in G$, as required.
\end{proof}

By combining Propositions~\ref{p:e6_a}, \ref{p:e6_b}, \ref{p:e6_c} and~\ref{p:e6_d_2}, we obtain the following.

\begin{thm} \label{t:e6}
The conclusion to Theorem~\ref{t:exceptional} holds when $G_0 = E_6^\e(q)$.
\end{thm}

\section{Proof of Theorem \ref{t:exceptional}: \texorpdfstring{$G_0 = F_4(q)$}{G0 = F4(q)}} \label{s:f4}

We now turn to the groups with socle $G_0 = F_4(q)$. By Proposition~\ref{p:cases}, in order to prove Theorem~\ref{t:exceptional} we may assume that $G = \< G_0, g \>$, where $g$ is recorded in Table~\ref{tab:f4_cases}. We will analyse cases~(a) and~(b) in Sections~\ref{ss:f4_a} and~\ref{ss:f4_b}, respectively. It will be useful to observe that for all $H \in \mathcal{M}$ and all nontrivial $z \in G$, \cite[Theorem~1]{LLS} gives
\begin{equation} \label{e:f4_fpr}
{\rm fpr}(z,G/H) \leq (q^4-q^2+1)^{-1}. 
\end{equation}

\begin{table}
\caption{The relevant groups $G = \la G_0,g \ra$ for $G_0 = F_4(q)$} \label{tab:f4_cases}
\begin{center}
\vspace{-8mm}
\[
\begin{array}{ccl} 
\hline
\text{Case} & g      & \text{Conditions}                             \\ 
\hline    
\text{(a)}  & \p^i   & \text{$i$ is a proper divisor of $f$}         \\
\text{(b)}  & \rho^i & \text{$i$ is an odd divisor of $f$ \&\ $p=2$} \\ 
\hline 
\end{array}
\]
\end{center}
\end{table}

\subsection{Case (a): field automorphisms} \label{ss:f4_a}

Fix a proper divisor $i$ of $f$ and write $e=f/i$ and $q=q_0^e$. Let $\bar{G}$ be the algebraic group $F_4$ and let $\s$ be the Steinberg endomorphism $\p^i$ of $\bar{G}$. Let $F\:F_4(q)g \to F_4(q_0)$ be the Shintani map of $(\bar{G},\s,e)$. Fix an element $y \in F_4(q_0)$ such that 
\[
|y| = 
\left\{
\begin{array}{ll}
q_0^4-q_0^2+1 & \text{if $q_0 > 2$ } \\
17            & \text{if $q_0 = 2$} 
\end{array}
\right.
\]
and $C_{F_4(q_0)}(y)=\<y\>$ (see \cite{Lubeck}). Choose $x \in G$ with $F(x) = y$ and note that $|x| = e|y|$ (see Lemma~\ref{l:shintani_order}). In addition, note that $|y|$ is $7$, $73$, $241$, $601$ (all of which are prime) when $q_0$ is $2$, $3$, $4$, $5$, respectively, and that $|y| \geqs 2353$ for $q_0 \geqs 7$.

\begin{prop} \label{p:f4_a_max}
Let $H \in \mathcal{M}(x)$. Then $H \in \mathcal{M}_1$ and $H$ is non-parabolic.
\end{prop}

\begin{proof}
First observe that the order of each maximal parabolic subgroup of $F_4(q_0)$ is indivisible by $|y|$, so Corollary~\ref{c:shintani_parabolic} implies that there are no parabolic subgroups in $\mathcal{M}(x)$. For the remainder, let us assume $H \in \mathcal{M}_2 \cup \mathcal{M}_3$ has socle $S$.

Suppose $H \in \mathcal{M}_3$. By inspecting the possibilities for $S$ given in \cite{LS99}, it is easy to see that ${\rm Aut}(S) \setminus S$ does not contain an element of order divisible by $|y|$. For the remainder, let us assume $H \in \mathcal{M}_2$. 

Suppose $S = {\rm L}_{3}^{\e}(16)$, so $p=2$ and Proposition \ref{p:meo} gives ${\rm meo}(H) \leqs 273$. Therefore, $q_0=2$ is the only possibility and we find that ${\rm Aut}(S) \setminus S$ contains elements of order $34$ (and there are also elements of order $3\cdot 17$ and $15 \cdot 17$ when $\e=+$). This implies that $e \in \{2,3,15\}$. However, if $e \in \{2,3\}$, then \cite[Lemma~8.5]{BTh} states that $G$ does not have a maximal subgroup with socle ${\rm L}_{3}^{\e}(16)$ and it is easy to see that the same proof also applies when $e=15$. 

Next assume $S = {\rm L}_{2}(t)$ with $t \leqs 68\,(2,t-1)$. By applying the bound on ${\rm meo}(H)$ from Proposition~\ref{p:meo}, we quickly reduce to the case $q_0=2$ with $t=2^6$, but $|{\rm L}_{2}(64)|$ is indivisible by $|y|=17$, so this case does not arise. The cases $S = {}^2B_2(t)$ and ${}^2G_2(t)'$ are handled in the same way.

To complete the proof, we may assume that ${\rm rk}(S)=2$ and $t \leqs 9$. If $S$ is classical, then by applying Proposition~\ref{p:meo} we may assume $q_0=2$ and $t=8$, but in each case one checks that $|S|$ is indivisible by $17$. Finally, suppose $S = G_2(t)'$. Here the bound from Proposition~\ref{p:meo} is sufficient unless $q_0 \in \{2,3\}$. For $t \in \{2,4,8\}$, it is easy to check that $|S|$ is indivisible by $17$. Similarly, $|G_2(3)|$ is indivisible by $73$ and the case $S = G_2(9)$ is ruled out by Proposition~\ref{p:sp}.
\end{proof}

\begin{prop} \label{p:f4_a}
The conclusion to Theorem~\ref{t:exceptional} holds in case~(a) of Table \ref{tab:f4_cases}.
\end{prop}

\begin{proof}
Let $H \in \mathcal{M}(x)$ and let $z \in G$ be nontrivial. By combining Propositions~\ref{p:classes} and \ref{p:f4_a_max} with Lemma~\ref{l:centraliser_bound}, we deduce that
\[
P(z,x) \leq \sum_{H \in \mathcal{M}(x)}{\rm fpr}(z,G/H) < (25+\log\log q) \cdot |y| \cdot (q^{4}-q^2+1)^{-1}.
\]
If $q_0 \geq 3$, then this bound is always less than $\frac{1}{2}$ and it is less than $q^{-1}$ for $q>25$. If $q_0=2$, then $|y|=7$ and again this bound is sufficient unless $q=4$. 

Therefore, for the remainder of the proof, we may assume that $q=4$. In this case, $|x|=34$. By Proposition~\ref{p:f4_a_max}, we know that $H \in \mathcal{M}_1$ is non-parabolic. Moreover, by carefully considering the subgroups of type~(I) to~(IV) in Theorem~\ref{t:types}, noting that $|H \cap \bar{G}_{\s}|$ is divisible by $17$, we deduce that there are at most $5$ $G_0$-classes of subgroups in $\mathcal{M}(x)$, namely
\[
\begin{array}{l}
{\rm Aut}({\rm Sp}_{8}(4)) = {\rm Sp}_{8}(4).2 \mbox{ (two classes)} \\
{\rm Aut}(\Omega_{8}^{+}(4)) = \Omega_{8}^{+}(4).{\rm Sym}_3.2  \mbox{ (two classes)} \\ F_4(2) \times 2
\end{array}
\]
Now $y$ is contained in exactly two maximal subgroups of $F_4(2)$, both of type ${\rm Sp}_{8}(2)$ (see \cite[Table~IV]{GK}). Since $B_4$ and $C_4$ are closed connected maximal subgroups of $\bar{G}$, Lemma~\ref{l:shintani_subgroups} implies that $x$ is contained in exactly two maximal subgroups of $G$ of type ${\rm Sp}_{8}(4).2$. Therefore
\[
P(z,x) \leq \sum_{H \in \mathcal{M}(x)}{\rm fpr}(z,G/H) < \frac{2 + 3 \cdot 17}{4^4-4^2+1}  = \frac{53}{241} < \frac{1}{2}
\]
for all nontrivial $z \in G$ and we have proved the result.
\end{proof}

\subsection{Case (b): graph-field automorphisms} \label{ss:f4_b}

Now let us turn to case (b) in Table \ref{tab:f4_cases}. Here $q=2^f$ and $i$ is an odd divisor of $f$. As usual, write $e=f/i$ and $q=q_0^e$, where $e \geq 1$. Let $\s$ be the graph-field Steinberg endomorphism $\r^i$ of $\bar{G} = F_4$ and let $F\:F_4(q)g \to {}^2F_4(q_0)$ be the Shintani map of $(\bar{G},\s,2e)$. Let $y \in {}^2F_4(q_0)$ be a regular semisimple element such that 
\[
|y| = q_0^2 + \sqrt{2q_0^3} + q_0 + \sqrt{2q_0}+1
\]
and $C_{{}^2F_4(q_0)}(y) = \la y \ra$ (see \cite[Table IV]{Shinoda75}). Let $x \in G$ be a Shintani correspondent of $y$ and note that $|x| = 2e|y|$ by Lemma~\ref{l:shintani_order}. In addition, observe that $|y|$ is $13$, $109$, $1321$ (all of which are prime) when $q_0$ is $2$, $8$, $32$, respectively, and $|y| \geq 18577$ if $q>32$.

First we settle the case $q=2$.

\begin{prop}\label{p:f4_b_2}
The conclusion to Theorem~\ref{t:exceptional} holds in case (b) of Table \ref{tab:f4_cases} with $q=2$.
\end{prop}

\begin{proof}
Here $G = {\rm Aut}(G_0) = G_0.2$ and the maximal subgroups of $G$ are determined up to conjugacy by Norton and Wilson in \cite{NW}. By inspecting \cite[Table~1]{NW}, we see that the only maximal subgroups of $G$ containing $y$ (other than $G_0$) are of the form $H = {}^2F_4(2) \times 2$ and $K = {\rm L}_{4}(3){:}2^2$ (the latter is in $\mathcal{M}_3$) and there is a unique conjugacy class of each type of subgroup. As before, we write $n(H)$ and $n(K)$ for the number of conjugates of $H$ and $K$, respectively, that contain $y$. Now $F_4(2)$ has a unique class of elements of order $13$, so $C_{G_0}(y) = \la y \ra$ and we deduce that $|y^G \cap H| = i_{13}(H) = |{}^2F_4(2)|/13$ and $|y^G \cap K| = i_{13}({\rm L}_{4}(3)) = 4|{\rm L}_{4}(3)|/13$. Therefore, $n(H) = 1$ and $n(K) = 2$. Finally, by applying the bound in \eqref{e:f4_fpr}, we see that
\[
P(z,x) \leq \sum_{H \in \mathcal{M}(x)}{\rm fpr}(z,G/H) \leqs \frac{3}{13}
\]
for all nontrivial $z \in G$ and the result follows.
\end{proof}

\begin{prop}\label{p:f4_b_max_e1}
If $e=1$ and $q \geq 8$, then each $H \in \mathcal{M}(x)$ is of type ${}^2F_4(q)$ or $C_{q^4-q^2+1}{:}C_{12}$.
\end{prop}

\begin{proof}
Let $H \in \mathcal{M}(x)$ and observe that $|H|$ is divisible by $|y|$. Then just by considering the orders of the maximal parabolic subgroups $P_{1,4}$ and $P_{2,3}$, we immediately deduce that $H$ is non-parabolic. As in previous cases, the bound on ${\rm meo}(H)$ in Proposition~\ref{p:orders} eliminates the subgroups in $\mathcal{M}_3$. Similarly, if $H \in \mathcal{M}_2$ has socle $S$, then Proposition~\ref{p:meo} reduces the analysis to a handful of cases with $q \in \{8,32\}$ and in each one it is clear that there are no elements in ${\rm Aut}(S) \setminus S$ of order $|x|$. Finally, if $H \in \mathcal{M}_1$ then the fact that $|H|$ is divisible by $|y|$ is highly restrictive and by inspecting the subgroups of type (I), (II) and (III) in Theorem~\ref{t:types} it is easy to check that the only possibilities are those of type ${}^2F_4(q)$ and $C_{q^4-q^2+1}{:}C_{12}$ (here we note that the maximal rank subgroups of type ${}^3D_4(q).3$ are non-maximal since $G$ contains graph-field automorphisms). This completes the proof and we note that there is a unique $G_0$-class of subgroups of each type.
\end{proof}

\begin{prop}\label{p:f4_b_max}
If $e>1$ then each $H \in \mathcal{M}(x)$ is non-parabolic and contained in $\mathcal{M}_1$.
\end{prop}

\begin{proof}
To see this, let us first observe that $|x| \geqs 52$, so subgroups in $\mathcal{M}_3$ are ruled out by Proposition~\ref{p:orders}. Now assume $H \in \mathcal{M}_2$. By applying the bound on ${\rm meo}(H)$ in Proposition~\ref{p:meo}, we quickly reduce to a handful of cases with $q_0 \in \{2,8\}$. In each of these, one checks that ${\rm Aut}(S) \setminus S$ has an element of order divisible by $|y|$ if and only if $S = {\rm L}_{3}(16)$ or ${\rm PSp}_{4}(8)$ (both with $q_0=2$, so $|y|=13$). However, in both cases, there are no elements in  ${\rm Aut}(S) \setminus S$ of order $26e$ with $e \geqs 2$. Finally, observe that the order of $y$ is not compatible with the containment of $y$ in a maximal parabolic subgroup of ${}^2F_4(q_0)$, so Corollary~\ref{c:shintani_parabolic} implies that there are no maximal parabolic subgroups in $\mathcal{M}(x)$. 
\end{proof}

\begin{prop}\label{p:f4_b}
The conclusion to Theorem~\ref{t:exceptional} holds in case~(b) of Table \ref{tab:f4_cases}.
\end{prop}

\begin{proof}
In view of Proposition \ref{p:f4_b_2}, we may assume $q>2$. Let $H \in \mathcal{M}(x)$ and let $z \in G$ be nontrivial. Recall that $C_{{}^2F_4(q_0)}(y) = \la y \ra$. If $e=1$, then by combining  Lemma \ref{l:centraliser_bound} with Proposition~\ref{p:f4_b_max_e1} and the bound in \eqref{e:f4_fpr}, we deduce that
\[
P(z,x) \leq \sum_{H \in \mathcal{M}(x)} {\rm fpr}(z,G/H) < 2 \cdot |y| \cdot (q^4-q^2+1)^{-1} < \frac{1}{q}
\]
for all nontrivial $z \in G$. Similarly, if $e > 1$ then by applying Propositions~\ref{p:classes} and \ref{p:f4_b_max} we get
\[
P(z,x) \leq \sum_{H \in \mathcal{M}(x)} {\rm fpr}(z,G/H) < (21 + \log\log q) \cdot |y| \cdot (q^4-q^2+1)^{-1}.
\]
One checks that this upper bound is less than $q^{-1}$ for $q>4$, but the case $q=4$ requires further attention (even for the desired $\frac{1}{2}$ bound).

Assume that $q=4$ with $q_0=2$ and $e=2$, so $|x|=52$. By Proposition~\ref{p:f4_b_max}, we know that each $H \in \mathcal{M}(x)$ is non-parabolic and is contained in $\mathcal{M}_1$. There are no exotic local subgroups (see \cite{CLSS}) and so it remains to consider the maximal rank subgroups in $\mathcal{M}(x)$, together with the subfield subgroup $F_4(2)$. By inspecting \cite{LSS}, using the fact that $H \cap G_0$ must contain elements of order $13$ and $G$ contains graph-field automorphisms, we deduce that the only possible maximal rank subgroups in $\mathcal{M}(x)$ are of type ${\rm U}_{3}(4)^2.2$ and $13^2{:}(3 \times {\rm SL}_{2}(3))$. In particular, we may replace the leading factor $21+\log\log q$ in the above bound by $3$ and the result follows.
\end{proof}

By combining Propositions~\ref{p:f4_a} and \ref{p:f4_b}, we have now proved the following theorem.

\begin{thm} \label{t:f4}
The conclusion to Theorem~\ref{t:exceptional} holds when $G_0 = F_4(q)$.
\end{thm}

\section{Proof of Theorem \ref{t:exceptional}: \texorpdfstring{$G_0 = {}^3D_4(q)$}{G0 = 3D4(q)}} \label{s:3d4}

In this final section, we complete the proof of Theorem~\ref{t:exceptional} by handling the almost simple groups $G$ with socle $G_0 = {}^3D_4(q)$. In \cite{Kleidman3D4}, Kleidman determines the maximal subgroups of $G$ and we note that $G$ has at most $10+\log\log q$ conjugacy classes of maximal subgroups. In addition, \cite[Theorem~1]{LLS} gives the bound
\begin{equation} \label{e:3d4_fpr}
{\rm fpr}(z,G/H) \leqs (q^4-q^2+1)^{-1}
\end{equation}
for all $H \in \mathcal{M}$ and all nontrivial $z \in G$. 

By considering Proposition~\ref{p:cases}, we see that it suffices to assume $G = \la G_0, g\ra$, where $g$ is recorded in Table \ref{tab:3d4_cases}.
In this table, we write $\Delta(f)$ for the set of positive proper divisors of $f$ and $\tau$ is the triality graph automorphism of $G_0$ in Definition \ref{d:aut}(iii).

\begin{table}
\caption{The relevant groups $G = \la G_0,g \ra$ for $G_0 = {}^3D_4(q)$} \label{tab:3d4_cases}
\begin{center}
\vspace{-8mm}
\[
\begin{array}{ccl} 
\hline
\text{Case} & g      & \text{Conditions}                                       \\ 
\hline    
\text{(a)}  & \t\p^i & \text{$i \in \Delta(f)$ \&\ $f/i \not\equiv 0 \mod{3}$} \\
\text{(b)}  &   \p^i & i \in \Delta(f)                                         \\ 
\text{(c)}  & \t     &                                                         \\ 
\hline 
\end{array}
\]
\end{center}
\end{table}

\subsection{Case (a): Shintani descent} \label{ss:3d4_a}

Here $i$ is a proper divisor of $f$ and $e=f/i$ is indivisible by $3$. Set $q=q_0^e$ and let $\bar{G}$ be the adjoint algebraic group $D_4$ over $\bar{\mathbb{F}}_p$. Let $\s$ be the Steinberg endomorphism $\t\p^i$ of $\bar{G}$ and let $F\:{}^3D_4(q)g \to {}^3D_4(q_0)$ be the Shintani map of $(\bar{G},\s,e)$. Choose $y \in {}^3D_4(q_0)$ such that $|y|= q_0^4-q_0^2+1$ and $C_{{}^3D_4(q_0)}(y) = \< y \>$. Let $x \in G$ be a Shintani correspondent, so $|x|=e|y|$.

\begin{prop} \label{p:3d4_a}
The conclusion to Theorem~\ref{t:exceptional} holds in case~(a) of Table \ref{tab:3d4_cases}.
\end{prop}

\begin{proof}
First observe that $y$ is not contained in a maximal parabolic subgroup of ${}^3D_4(q_0)$, since $|y|$ does not divide the order of any such group, whence $x$ is not contained in a maximal parabolic subgroup of $G$ by Corollary~\ref{c:shintani_parabolic}. Therefore, in view of Lemma \ref{l:centraliser_bound} and the bound in \eqref{e:3d4_fpr}, we deduce that
\[
\sum_{H \in \mathcal{M}(x)}{\rm fpr}(z,G/H) \leqs (8+\log\log q)\cdot |y| \cdot (q^4-q^2+1)^{-1}
\]
for all nontrivial $z \in G$. One checks that this bound is always sufficient (in particular, the bound is less than $q^{-1}$ for $q>4$).
\end{proof}

\subsection{Case (b): Shintani descent over \texorpdfstring{$G_2$}{G2}} \label{ss:3d4_b}

In this case, we can proceed as in Section~\ref{ss:e6_d}, using Lemma~\ref{l:shintani_substitute}. Fix a proper divisor $i$ of $f$ and write $e=f/i$ and $q=q_0^e$. Set $\bar{G} = D_4$ and let $\s$ be the Steinberg endomorphism $\p^i$. In addition, let $\tau$ be a triality graph automorphism of $\bar{G}$ such that $\bar{K} = C_{\bar{G}}(\t) = G_2$ and note that $\bar{K}$ is $\s$-stable and $\bar{G}_{\t\s^e} = {}^3D_4(q)$. Fix an element $y \in \bar{K}_\s = G_2(q_0)$ of order 
\[
|y| = 
\left\{
\begin{array}{ll}
q_0^2-q_0+1 & \text{if $q_0 > 2$} \\
7           & \text{if $q_0 = 2$.} 
\end{array}
\right.
\]
By Lemma~\ref{l:shintani_substitute}(i), there exists $x \in \bar{K}_{\s^e}\s = G_2(q)g \subseteq {}^3D_4(q)g$ such that $x^e$ is $\bar{G}$-conjugate to $y\t^2$. In particular, $x^{3e}$ is $\bar{G}$-conjugate to $y^3$. By Remark~\ref{r:shintani_order}, $|x| = 3e|y^3|$ and we note that $|y^3| = (q_0^2-q_0+1)/(3,q_0+1)$ if $q_0 > 2$.

\begin{prop} \label{p:3d4_b}
The conclusion to Theorem~\ref{t:exceptional} holds in case~(b) of Table \ref{tab:3d4_cases}.
\end{prop}

\begin{proof}
Here $\bar{G}_\s = \Inndiag({\rm P}\Omega_8^+(q_0))$ and $y^3 \in G_2(q_0) \leq \bar{G}_\s$. From \cite{Lubeck}, we see that the order of $y$ implies that
\[
|C_{\bar{G}_\s}(y^3)| = c(q_0) = 
\left\{
\begin{array}{ll}
(q_0^3+1)(q_0+1) & \text{if $q_0 > 2$} \\
7                & \text{if $q_0 = 2$.}
\end{array}
\right.
\] 
Therefore, by applying Lemma~\ref{l:shintani_substitute}(ii)(a) and the bound in \eqref{e:3d4_fpr} we deduce that
\[
P(z,x) \leq \sum_{H \in \mathcal{M}(x)}{\rm fpr}(z,G/H) \leqs (10+\log\log q)\cdot c(q_0) \cdot (q^4-q^2+1)^{-1}
\]
for all nontrivial $z \in G$. The result follows (in particular, the upper bound is less than $q^{-1}$ if $q>9$).
\end{proof}

\subsection{Case (c): triality graph automorphisms} \label{ss:3d4_c}

We have reached the final case. Here we may assume that $G = \< G_0, g \>$, where $g$ is the standard triality graph automorphism of $G_0$. As in Section~\ref{ss:e6_c}, we cannot apply Shintani descent in this case.

First we handle the case $q=2$.

\begin{prop}\label{p:case2}
The conclusion to Theorem~\ref{t:exceptional} holds in case (c) of Table \ref{tab:3d4_cases} with $q=2$.
\end{prop}

\begin{proof}
As in the proof of Proposition~\ref{p:small}, it is straightforward to use {\sc Magma} to handle this case. In particular, we find that the class labelled ${\tt 24A}$ in the {\sc Atlas} \cite{ATLAS} witnesses $u(G) \geq 4$.
\end{proof}

For the remainder, we will assume $q \geq 3$. Choose $y \in C_{G_0}(g) = G_2(q)$ such that  
\[
|y| = q^2-q+1
\]
and $C_{G_2(q)}(y) = \la y \ra$. Set $x = yg \in G$ and note that $x^3 = y^3 \in G_0$ and $|x| = 3|y|/(3,q+1)$. Write 
\[
r = |y^3| = (q^2-q+1)/(q+1,3),
\] 
which is divisible by a primitive prime divisor of $q^6-1$. 

\begin{lem}\label{l:3d40}
We have $C_{G_0}(y) = C_{q^2-q+1} \times C_{q^2-q+1}$.
\end{lem}

\begin{proof}
We may choose $y \in {\rm SU}_{3}(q) < G_2(q)$, so $y \in \bar{L}< \bar{H} < \bar{G}$, where $\bar{L} = A_2$, $\bar{H} = G_2$ and $\bar{G} = D_4$ are the corresponding algebraic groups. Let $V$ and $U$ be the natural modules for $\bar{G}$ and $\bar{L}$, respectively, and observe that $V|_{\bar{L}} = U \oplus U^* \oplus 0^2$, where $0$ is the trivial module and $U^*$ is the dual of $U$. By first considering the eigenvalues of $y$ on $U$, and then on $V$ via the given decomposition, we deduce that the connected component of $C_{\bar{G}}(y)$ is a maximal torus. In particular, $y$ is a regular semisimple element of $G_0$ and by inspecting \cite[Table II]{Kleidman3D4} we deduce that $C_{G_0}(y)$ is either $C_{q^2-q+1} \times C_{q^2-q+1}$ or $C_{q^3+1} \times C_{q+1}$. Finally, we observe that the ${\rm SU}_{3}(q)$ subgroup of $G_2(q)$ containing $y$ is centralised in ${}^3D_4(q)$ by a torus of order $q^2-q+1$ and this rules out the latter possibility.
\end{proof}

\begin{prop}\label{p:3d4_c_max}
If $q > 2$ then each $H \in \mathcal{M}(x)$ is of one of the following types:
\[
G_2(q), \ {\rm PGU}_{3}(q) \, \mbox{$(q \equiv 2 \imod{3})$}, \ {\rm SU}_{3}(q) \times C_{q^2-q+1}, \ C_{q^2-q+1} \times C_{q^2-q+1}.
\]
\end{prop}

\begin{proof}
Since $\mathcal{M}(x) \subseteq \mathcal{M}(y^3)$, we proceed by considering the maximal overgroups $H_0$ of $y^3$ in $G_0$, referring to the main theorem of \cite{Kleidman3D4} (also see \cite[Theorem~4.3]{Wilson}). 

By inspection, the only parabolic subgroup with order divisible by $r$ is of the form 
\[
H_0 = q^{1+8}{:}{\rm SL}_{2}(q^3).C_{q-1}.
\]
However, the maximal tori of ${\rm SL}_{2}(q^3)$ have order $q^3 \pm 1$, so there are no elements in $H_0$ with the appropriate centraliser in $G_0$. Therefore, there are no parabolic subgroups in $\mathcal{M}(x)$.

Plainly, we will find  subgroups of type $G_2(q)$ in $\mathcal{M}(y^3)$, and there may also be subgroups of type ${\rm PGU}_{3}(q)$ when $q \equiv 2 \imod{3}$. If $p=2$ and $H_0 = {\rm L}_{2}(q^3) \times {\rm L}_{2}(q)$, then $C_{H_0}(z)$ has a cyclic subgroup of order $q^3+1$ (a maximal torus in the first factor) for each $z \in H_0$ of order $r$, so these subgroups do not arise. Since $y^3$ does not commute with an involution, we can also exclude the involution centraliser when $p$ is odd. Subfield subgroups can be ruled out by Lagrange's theorem. Similarly, just by considering divisibility, we see that the only other possibilities are subgroups of type ${\rm SU}_{3}(q) \times C_{q^2-q+1}$ and $C_{q^2-q+1} \times C_{q^2-q+1}$ (the latter is the centraliser of $y^3$ in $G_0$). The result follows.
\end{proof}

Let $H \in \mathcal{M}(x)$. If $H$ is of type $G_2(q)$, then \cite{LLS} gives ${\rm fpr}(z,G/H) \leqs (q^4-q^2+1)^{-1}$ for all nontrivial $z \in G$ and this bound is best possible. Indeed, equality holds if $z$ is a long root element in $G$ (see the proof of \cite[Lemma~6.3]{LLS}). For the other subgroups arising in Proposition~\ref{p:3d4_c_max}, we need to sharpen the bound on ${\rm fpr}(z,G/H)$ in \cite{LLS}. To do this, it will be helpful to observe that if $z \in G$ has prime order, but is not a long root element, then $|z^G|>q^{14}$. In addition, if $1 \ne z \in G_0$ is not a long root element, then $|z^G| > q^{16}$. For both of these claims, see \cite{DM}.

\begin{lem}\label{l:3d4_c_fpr1}
Let $H$ be a maximal subgroup of $G$ of type ${\rm PGU}_{3}(q)$, where $q \equiv 2 \imod{3}$ and $q \geqs 5$, and let $z \in G$ be nontrivial. Then 
\[
{\rm fpr}(z,G/H) < 2q^{-6}. 
\]
\end{lem}

\begin{proof}
By replacing $z$ by a suitable conjugate, we may as well assume $z$ is contained in $H$ and has prime order. Observe that $H = {\rm PGU}_{3}(q) \times 3 =  C_{G}(g')$, where $g'$ is a certain triality graph automorphism of $G_0$. 

First we claim that $z$ is not a long root element in $G_0$. To see this, let $\bar{H}=A_2$ and $\bar{G}=D_4$ be the corresponding algebraic groups and observe that the natural module $V$ for $\bar{G}$ is the adjoint module for $\bar{H}$. This allows us to compute the Jordan form of each unipotent element in $\bar{H}$ on $V$. Indeed, if $p=2$ then $\bar{H}$ has a unique class of involutions and such an element has Jordan form $[J_2^4]$ on $V$. Similarly, if $p \geqs 5$ then each element in $\bar{H}$ of order $p$ has Jordan form $[J_3,J_2^2,J_1]$ or $[J_5,J_3]$ on $V$. The claim now follows since the long root elements in $\bar{G}$ have Jordan form $[J_2^2, J_1^4]$ on $V$. 

To complete the proof, recall that $|z^G|>q^{14}$ if $z$ is not a long root element, so the result follows from the trivial bound $|z^G \cap H| \leqs 2|{\rm PGU}_{3}(q)|< 2q^8$.
\end{proof}

\begin{lem}\label{l:3d4_c_fpr2}
Let $H$ be a maximal subgroup of $G$ of type $C_{q^2-q+1} \times C_{q^2-q+1}$, where  $q \geqs 3$, and let $z \in G$ be nontrivial. Then 
\[
{\rm fpr}(z,G/H) < q^{-6}. 
\]
\end{lem}

\begin{proof}
Here $H = (C_{q^2-q+1} \times C_{q^2-q+1}){:}{\rm SL}_{2}(3).3$. Assume that $z \in H$ has prime order. As in the proof of Lemma~\ref{l:3d4_c_fpr1}, if $z \in G$ is not a long root element, then $|z^G|>q^{14}$ and we get
\[
{\rm fpr}(z,G/H) \leqs |H|q^{-14} < q^{-6}.
\]
Therefore, we just need to rule out the existence of long root elements in $H$.

If $p \geqs 5$, then $|H|$ is indivisible by $p$, so we may assume $p \in \{2,3\}$. Seeking a contradiction, suppose $z \in H$ is a long root element. Viewing $z$ as an element of the algebraic group $\bar{G} = D_4$, note that $z$ normalises a maximal torus of $\bar{G}$, so \cite[Proposition~1.13(iii)]{LLSAlgebraic} implies that $p=2$ and thus $z$ is an involution. In particular, $z$ is in the coset $St$ of $S = C_{q^2-q+1} \times C_{q^2-q+1}$, where $t$ is the unique involution in ${\rm SL}_{2}(3)$. However, all the involutions in $St$ are contained in the largest class of involutions in $\bar{G}$ (see \cite[Corollary~4.4]{BTh}, for example), whence all the involutions in $H$ are in the $G_0$-class labelled $3A_1$. In particular, there are no involutions in the class $A_1$, which comprises the long root elements in $G_0$. This is a contradiction and the result follows.
\end{proof}

\begin{lem}\label{l:3d4_c_fpr3}
Let $H$ be a maximal subgroup of $G$ of type ${\rm SU}_{3}(q) \times C_{q^2-q+1}$ and let $z \in G$ be nontrivial. Then 
\[
{\rm fpr}(z,G/H) < 2q^{-6}.
\]
\end{lem}

\begin{proof}
Assume that $z \in H$ has prime order. Write $H_0 = H \cap G_0$ and observe that 
\[
H_0 = ({\rm SU}_{3}(q) \circ C_{q^2-q+1}).(3,q+1).2 = {\rm SU}_{3}(q).C_{q^2-q+1}.2. 
\]
First assume $z \in G$ is either semisimple or unipotent, but not a long root element. Then $|z^G|>q^{16}$ and the result follows since $|z^G \cap H| \leqs |H_0|< 2q^{10}$. In addition, the long root elements in $H_0$ coincide with the long root elements in the ${\rm SU}_{3}(q)$ subgroup, so if $z$ is such an element then  
\[
|z^G \cap H| = (q-1)(q^3+1),\; |z^G| = (q^2-1)(q^8+q^4+1)
\]
and thus ${\rm fpr}(z,G/H) < q^{-6}$.

To complete the proof, assume $z \in G$ is a graph automorphism of order $3$. If $C_{G_0}(z) \ne G_2(q)$, then $|z^G|>\frac{1}{2}q^{20}$ and the desired bound follows since $|z^G \cap H| \leqs 2|H_0|< q^{12}$. Finally, suppose $C_{G_0}(z) = G_2(q)$, so $|z^G|>q^{14}$. In terms of algebraic groups, let $\bar{J} = A_2T_2 < \bar{G}=D_4$ and let $\tau$ be a graph automorphism of $\bar{G}$ with $C_{\bar{G}}(\tau) = G_2$. By arguing as in the proof of \cite[Proposition~3.3]{Bur4}, we see that $t\tau \in \bar{J}\tau$ is a $G_2$-type triality graph automorphism if and only if $t \in Z(\bar{J})$. Therefore, returning to the finite groups, we deduce that 
\[
|z^G \cap H| = 2|Z({\rm SU}_{3}(q))| \cdot \frac{q^2-q+1}{(3,q+1)} = 2(q^2-q+1)
\]
and the desired bound follows.
\end{proof}

\begin{prop} \label{p:3d4_c}
The conclusion to Theorem~\ref{t:exceptional} holds in case~(c) of Table \ref{tab:3d4_cases}.
\end{prop}

\begin{proof}
In view of Proposition \ref{p:case2}, we may assume $q \geqs 3$. Recall that the maximal overgroups $H$ of $x$ are described in Proposition~\ref{p:3d4_c_max}. For each type of subgroup, we need to bound the number of conjugates of $H$ containing $x$. As before, we do this by estimating the number of conjugates containing $y$, which we denote by $n(H)$.

First assume $H \in \mathcal{M}(x)$ is a subgroup of type $G_2(q)$. By inspecting \cite{Chang, Enomoto}, we see that $G_2(q)$ has at most $\frac{1}{6}q(q-1)$ conjugacy classes of semisimple elements with centraliser $C_{q^2-q+1}$ and thus
\[
|y^G \cap H| \leqs \frac{1}{6}q(q-1) \cdot \frac{|G_2(q)|}{q^2-q+1}.
\]
This implies that
\[
n(H) \leqs \frac{1}{6}q(q-1)(q^2-q+1).
\]
Alternatively, by arguing as in the proof of \cite[Lemma~4.5]{LLS} we see that $y^G \cap H$ is a union of at most $|W(D_4):W(G_2)| = 16$ distinct $H$-classes and this yields $n(H) \leqs 16(q^2-q+1)$. Notice that the latter bound is better for $q>9$.

Next assume $q \equiv 2 \imod{3}$ and $H \in \mathcal{M}(x)$ is of type ${\rm PGU}_{3}(q)$. Now ${\rm PGU}_{3}(q)$ has $\frac{1}{3}(q^2-q-2)$ classes of semisimple elements with centraliser $C_{q^2-q+1}$, so 
\[
|y^G \cap H| \leqs \frac{1}{3}(q^2-q-2) \cdot \frac{|{\rm PGU}_{3}(q)|}{q^2-q+1}
\]
and we get $n(H) \leqs \frac{1}{3}(q^2-q-2)(q^2-q+1)$.

Now suppose $H \in \mathcal{M}(x)$ is of type ${\rm SU}_{3}(q) \times C_{q^2-q+1}$. Set $H_0 = H \cap G_0$ and recall that $H_0 = {\rm SU}_{3}(q).C_{q^2-q+1}.2$. Now ${\rm SU}_{3}(q)$ has $\lceil \frac{1}{3}(q^2-q-2)\rceil \leqs \frac{1}{3}q(q-1)$ conjugacy classes of semisimple elements with centraliser of order $q^2-q+1$ and this implies that 
\[
|y^G \cap H| \leqs \frac{1}{3}q(q-1) \cdot \frac{|{\rm SU}_{3}(q)|}{q^2-q+1}\cdot (q^2-q+1).
\]
In turn, this gives $n(H) \leqs \frac{1}{6}q(q-1)(q^2-q+1)$.

Finally, if $H \in \mathcal{M}(x)$ is of type $C_{q^2-q+1} \times C_{q^2-q+1}$ then $|y^G \cap H| \leqs (q^2-q+1)^2$ and we deduce that $n(H) \leqs \frac{1}{24}(q^2-q+1)^2$. 

By combining the above bounds with the fixed point ratio estimates in \eqref{e:3d4_fpr} and Lemmas \ref{l:3d4_c_fpr1}--\ref{l:3d4_c_fpr3}, we conclude that 
\begin{align*}
P(z,x) &\leq \sum_{H \in \mathcal{M}(x)} {\rm fpr}(z,G/H) \\ 
&< a(q)\cdot (q^2-q+1)\cdot (q^4-q^2+1)^{-1} + \frac{1}{3}(q^2-q-2)(q^2-q+1)\cdot 2q^{-6} \\
&\quad + \frac{1}{6}q(q-1)(q^2-q+1)\cdot 2q^{-6}  + \frac{1}{24}(q^2-q+1)^2\cdot q^{-6}
\end{align*}
for all nontrivial $z \in G$, where $a(q) = \frac{1}{6}q(q-1)$ if $q \leqs 9$ and $a(q)=16$ for $q>9$. One checks that this upper bound is less than $\frac{1}{2}$ for all $q > 2$ and less than $q^{-1}$ for $q>16$.
\end{proof}

In view of Propositions~\ref{p:3d4_a}, \ref{p:3d4_b} and~\ref{p:3d4_c}, we have now proved the following result. 

\begin{thm} \label{t:3d4}
The conclusion to Theorem~\ref{t:exceptional} holds when $G_0={}^3D_4(q)$.
\end{thm}

Moreover, by combining this with Theorems~\ref{t:low}, \ref{t:e8}, \ref{t:e7}, \ref{t:e6} and  \ref{t:f4}, we conclude that the proof of Theorem~\ref{t:exceptional} is complete.


\begin{thebibliography}{999}

\bibitem{AG}
M. Aschbacher and R. Guralnick, \emph{Some applications of the first cohomology group}, J. Algebra \textbf{90} (1984), 446--460. 

\bibitem{AS} 
M. Aschbacher and L.L. Scott,  \emph{Maximal subgroups of finite groups},  J. Algebra \textbf{92} (1985),  44--80. 

\bibitem{BBR} 
J. Ballantyne, C. Bates and P. Rowley, \emph{The maximal subgroups of $E_7(2)$}, LMS J. Comput. Math. \textbf{18} (2015), 323--371.

\bibitem{Binder68} 
G.J. Binder, \emph{The bases of the symmetric group}, Izv. Vyssh. Uchebn. Zaved. Mat. \textbf{78} (1968), 19--25.

\bibitem{Binder70I} 
G.J. Binder, \emph{The two-element bases of the symmetric group}, Izv. Vyssh. Uchebn. Zaved. Mat. \textbf{90} (1970), 9--11.

\bibitem{Binder70II} G.J. Binder, \emph{Certain complete sets of complementary elements of the symmetric and the alternating group of the nth degree}, Mat. Zametki \textbf{7} (1970), 173--180.

\bibitem{magma} 
W. Bosma, J. Cannon and C. Playoust, \emph{The {\textsc{Magma}} algebra system I: The user language}, J. Symb. Comput. \textbf{24} (1997), 235--265.

\bibitem{BHR}
J.N. Bray, D.F. Holt and C.M. Roney-Dougal, \emph{The maximal subgroups of the low-dimensional finite classical groups}, London Math. Soc. Lecture Notes Series, vol. 407, Cambridge University Press, 2013. v+438 pp.

\bibitem{BW} 
J.L. Brenner and J. Wiegold, \emph{Two generator groups, {I}}, Michigan Math. J. \textbf{22} (1975), 53--64.

\bibitem{BGK} 
T. Breuer, R.M. Guralnick and W.M. Kantor, \emph{Probabilistic generation of finite simple groups II}, J. Algebra \textbf{320} (2008), 443--494.

\bibitem{Bur4} 
T.C. Burness, \emph{Fixed point ratios in actions of finite classical groups,  IV}, J. Algebra \textbf{314} (2007), 749--788.

\bibitem{Bur18} 
T.C. Burness, \emph{On base sizes for almost simple primitive groups}, J. Algebra \textbf{516} (2018), 38--74.

\bibitem{Bur19}
T.C. Burness, \emph{Simple groups, generation and probabilistic methods},  Groups St Andrews 2017 in Birmingham, 200--229, London Math. Soc. Lecture Note Ser., 455, Cambridge Univ. Press, Cambridge, 2019.

\bibitem{BG} 
T.C. Burness and S. Guest, \emph{On the uniform spread of almost simple linear groups}, Nagoya Math. J. \textbf{109} (2013), 35--109.

\bibitem{BH}
T.C. Burness and S. Harper, \emph{Finite groups, $2$-generation and the uniform domination number}, Israel J. Math. \textbf{239} (2020), 271--367.

\bibitem{BLS} 
T.C. Burness, M.W. Liebeck and A. Shalev, \emph{Base sizes for simple groups and a conjecture of Cameron}, Proc. Lond. Math. Soc. \textbf{98} (2009), 116--162.

\bibitem{BTh} 
T.C. Burness and A.R. Thomas, \emph{On the involution fixity of exceptional groups of Lie type}, Internat. J. Algebra Comput. \textbf{28} (2018), 411--466.

\bibitem{PRA}
F. Celler, C.R. Leedham-Green, S.H. Murray, A.C. Niemeyer and E.A. O'Brien, 
\emph{Generating random elements of a finite group}, Comm. Algebra \textbf{23} (1995), 4931--4948.

\bibitem{Chang} 
B. Chang, \emph{The conjugate classes of Chevalley groups of type $(G_2)$}, J. Algebra \textbf{9} (1968), 190--211.

\bibitem{CLSS}
A.M. Cohen, M.W. Liebeck, J. Saxl and G.M. Seitz, \emph{The local maximal subgroups of exceptional groups of {L}ie type}, Proc. Lond. Math. Soc. \textbf{64} (1992), 21--48.

\bibitem{ATLAS} 
J.H. Conway, R.T. Curtis, S.P. Norton, R.A. Parker and R.A. Wilson, \emph{Atlas of {F}inite {G}roups}, Oxford University Press, 1985.

\bibitem{Coop} 
B.N. Cooperstein, \emph{Maximal subgroups of $G_2(2^n)$}, J. Algebra \textbf{70} (1981), 23--36.

\bibitem{Craven17} 
D.A. Craven, \emph{Alternating subgroups of exceptional groups of Lie type},  Proc. Lond. Math. Soc. \textbf{115} (2017), 449--501.

\bibitem{CravenPSL} 
D.A. Craven, \emph{Maximal ${\rm PSL}_2$ subgroups of exceptional groups of Lie type}, Mem. Amer. Math. Soc., to appear.

\bibitem{CravenMedium} 
D.A. Craven, \emph{On medium-rank Lie primitive and maximal subgroups of exceptional groups of Lie type}, submitted.

\bibitem{DM}
D.I. Deriziotis and G.O. Michler, \emph{Character table and blocks of finite simple triality groups ${}^3D_4(q)$}, Trans. Amer. Math. Soc. \textbf{303} (1987), 39--70.

\bibitem{DH}
C. Donoven and S. Harper, \emph{Infinite $\frac{3}{2}$-generated groups}, Bull. London Math. Soc. \textbf{52} (2020), 657--673.

\bibitem{Enomoto}
H. Enomoto, \emph{The conjugacy classes of Chevalley groups of type $(G_2)$ over finite fields of characteristic $2$ or $3$}, J. Fac. Sci. Univ. Tokyo Sect. I \textbf{16} (1969), 497--512.

\bibitem{Evans} 
M. Evans, \emph{$T$-systems of certain finite simple groups}, Math. Proc. Cambridge Philos. Soc. \textbf{113} (1993), 9--22.

\bibitem{FJ}
P. Fleischmann and I. Janiszczak, \emph{The semisimple conjugacy classes and the generic class number of the finite simple groups of Lie type $E_8$}, Comm. Algebra \textbf{22} (1994), 2221--2303.
 
\bibitem{FGS}  
M. Fried, R. Guralnick and J. Saxl,  \emph{Schur covers and Carlitz's conjecture}, Israel J. Math. \textbf{82} (1993), 157--225.

\bibitem{GLS} 
D. Gorenstein, R. Lyons and R. Solomon, \emph{The classification of the finite simple groups. Number 3. Part I. Chapter A. Almost simple $K$-groups},  Mathematical Surveys and Monographs, 40.3. American Mathematical Society, Providence, RI, 1998. xvi+419 pp.

\bibitem{Guba}
V.S. Guba, \emph{A finitely generated simple group with free $2$-generated subgroups}, Sibirsk. Mat. Zh. \textbf{27} (1986), 50--67.

\bibitem{GMPS} 
S. Guest, J. Morris, C.E. Praeger and P. Spiga, \emph{On the maximum orders of elements of finite almost simple groups and primitive permutation groups}, Trans. Amer. Math. Soc. \textbf{367} (2015), 7665--7694.
 
\bibitem{GK} 
R.M. Guralnick and W.M. Kantor, \emph{Probabilistic generation of finite simple groups}, J. Algebra \textbf{234} (2000), 743--792.

\bibitem{GSh}
R.M. Guralnick and A. Shalev, \emph{On the spread of finite simple groups}, Combinatorica \textbf{23} (2003), 73--87.

\bibitem{Harper17} 
S. Harper, \emph{On the uniform spread of almost simple symplectic and orthogonal groups}, J. Algebra \textbf{490} (2017), 330--371.

\bibitem{HarperClassical}
S. Harper, \emph{The spread of almost simple classical groups}, Lecture Notes in Math., Springer, to appear.

\bibitem{KS}
W.M. Kantor and \'{A}. Seress, \emph{Large element orders and the characteristic of Lie-type simple groups}, J. Algebra \textbf{322} (2009), 802--832.

\bibitem{Kawanaka}
N. Kawanaka, \emph{On the irreducible characters of the finite unitary groups}, J. Math. Soc. Japan \textbf{29} (1977), 425--450.

\bibitem{Kleidman3D4} 
P.B. Kleidman, \emph{The maximal subgroups of the Steinberg triality groups ${}^3D_4(q)$ and of their automorphism groups}, J. Algebra \textbf{115} (1988), 182--199.

\bibitem{KleidmanG2} 
P.B. Kleidman, \emph{The maximal subgroups of the Chevalley groups $G_2(q)$ with $q$ odd, the Ree groups ${}^2G_2(q)$, and their automorphism groups},  J. Algebra \textbf{117} (1988), 30--71.
 
\bibitem{KL} 
P.B. Kleidman and M.W. Liebeck, \emph{The {S}ubgroup {S}tructure of the {F}inite {C}lassical {G}roups}, London Math. Soc. Lecture Note Series, vol. 129, Cambridge University Press, 1990.
 
\bibitem{KW} 
P.B. Kleidman and R.A. Wilson, \emph{The maximal subgroups of $E_6(2)$ and ${\rm Aut}(E_6(2))$}, Proc. Lond. Math. Soc. \textbf{60} (1990), 266--294.

\bibitem{Lawther}
R. Lawther, \emph{Sublattices generated by root differences}, J. Algebra \textbf{412} (2014), 255--263.

\bibitem{LLSAlgebraic} 
R. Lawther, M.W. Liebeck and G.M. Seitz, \emph{Fixed point spaces in actions of exceptional algebraic groups}, Pacific J. Math. \textbf{205} (2002), 339--391.

\bibitem{LLS} 
R. Lawther, M.W. Liebeck and G.M. Seitz, \emph{Fixed point ratios in actions of finite exceptional groups of Lie type}, Pacific J. Math. \textbf{205} (2002), 393--464.

\bibitem{LSS} 
M.W. Liebeck, J. Saxl and G.M. Seitz, \emph{Subgroups of maximal rank in finite exceptional groups of Lie type}, Proc. Lond. Math. Soc. \textbf{65} (1992), 297--325.

\bibitem{LS90}
M.W. Liebeck and G.M. Seitz, \emph{Maximal subgroups of exceptional groups of Lie type, finite and algebraic}, Geom. Dedicata \textbf{35} (1990), 353--387.

\bibitem{LS99} 
M.W. Liebeck and G.M. Seitz, \emph{On finite subgroups of exceptional algebraic groups}, J. reine angew. Math. \textbf{515} (1999), 25--72.

\bibitem{LS03} 
M.W. Liebeck and G.M. Seitz, \emph{A survey of of maximal subgroups of exceptional groups of Lie type}, in Groups, combinatorics \& geometry (Durham, 2001), 139--146, World Sci. Publ., River Edge, NJ, 2003.

\bibitem{LSh96}
M.W. Liebeck and A. Shalev, \emph{Classical groups, probabilistic methods, and the $(2,3)$-generation problem}, Annals of Math. \textbf{144} (1996), 77--125.

\bibitem{Litt}
A.J. Litterick, \emph{On non-generic finite subgroups of exceptional algebraic groups},  Mem. Amer. Math. Soc. \textbf{253} (2018), no. 1207, v+156 pp.

\bibitem{Lubeck} 
F. L\"{u}beck, \emph{Centralisers and numbers of semisimple classes in exceptional groups of Lie type}, \\ \url{http://www.math.rwth-aachen.de/\~Frank.Luebeck/chev/CentSSClasses}

\bibitem{LM}
A. Lucchini and F. Menegazzo, \emph{Generators for finite groups with a unique minimal normal subgroup}, Rend. Semin. Mat. Univ. Padova \textbf{98} (1997), 173--191.

\bibitem{Malle}
G. Malle, \emph{The maximal subgroups of ${}^2F_4(q^2)$}, J. Algebra \textbf{139} (1991), 52--69.

\bibitem{NW}
S.P. Norton and R.A. Wilson, \emph{The maximal subgroups of $F_4(2)$ and its automorphism group}, Comm. Algebra \textbf{17} (1989), 2809--2824.

\bibitem{Pak}
I. Pak, \emph{What do we know about the product replacement algorithm?}, in Groups and computation, III (Columbus, OH, 1999), 301--347, Ohio State Univ. Math. Res. Inst. Publ., de Gruyter, Berlin, 2001.

\bibitem{Piccard} 
S. Piccard, \emph{Sur les bases du groupe sym\'{e}trique et du groupe alternant}, Math. Ann. \textbf{116} (1939), 752--767.

\bibitem{Shinoda74}
K. Shinoda, \emph{The conjugacy classes of Chevalley groups of type $(F_4)$ over finite fields of characteristic $2$}, J. Fac. Sci. Univ. Tokyo Sect. I A Math. \textbf{21} (1974), 133--159.

\bibitem{Shinoda75}
K. Shinoda, \emph{The conjugacy classes of the finite Ree groups of type $(F_4)$}, J. Fac. Sci. Univ. Tokyo Sect. I A Math. \textbf{22} (1975), 1--15.

\bibitem{Shintani}
T. Shintani, \emph{Two remarks on irreducible characters of finite general linear groups}, J. Math. Soc. Japan \textbf{28} (1976), 396--414.

\bibitem{Shoji} 
T. Shoji, \emph{The conjugacy classes of Chevalley groups of type $(F_4)$ over finite fields of characteristic $p \ne 2$}, J. Fac. Sci. Univ. Tokyo Sect. IA Math. \textbf{21} (1974), 1--17.

\bibitem{Stein} 
A. Stein, \emph{$1\frac{1}{2}$-generation of finite simple groups}, Beitr\"{a}ge Algebra Geom. \textbf{39} (1998), 349--358.

\bibitem{Steinberg}
R. Steinberg, \emph{Generators for simple groups}, Canadian J. Math. \textbf{14} (1962), 277--283.

\bibitem{Suzuki}
M. Suzuki, \emph{On a class of doubly transitive groups}, Annals of Math. \textbf{75} (1962), 105--145.

\bibitem{Ward}
H.N. Ward, \emph{On {R}ee's series of simple groups}, Trans. Amer. Math. Soc. \textbf{121} (1966), 62--89.
  
\bibitem{Weigel} 
T.S. Weigel, \emph{Generation of exceptional groups of Lie-type}, Geom. Dedicata \textbf{41} (1992), 63--87.

\bibitem{Wilson}
R.A. Wilson, \emph{The finite simple groups}, Graduate Texts in Mathematics, vol. 251. Springer-Verlag London, Ltd., London, 2009. xvi+298 pp.

\bibitem{Wilson2E62}
R.A. Wilson, \emph{Maximal subgroups of ${}^2E_6(2)$ and its automorphism groups}, preprint (arxiv:1801.08374).
 
\bibitem{Zsigmondy}
K. Zsigmondy, \emph{Zur Theorie der Potenzreste}, Monat. Math. Physik \textbf{3} (1892), 265--284.

\end{thebibliography}
\end{document}